\documentclass{imanum} 

\usepackage{float}
\usepackage[utf8]{inputenc} 
\usepackage[T1]{fontenc}
\usepackage{anyfontsize}

\usepackage[dvipsnames]{xcolor}
\usepackage{verbatim}
\usepackage{listings}
\usepackage{multirow}
\usepackage{tabularx}

\usepackage{amsmath}
\usepackage{amsthm}	
\usepackage{amsfonts}	
\usepackage{amssymb	
	,bbm
	,units 
}
\usepackage{enumerate}

\usepackage{algorithm}
\usepackage[noend]{algpseudocode}

\makeatletter
\renewcommand{\ALG@beginalgorithmic}{\footnotesize}
\makeatother

\usepackage{wrapfig}
\usepackage{graphicx}
\usepackage{subfig} 

\numberwithin{equation}{section} 
\numberwithin{figure}{section} 
\numberwithin{algorithm}{section} 
\numberwithin{table}{section} 
\renewcommand*{\thefootnote}{\fnsymbol{footnote}}

\renewcommand{\thefootnote}{\arabic{footnote}}

\theoremstyle{plain}
\newtheorem{assumption}[theorem]{Assumption}

\newcommand{\bE}{\mathbb{E}}

\newcommand{\bN}{\mathbb{N}}
\newcommand{\bP}{\mathbb{P}}

\newcommand{\bR}{\mathbb{R}}
\newcommand{\bS}{\mathbb{S}}

	\newcommand{\IR}{\bR}  

\newcommand{\cB}{\mathcal{B}}

\newcommand{\cF}{\mathcal{F}}

\newcommand{\cP}{\mathcal{P}}

\newcommand{\dd}{\mathrm{d}}

\definecolor{darkgreen}{rgb}{0,0.35,0}

\newcommand{\1}{\mathbbm{1}}

\jno{drnxxx} 
\received{09 June 2019} 
\accepted{13 December 2020 (This is final Author Version)}

\begin{document}

\title{Simulation of McKean Vlasov SDEs with super linear growth}
\shorttitle{Simulation of MV-SDEs with super linear growth} 

\author{%
	{\sc
		Gon\c calo dos Reis\thanks{Corresponding author. Email:  G.dosReis@ed.ac.uk}} \\[2pt]
	 University of Edinburgh (UK)  and 
	Centro de Matem\'atica e Aplica\c c$\tilde{\text{o}}$es (PT),\\[6pt]
	{\sc Stefan Engelhardt\thanks{Email: Engelhardt.stefan@uni-jena.de}}\\[2pt]
	Friedrich Schiller Universit\"at, Jena (DE),\\[2pt]
	{\sc and}\\[6pt]
	{\sc
	Greig Smith\thanks{Email:  G.Smith-13@sms.ed.ac.uk}}\\[2pt]
	University of Edinburgh, Maxwell Institute for Mathematical Sciences, UK  
}
\shortauthorlist{G. dos Reis \emph{et al.}} 

\maketitle

\renewcommand*{\thefootnote}{\arabic{footnote}}

\begin{abstract}
{We present two fully probabilistic Euler schemes, one explicit and one implicit, for the simulation of McKean-Vlasov Stochastic Differential Equations (MV-SDEs) with drifts of super-linear growth and random initial condition. 

We provide a pathwise propagation of chaos result and show strong convergence for both schemes on the consequent particle system. The explicit scheme attains the standard $1/2$ rate in stepsize. From a technical point of view, we successfully use stopping times to prove the convergence of the implicit method although we avoid them altogether for the explicit one. The combination of particle interactions and random initial condition makes the proofs technically more involved. 

Numerical tests recover the theoretical convergence rates and illustrate a computational complexity advantage of the explicit over the implicit scheme. Comparative analysis is carried out on a {stylized non Lipschitz MV-SDE and a mean-field model for FitzHugh-Nagumo neurons}. We provide numerical tests illustrating \emph{particle corruption} effect where one single particle diverging can ``corrupt'' the whole particle system. Moreover, the more particles in the system the more likely this divergence is to occur. }
%
{McKean-Vlasov Stochastic Differential Equation, Interacting Particle System, Monte Carlo Simulation, Taming, Implicit and Explicit Schemes, Stochastic Neuron Networks}
\end{abstract}

%

%

\section{Introduction}

The aim of this paper is to develop a numerical scheme for simulating McKean-Vlasov Stochastic Differential Equations (MV-SDEs) with drifts of super-linear growth and Lipschitz diffusion coefficients (with linear growth). MV-SDEs differ from standard SDEs by means of the presence of the law of the solution process in the coefficients:
\begin{equation*}
	\dd X_{t} = b(t,X_{t}, \mu_{t}^{X})\dd t + \sigma(t,X_{t}, \mu_{t}^{X})\dd W_{t},  \quad X_{0} \in L_{0}^{m}( \bR^{d}),
\end{equation*}
where $\mu_{t}^{X}$ denotes the law of the process $X$ at time $t$. Similar to standard SDEs, MV-SDEs have been shown to have a unique strong solution in the super-linear growth in spatial parameter setting, see \cite{dosReisSalkeldTugaut2017} (see also \cite{adams2020large}). Of course, many mean-field models exhibit non-global Lipschitz growth, for example mean-field models for neuronal activity (e.g.~stochastic mean-field FitzHugh-Nagumo models or the network of Hodgkin-Huxley neurons) \cite{BaladronFasoliFaugerasEtAl2012} {with correction note \cite{BossyEtAl2015}}, \cite{Bolley2011}, appearing in biology or physics \cite{dreyer2011phase}, \cite{DreyerFrizGajewskiEtAl2016}. We refer to {the reviews in \cite{BaladronFasoliFaugerasEtAl2012}, \cite{Mehri2020}} for further motivation of the problem. We precise that these works prove that Fokker-Planck-McKean-Vlasov equations, McKean-Vlasov equations and particle systems are well posed.

In general, closed form solutions for such equations are rare. Hence, to fully utilize MV-SDEs as a modelling tool, one needs a reliable way in which to simulate them. It is well known that for SDEs the explicit Euler scheme runs into difficulties in the super-linear growth setting, see \cite{HutzenthalerEtAl2011}, even though the SDE is known to have a unique strong solution. The original solution to this problem was to consider an implicit (or backward) Euler scheme developed in \cite{HighamEtAl2002}. Although implicit schemes allowed one to tackle more general SDEs they are slower especially in higher dimensions. The reason for this boils down to the fact that one is required to solve a fixed point equation at every time-step, which can be computationally expensive. To solve this problem an explicit scheme was then developed in \cite{HutzenthalerEtAl2012}, a so-called \emph{Tamed Euler} scheme. Since then several authors have built on this result and developed algorithms to deal with coefficients that grow super-linearly, see \cite{ChassagneuxEtAl2016}, \cite{Sabanis2013}, \cite{FangGiles2016} for example. The taming techniques we use fall within the universal taming methodologies developed in \cite{Lionnet2018}. Lastly, we refer to \cite{dosReisEtAl2018} for importance sampling Monte Carlo methods for MV-SDEs with super-linear drift.

An extra complication MV-SDEs offer over standard SDEs is the requirement to approximate the law $\mu$ at each time step. Although there are other techniques (see \cite{GobetPagliarani2018}), the most common is the so-called interacting particle system
\begin{align*}
	\dd {X}_{t}^{i,N} 
	= b\big(t,{X}_{t}^{i,N}, \mu^{X,N}_{t} \big) \dd t 
	+ \sigma\big(t,{X}_{t}^{i,N} , \mu^{X,N}_{t} \big) \dd W_{t}^{i},
\end{align*}
where $\mu^{X,N}_{t}(\dd x) := \frac{1}{N} \sum_{j=1}^N \delta_{X_{t}^{j,N}}(\dd x)$ and $\delta_{{X}_{t}^{j,N}}$ is the Dirac measure at point ${X}_{t}^{j,N}$, and $W^{i}, i=1,\dots,N$ are independent Brownian motions. Under Lipschitz type conditions this particle system is known to converge pathwise to the true solution of the MV-SDE (see \cite{Sznitman1991}, \cite{Meleard1996}). However, this convergence (with corresponding rate) in a super-linear growth setting has thus far not been considered in full generality. 

Closer to our work, we highlight: \cite{BudhirajaFan2017} develop an explicit Euler scheme to deal with a specific MV-SDE type equation from a chemotaxis model; convergence is given but under Lipschitz conditions and constant diffusion coefficient. \cite{Malrieu2003} studies an implicit Euler scheme in order to approximate a specific equation and requires a constant diffusion coefficient, symmetry and uniform convexity of the interaction potential. Lastly, in \cite[Section 3.5]{gomes2019mean} the authors are only able to justify their simulation for the Lipschitz case and the results we propose would allow for more general potentials.

\emph{Our contribution.} Firstly, we show that the above particle scheme converges (propagation of chaos) in the super-linear growth case without coercivity/dissipativity. This result is crucial in showing convergence of the numerical scheme to the particle system rather than to the original MV-SDE, with corresponding rate. 

The second contribution is the development and strong convergence of the explicit scheme to the MV-SDE, inspired by the explicit scheme originally developed in \cite{HutzenthalerEtAl2012}, \cite{Sabanis2013} and the universal taming methodologies developed in \cite{Lionnet2018}. A crucial technical point of difference to \cite{Sabanis2013} is while stopping-time arguments in way of localization techniques were put to vigorous use there, the MV-SDE setting and its associated particle systems are (a priori) not amenable to such techniques: we follow a different path in our proofs. We also obtain the classical $1/2$ rate of convergence in the stepsize. Combining this with the propagation of chaos result gives an overall convergence rate for the explicit scheme.

The final contribution is to show strong convergence of an implicit scheme. This turns out to be a challenging problem since results involving implicit schemes rely on stopping time arguments. This causes several issues when generalizing results to the MV-SDE setting and we have to make stronger assumptions on the coefficients in this setting in order for the arguments to continue to hold. On the other hand, we allow for random initial conditions and time dependent coefficients that, to the best of our knowledge, have not been fully treated in the standard SDE setting. We discuss these issues in Remark \ref{rem:No stopping times}. We only focus on strong convergence of this scheme and not the rate, mainly because the explicit scheme is shown to work under more general assumptions, scales better (as our numerical testing shows) and such proof would lead to lengthy statements without substantially enhancing the scope of applications. The question is left for future research with a tentative methodology discussed in Remark \ref{rem:ConvRateImplicit} below.

From a technical point of view, we highlight the successful use of stopping time arguments in combination with McKean-Vlasov equations and associated particle systems to show the convergence of the implicit scheme with diffusion coefficients that are measure dependent. \\

The paper is structured in the following way: In Section \ref{Sec:Preliminaries} we introduce the notation and our tamed particle scheme. In Section \ref{Sec:Main Result}, we state our main result, namely, propagation of chaos and convergence results for the two schemes.
Following that, in Section \ref{Sec:Examples} we provide several numerical examples and highlight the  \emph{particle corruption} phenomena. This analysis implies one cannot hope to build a reliable scheme based on a standard Euler scheme. We further show the increased computational complexity associated with a MV-SDE makes the implicit scheme a less viable option than the explicit (tamed) scheme. Finally, the proofs are given in Section \ref{Sec:Proofs}.

%
%
%
%

\section{Preliminaries}
\label{Sec:Preliminaries}

Throughout the paper we work on a filtered probability space $(\Omega,
\cF, (\cF_{t})_{t \ge 0},  \bP)$ satisfying the usual conditions,
where $\cF_{t}$ is the augmented filtration of a standard multidimensional Brownian
motion $W$. 
We work with $\bR^d$, the $d$-dimensional Euclidean space of real numbers, and for $a=(a_{1},\cdots,a_{d})\in\bR^d$ and $b=(b_{1},\cdots,b_{d})\in\bR^d$ we denote by $|a|^2=\sum_{i=1}^d a_{i}^{2}$ the usual Euclidean distance on $\bR^d$ and by $\langle a,b\rangle=\sum_{i=1}^d a_{i} b_{i}$ the usual scalar product. For matrices $V \in \bR^{k \times \ell}$ we define $|V| = \sup_{u \in \bR^{\ell}, ~ |u| \le 1} |Vu|$. 
  
We consider some finite terminal time $T < \infty$ and use the following notation for spaces, which are standard in the McKean-Vlasov literature (see \cite{Carmona2016}): We define $\bS^{p}$ for $p \ge 1$, as the space of $\bR^{d}$-valued, $\cF_{\cdot}$-adapted processes $Z$, that satisfy $\bE [ \sup_{0 \le t \le T} |Z(t)|^{p}]^{1/p} < \infty$. Similarly, $L_{t}^{p}(\bR^d)$, defines the space of $\bR^{d}$-valued, $\cF_{t}$-measurable random variables $X$, that satisfy $\bE [|X|^{p}]^{1/p} < \infty$.

Given the measurable space $(\bR^{d}, \cB (\bR^{d}))$, we denote by
$\cP(\bR^{d})$ the set of probability measures on this space, and for $p\ge 1$ write $\mu \in \cP_{p}(\bR^{d})$ if $\mu \in \cP(\bR^{d})$ and for
some $x \in \bR^{d}$, $\int_{\bR^{d}} |x-y|^{p} \mu(\dd y) <
\infty$. We then have the following metric on the space $\cP_p(\mathbb
R^d)$ (Wasserstein metric) for $\mu, ~ \nu \in \cP_{p}(\bR^{d})$ (see \cite{Villani2008}, \cite{dosReisSalkeldTugaut2017} among others),
\begin{align*}
W^{(p)}(\mu, \nu) 
=
\inf_{\pi} \Big\{
\Big( \int_{\bR^{d} \times \bR^{d}} |x&-y|^p \pi( \dd x, \dd y) \Big)^{\frac1p}
:
~ \pi \in \cP(\bR^{d} \times \bR^{d})
~ 
\text{with marginals $\mu$ and $\nu$}
\Big\} \, .
\end{align*}
The most common choice in the McKean-Vlasov setting is, $p=2$, and is what we shall use throughout most of this paper. As $W^{(2)}$ is a metric (see \cite{Villani2008} Chapter 6), we have for $\mu_1,\mu_2,\mu_3 \in \cP_2(\mathbb R^d)$
\begin{align*}
	W^{(2)}(\mu_1, \mu_3)
	\le
	W^{(2)}(\mu_1, \mu_2)
	+
	W^{(2)}(\mu_2, \mu_3) \hspace*{0.2mm} .
\end{align*}
As in \cite{Carmona2016}, we introduce the empirical measure  constructed from i.i.d.~samples of some process $X$ by $\mu_{s}^{X,N} := \frac{1}{N} \sum_{j=1}^{N} \delta_{X_{s}^{j}}$. 
Another standard result for the Wasserstein metric for two such empirical measures $\mu_{s}^{X,N}$, $\mu^{Y,N}_s$ is that 
\begin{align*}
	W^{(2)}(\mu_{s}^{X,N}, \mu^{Y,N}_s) 
	\le
	\Big( \frac{1}{N} \sum_{j=1}^{N} |X_{s}^{j} - Y_{s}^{j}|^{2} \Big)^{1/2} \hspace*{0.2mm} .
\end{align*}

\subsection{McKean-Vlasov stochastic differential equations}

Let $W$ be an $l$-dimensional Brownian motion and take the progressively measurable maps $b:[0,T] \times \bR^d \times\cP_2(\bR^d) \to \bR^d$ and $\sigma:[0,T] \times \bR^d \times \cP_2(\bR^d) \to \bR^{d\times l}$.
MV-SDEs are typically written in the form,
\begin{equation}
\label{Eq:General MVSDE}
\dd X_{t} = b(t,X_{t}, \mu_{t}^{X})\dd t + \sigma(t,X_{t}, \mu_{t}^{X})\dd W_{t},  \quad X_{0} \in L_{0}^{p}( \bR^{d}),
\end{equation}
where $\mu_{t}^{X}$ denotes the law of the process $X$ at time $t$, i.e.~$\mu_{t}^{X}=\bP\circ X_t^{-1}$. We make the following assumption on the coefficients throughout.

\begin{assumption}
\label{Ass:Monotone Assumption}
	Assume that $\sigma$ is Lipschitz in the sense that there
        exists $ L_\sigma>0$ such that for all $t \in[0,T]$ and all $x, x'\in \bR^d$ and $\forall \mu, \mu'\in \cP_2(\bR^d)$ we have that
	$$
	 |\sigma(t, x, \mu)-\sigma(t, x', \mu')|\leq L_\sigma (|x-x'| + W^{(2)}(\mu, \mu') ),
	$$and let $b$ satisfy
	\begin{enumerate}
		\item One-sided Lipschitz in $x$ and Lipschitz
                  in law: there exist $ L_{b}, ~ L >0$ such that for all $t
                  \in[0,T]$, all $ x, x'\in \bR^d$ and all $\mu, \mu'\in \cP_2(\bR^d)$ we have that
\begin{align*}
	\langle x-x', b(t, x, \mu)-b(t, x',\mu) \rangle & \leq L_{b}|x-x'|^{2}
	\\
\textrm{and} \quad	|b(t, x, \mu)-b(t, x,\mu')| &\le  LW^{(2)}(\mu, \mu') .
\end{align*}

		\item Locally Lipschitz with polynomial growth in $x$:
                  there exist $L>0$ and $q \in \bN$ with $q>1$ such that for
                  all  $t \in [0,T]$, $\forall \mu \in
                  \cP_{2}(\bR^{d})$ and all $x, ~ x' \in \bR^{d}$ 
		 $$
		 |b(t, x, \mu)-b(t, x',\mu)| \leq L(1+ |x|^{q} + |x'|^{q}) |x-x'| .
		 $$
	\end{enumerate}
\end{assumption} 

\begin{assumption}
	\label{Ass:Holder in Time}
	Assume that $b$ and $\sigma$ are $1/2$-H\"{o}lder continuous in time, uniformly in $x$ and $\mu$.
\end{assumption}

These assumptions allow for a far larger class of models than the standard globally Lipschitz drift assumption. Under the one-sided Lipschitz drift assumption, \cite[Theorem 3.3]{dosReisSalkeldTugaut2017} provides a result for existence and uniqueness. We use a particularised version of this result as we also require Assumption \ref{Ass:Holder in Time} to control the growth in time of the coefficient.
Any globally space-measure Lipschitz drift such as the Kuramoto model, see e.g.~\cite{dosReisEtAl2018}, can of course be simulated using our techniques. We present examples with one-sided Lipschitz drifts that satisfy these assumptions in our numerics Section \ref{Sec:Examples} including the adjusted Ginzburg Landau equation. There are also many other classes of models where one-sided Lipschitz drifts appear such as Kinetic models e.g.~in \cite{gomes2019mean} and Self-stabilizing diffusions \cite{Bolley2011}, \cite{Malrieu2003}; see also \cite{MalrieuTalay2006}.

\begin{theorem}[\cite{dosReisSalkeldTugaut2017}]
\label{Thm:MV Monotone Existence}
	Suppose that $b$ and $\sigma$ satisfy Assumption \ref{Ass:Monotone Assumption} and \ref{Ass:Holder in Time}.
	Further, assume for some $m \ge 2$, $X_{0} \in L_{0}^{m}(\bR^{d})$.
	Then there exists a unique solution $X\in \bS^{m}([0,T])$ to the MV-SDE \eqref{Eq:General MVSDE}. 
	For some positive constant $C$  we have
	\begin{align*}
	\mathbb E \big[ \sup_{t\in[0,T]} |X_{t}|^{m} \big] 
	\leq C \left(\bE[|X_0|^m]
	+ 1
	\right) e^{C T}.
	\end{align*} 
\end{theorem}
If the law $\mu^X$ is known beforehand, then the  MV-SDE reduces to a classical SDE with added time-dependency. Typically this is not the case and usually the MV-SDE is approximated by a particle system. 

\textbf{The interacting particle system approximation.} We approximate \eqref{Eq:General MVSDE} (driven by the Brownian motion $W$), using an $N$-dimensional system of interacting particles. Let $i=1, \dots, N$ and consider $N$ particles $X^{i,N}$ satisfying the SDE with i.i.d.~${X}_{0}^{i,N}=X_{0}^{i}$ (the initial condition is random, but independent of other particles)  
\begin{align}
\label{Eq:MV-SDE Propagation}
\dd {X}_{t}^{i,N} 
= b\Big(t,{X}_{t}^{i,N}, \mu^{X,N}_{t} \Big) \dd t 
+ \sigma\Big(t,{X}_{t}^{i,N} , \mu^{X,N}_{t} \Big) \dd W_{t}^{i}
, 
\end{align}
where $\mu^{X,N}_{t}(\dd x) := \frac{1}{N} \sum_{j=1}^N \delta_{X_{t}^{j,N}}(\dd x)$ and $\delta_{{X}_{t}^{j,N}}$ is the Dirac measure at point ${X}_{t}^{j,N}$, and the independent Brownian motions $W^{i}, i=1,\dots,N$ (also independent of the BM $W$ appearing in \eqref{Eq:General MVSDE}; with a slight abuse of notation to  avoid re-defining the probability space's Filtration).

\textbf{Propagation of chaos.}
In order to show that the particle approximation is of use, one shows a pathwise propagation of chaos result. Although different types exist we are interested in the strong error. Hence a pathwise convergence result is needed and we consider the system of non interacting particles  
\begin{align}
	\label{Eq:Non interacting particles}
	\dd X_{t}^{i} = b(t, X_{t}^{i}, \mu^{X^{i}}_{t}) \dd t + \sigma(t,X_{t}^{i}, \mu^{X^{i}}_{t}) \dd W_{t}^{i}, \quad X_{0}^{i}=X_{0}^{i} \, ,\quad t\in [0,T] \, ,
\end{align}
which are of course just MV-SDEs and since the $X^{i}$s are
independent, $\mu^{X^{i}}_{t}=\mu^{X}_{t}$ for all $
i$. Under global Lipschitz conditions, one can then prove the
following convergence result (see \cite{Sznitman1991}, \cite{Meleard1996})
\begin{align*}
\lim_{N \rightarrow \infty} \sup_{1 \le i \le N}
\bE \big[ 
\sup_{0 \le t \le T} |X_{t}^{i,N} - X_{t}^{i}|^{2}\,
\big] = 0 \, .
\end{align*}
Several propagation of chaos results have been shown over the years under varying conditions, see 
\cite[Theorem 1.10]{Carmona2016} and \cite{Lacker2018} among others.
All SDEs appearing below have initial condition $X_{0}^{i}$ and we work on the interval $[0,T]$.

\textbf{Standard Euler scheme particle system.}
In general one cannot simulate \eqref{Eq:MV-SDE Propagation} directly and therefore turns to a numerical scheme such as Euler. We partition the time interval $[0,T]$ into $M$ steps of size $h:=T/M$, we then define $t_{k}:=kh$ and recursively define the particle system for $k \in \{0, \dots , M-1\}$ as,
\begin{align*}
& \bar{X}_{t_{k+1}}^{i,N,M}
=
 \bar{X}_{t_{k}}^{i,N,M}
 + 
 b\Big(t_{k} , \bar{X}_{t_{k}}^{i,N,M}, \bar{\mu}^{X,N}_{t_{k}} \Big)
 h  
+ \sigma\Big(t_{k}, \bar{X}_{t_{k}}^{i,N,M} , \bar{\mu}^{X,N}_{t_{k}} \Big) \Delta W_{t_{k}}^{i} , 
\end{align*}
where $\bar{\mu}^{X,N}_{t_{k}}(\dd x) := \frac{1}{N} \sum_{j=1}^N \delta_{\bar{X}_{t_{k}}^{j,N,M}}(\dd x)$, $\Delta W_{t_{k}}^{i}:=W_{t_{k+1}}^{i}-W_{t_{k}}^{i}$ and $\bar{X}_{0}^{i,N,M} := X_{0}^{i}$. Under Lipschitz regularity it is well known that this scheme converges, see \cite{BossyTalay1997} or \cite{KohatsuOgawa1997} (here a weak rate of convergence is shown under an additional regularity assumption).

\textbf{Euler particle system for the super-linear case: explicit and implicit.}
However, as discussed in works such as \cite{HutzenthalerEtAl2011}, \cite{HutzenthalerEtAl2012}, \cite{Sabanis2013} one does not have convergence of the Euler scheme when we move away from the global Lipschitz setting. The goal of this paper is to therefore construct a suitable numerical scheme which converges. Inspired by the above works we consider a so-called \emph{tamed} Euler scheme.
With the notation above we consider the following scheme
\begin{align}
	\label{Eq:Tamed MVSDE}
 \bar{X}_{t_{k+1}}^{i,N,M}
	=
	\bar{X}_{t_{k}}^{i,N,M}
	&+ 
	\dfrac{b\Big(t_{k} , \bar{X}_{t_{k}}^{i,N,M}, \bar{\mu}^{X,N}_{t_{k}} \Big) }
	{1+ M^{-\alpha} \Big\vert b\Big(t_{k} , \bar{X}_{t_{k}}^{i,N,M}, \bar{\mu}^{X,N}_{t_{k}} \Big) \Big\vert } h
	+ 
	\sigma\Big(t_{k} , \bar{X}_{t_{k}}^{i,N,M}, \bar{\mu}^{X,N}_{t_{k}} \Big)
	\Delta W_{t_{k}}^{i} ,
\end{align}
where $\bar{\mu}^{X,N}_{t_{k}}(\dd x) = \frac{1}{N} \sum_{j=1}^N \delta_{\bar{X}_{t_{k}}^{j,N,M}}(\dd x)$ and $\alpha \in (0, 1/2]$ with $\bar{X}_{0}^{i,N,M}= X_{0}^{i}$.

Of course, explicit schemes are not the only method one can deploy to solve this problem, we also consider the following implicit scheme  
\begin{align}
\label{Eq:implicitScheme}
& \tilde{X}_{t_{k+1}}^{i,N,M} 
= \tilde{X}_{t_{k}}^{i,N,M} +  b\Big(t_{k},\tilde{X}_{t_{k+1}}^{i,N,M}, \tilde{\mu}^{X,N,M}_{t_{k}} \Big) h 
+ \sigma\Big(t_k,\tilde{X}_{t_{k}}^{i,N,M} , \tilde{\mu}^{X,N,M}_{t_{k}} \Big) \Delta W_{t_{k}}^{i},
\end{align}
where $\tilde{\mu}^{X,N,M}_{t_{k}}(\dd x) := \frac{1}{N} \sum_{j=1}^N \delta_{\tilde{X}_{t_{k}}^{j,N,M}}(\dd x)$ and $\tilde{X}^{i,N,M}_0 = X^i_0$.

\section{Main Results}
\label{Sec:Main Result}
We state our main results and assumption here, the proofs are postponed to Section \ref{Sec:Proofs}. Recall that we want to associate a particle system to the MV-SDE and show its convergence, so-called \emph{propagation of chaos}. We have the following result that holds under weaker assumptions than those in Theorem \ref{Thm:Main Theorem}.
\begin{proposition}[Propagation of chaos]
	\label{Prop:Propagation of Chaos}
	Let the assumption in Theorem \ref{Thm:MV Monotone Existence} hold for $m>4$. Let $X^{i}$ be the solution to \eqref{Eq:Non interacting particles}, and $X^{i,N}$ be the solution to \eqref{Eq:MV-SDE Propagation}.
	
	Then we have the following convergence result.
	\begin{align*}
	\sup_{1 \le i \le N} 
	\bE[\sup_{0 \le t \le T} |X_{t}^{i} - X_{t}^{i,N}|^{2}] 
	\le C
	\begin{cases}
	N^{-1/2} & \text{if } d<4,
	\\
	N^{-1/2} \log(N) \quad & \text{if } d=4,
	\\
	N^{-2/d} & \text{if } d>4.
	\end{cases} 
	\end{align*}
\end{proposition}
This result shows the particle scheme will converge to the MV-SDE with a given rate. Therefore, to show convergence between our numerical scheme and the MV-SDE, we only need to show that the numerical version of the particle scheme converges to the ``true'' particle scheme. {We point out that under the assumption of a constant diffusion matrix $\sigma$ it is possible to improve the PoC convergence rate above via the results in \cite{Delarue2019}.}


\subsection*{Explicit scheme}
We first introduce the continuous time version of the explicit scheme \eqref{Eq:Tamed MVSDE}. Denote by \linebreak $\kappa(t):= \sup \{ s \in \{0, h, 2h, \ldots, Mh\} : s \leq t \}$ for all $t \in [0,T]$ {and $h=T/M$}, $b_M(t,x,\nu):= \frac{b(t,x,\nu)}{1 + M^{-\alpha} \vert b(t,x,\nu) \vert}$ with $\alpha \in (0,1/2]$ for all $t\in [0,T]$, $x \in \IR^d$, $\nu \in \cP_2 (\IR^d)$ 
\begin{align}
	 X^{i,N,M}_t & = X_{0}^{i} + \int_0^t b_M \left( \kappa(s), X^{i,N,M}_{\kappa(s)}, \mu^{X,N,M}_{\kappa(s)} \right) \dd s 
	\notag
	\\
	&+ \int_0^t \sigma \left( \kappa(s), X^{i,N,M}_{\kappa(s)}, \mu^{X,N,M}_{\kappa(s)} \right) \dd W^i_s,\quad  
	\mu^{X,N,M}_{t} (\dd x) = \frac{1}{N} \sum_{j=1}^N \delta_{X^{j,N,M}_t} (\dd x)
	\label{Eq:Time Continuous Tamed}.
\end{align} 
Note that $\vert b_M(t,x,\nu) \vert \leq \min \left( M^\alpha, \vert b(t,x,\nu) \vert \right)$ and that $\bar{X}^{i,N,M}_{t_{k}} = X^{i,N,M}_{t_{k}}$ for all $k \in \{ 0, 1, \ldots ,M \}$ and hence $X^{i,N,M}$ is a continuous version of $\bar{X}^{i,N,M}$ from \eqref{Eq:Tamed MVSDE}.
We then obtain the following convergence result.

\begin{proposition}
	\label{Prop:Tamed Scheme Convergence}
	Let Assumption \ref{Ass:Monotone Assumption} and \ref{Ass:Holder in Time} hold, further let $X_{0} \in L^{m}(\bR^{d})$ for $m \ge 4(1+q)$ (note $q>1$) and set $\alpha =1/2$.
	Let $X^{i,N}$ be the solution to \eqref{Eq:MV-SDE Propagation}, and $X^{i,N,M}$ be that for \eqref{Eq:Time Continuous Tamed}.
	Then it holds that 
	\begin{align*}
	\sup_{1 \le i \le N}\bE[\sup_{0 \le t \le T} |X_{t}^{i,N}- X_{t}^{i,N,M}|^{2}] \le C h.
	\end{align*}
\end{proposition}

This then leads to our main explicit scheme convergence result.
\begin{theorem}[Strong convergence of explicit scheme]
	\label{Thm:Main Theorem}
	Let Assumption \ref{Ass:Monotone Assumption} and \ref{Ass:Holder in Time} hold, further let \linebreak $X_{0} \in L^{m}(\bR^{d})$ for $m \ge 4(1+q)$ (note $q>1$) and set $\alpha =1/2$. 
	 Let $X^{i}$ be the solution to \eqref{Eq:Non interacting particles}, and $X^{i,N,M}$ be that for \eqref{Eq:Time Continuous Tamed}.
	
	Then we obtain the following convergence result
	\begin{align*}
	\sup_{1 \le i \le N}  \bE \big[ \sup_{0 \le t \le T} |X_{t}^{i} - X_{t}^{i,N,M}|^{2}
	\big] 
	\le
	 C
	 \begin{cases}
	 N^{-1/2} +h & \text{if } d<4,
	 \\
	 N^{-1/2} \log(N) + h \quad & \text{if } d=4,
	 \\
	 N^{-2/d} + h & \text{if } d>4.
	 \end{cases} 
	\end{align*}
\end{theorem}
\begin{proof}[Proof of Theorem \ref{Thm:Main Theorem}]
	Theorem \ref{Thm:Main Theorem} is a consequence of Propositions \ref{Prop:Propagation of Chaos} and \ref{Prop:Tamed Scheme Convergence}.
\end{proof}


\begin{remark}[Issues using stopping times]
	\label{rem:No stopping times}
	The technique of using the stopping time \linebreak $\tau^i_R := \inf \{ t \geq 0 : \vert X^{i,N,M}_t \vert \geq R \}$ to control the particles is suboptimal here and several problems appear by introducing them (this is in stark opposition to \cite{Sabanis2013}). Namely, one can only consider stopping times that stop one particle since otherwise the convergence speed would decrease with a higher number of particles. However, applying a stopping time to a single particle does not allow us to fully bound the coefficients and moreover destroys the result of all particles being identically distributed. 
	
The stopping time arguments used for the implicit scheme below require stronger assumptions in order to make the theory hold.
\end{remark}

\begin{remark}[Controlling $W^{2}\big(\mu_{t}^{X}, \mu_{t}^{X, M, N}\big)$: Moments, functional CLTs, concentration inequalities and weak error]
Error estimates in $L^p$ norm do not provide optimal information on stochastic numerical methods since these methods are aimed to approximate probability distributions rather than particular paths. For McKean-Vlasov equations and their applications to
the numerical analysis of PDES and to stochastic modelling in target fields, the challenge consists in approximating the marginal distributions of the mean-field limit. 

The question is what can one say regarding $W^{2}\big(\mu_{t}^{X}, \mu_{t}^{X, M, N}\big)$?

Let us start by observing that $\mu_{t}^{X, M, N}$ is a random measure, hence four answers to the question are in order. (I) Firstly, the strong order convergence of Theorem \ref{Thm:Main Theorem} provides a direct control for $\bE[(\mu_{t}^{X, M, N})^2]$, but these are just less informative $L^p$ moments. (II) Asymptotic functional CLT type results are a second line of approach. We point to the very recent \cite{jourdain2020central}. Although not sufficiently general to capture the setting we present, since they require Lipschitz \& differentiability conditions (in space and measure), they prove a functional CLT that works crucially when the samples are dependent.
Their result is brilliant and asymptotic allowing one to design asymptotic confidence
intervals to control $W^{2}\big(\mu_{t}^{X}, \mu_{t}^{X, M, N}\big)$. The (III) third answer are the non-asymptotic concentration inequalities. \cite{Frikha2012concentration} address concentration inequalities for the general Euler scheme under uniform Lipschitz conditions and with space-dependent diffusion coefficient. The literature is incredibly rich for the constant diffusion-coefficient case, be it for classical SDEs or in the MV-SDE case, always with a convexity at infinity condition (or then strict convex potentials). On MV-SDEs we refer to \cite{MalrieuTalay2006} where either boundedness of the drift is used or all involved drift maps are strictly convex such that the Bakry-Emery machinery transfers. Concentration inequalities able to cover the case we present in this manuscript are still an open question -- mainly, current taming methodologies appearing in the literature although able to control the superlinear growth in $b$ entail the high-cost of losing the strict dissipativity of the one-sided Lipschitz condition (i.e.~when $L_b<0$ in Assumption \ref{Ass:Monotone Assumption}), this is evident in \cite{ChassagneuxEtAl2016}. {The last answer (IV) are weak convergence rate results which is currently an open question for this type of schemes in general. }
\end{remark}


\subsection*{The implicit scheme}

As alternative to the explicit scheme we now discuss the implicit or backward Euler scheme. That being said, the implicit scheme has some well documented disadvantages, namely it is expensive compared to its explicit counterpart, we discuss this issue further in Section \ref{Sec:Examples}. One can consult, \cite{MaoSzpruch2013} for example on the implicit scheme (and extensions) for standard SDEs.

Standard implicit scheme convergence results rely on the so called monotone growth condition, we therefore proceed with the following assumption.
\begin{assumption}
	\label{Ass:Extra Implicit Bounds}
	\begin{enumerate}[(H1).]
		\item 	\label{Ass:Bounded Measure}
	There exists a constant $C$ such that, for all $\mu \in \cP_{2}(\bR^{d})$,
	\begin{align*}
	|b(0,0,\mu)|
	+
	|\sigma(0,0,\mu)|
	\le C \, .
	\end{align*}
	
	\item \label{Ass:Sigma Indepedent of Measure}
	The drift and diffusion coefficient satisfy the stronger bound in measure condition, for all \linebreak$t\in[0,T]$, all $ x\in \bR^d$ and all $\mu, \mu'\in \cP_2(\bR^d)$
	\begin{align*}
	|b(t, x, \mu)-b(t, x,\mu')| + |\sigma(t, x, \mu)-\sigma(t, x,\mu')|  &\le  LW^{(1)}(\mu, \mu') \, ,
	\end{align*}
	where $W^{(1)}(\cdot, \cdot)$ is the Wasserstein-1 distance.
	\end{enumerate}
\end{assumption}
Although the main convergence theorem requires both H\ref{Ass:Bounded Measure} and H\ref{Ass:Sigma Indepedent of Measure}, we only use H\ref{Ass:Sigma Indepedent of Measure} at the end of the proof of convergence. We present our auxiliary results requiring only H\ref{Ass:Bounded Measure} as we believe them to be of general independent interest.

We now state the strong convergence of the implicit scheme \eqref{Eq:implicitScheme} to \eqref{Eq:MV-SDE Propagation}.

\begin{proposition}
	\label{Prop:Strong Implicit Scheme Convergence}
	Let Assumption \ref{Ass:Monotone Assumption}, \ref{Ass:Holder in Time} and \ref{Ass:Extra Implicit Bounds} hold. Fix a timestep $h^{*}< 1/\max(L_{b}, 2 \beta)$ and assume that $X_{0} \in L^{4(q+1)}( \bR^{d})$. Let $X^{i,N}$ be the solution to \eqref{Eq:MV-SDE Propagation}, and $\tilde{X}^{i,N,M}$ be that for \eqref{Eq:implicitScheme}.
	Then, for any $h$ and $M$ with $T = hM$ and $s \in [1,2)$
	\begin{align*}
	\sup_{1 \le i \le N}
	\lim_{h \rightarrow 0}
	\bE[ |X_{T}^{i,N} - \tilde{X}_{T}^{i,N,M}|^{s} ]
	=0 \, .
	\end{align*}
\end{proposition}

\begin{theorem}[Strong convergence of implicit scheme]
	\label{Thm:Strong Convergence of Implicit}
	Let the assumption in Proposition \ref{Prop:Strong Implicit Scheme Convergence} hold and let $X^{i}$ be the solution to \eqref{Eq:Non interacting particles}, and $\tilde{X}^{i,N,M}$ be that for \eqref{Eq:implicitScheme}. Then, for any $h$ and $M$ with $T = hM$ and $s \in [1,2)$ one has
	\begin{align*}
	\lim_{N \rightarrow \infty}
	\sup_{1 \le i \le N}
	\lim_{h \rightarrow 0}
	\bE[ |X_{T}^{i} - \tilde{X}_{T}^{i,N,M}|^{s} ]
	=0 \, .
	\end{align*}
\end{theorem}

\begin{proof}
	The proof of this result follows by combing Proposition \ref{Prop:Propagation of Chaos} and \ref{Prop:Strong Implicit Scheme Convergence} and noting that the assertion in Proposition \ref{Prop:Strong Implicit Scheme Convergence} is independent of $N$.
\end{proof}

\begin{remark}[On the convergence rate of the implicit scheme]
\label{rem:ConvRateImplicit}
Theorem \ref{Thm:Strong Convergence of Implicit} shows the convergence of the implicit scheme but without establishing a rate. Methodologically speaking, the approach proposed in \cite{HighamEtAl2002} seems applicable here, where the convergence rate of the implicit scheme would be shown by defining an intermediate process and considering the convergence of the implicit scheme to the intermediate process and then that of the intermediate process to the original equation, see \cite{HighamEtAl2002}. We suspect that such proof is not straightforward with several extra constraints appearing due to the presence of the law. As it stands, the convergence of our implicit scheme requires stronger assumptions (see Assumption \ref{Ass:Extra Implicit Bounds}) than the explicit one so we leave establishing the rate for future. Our numerical experiments hint that the convergence rate should be the same as the explicit, which is consistent with the case of standard SDEs.	
\end{remark}

\section{Numerical Testing and Examples}
\label{Sec:Examples}
We illustrate our results with numerical examples. We highlight the issues of using the standard Euler scheme in this setting and also compare the computational time and complexity of the explicit and implicit scheme. We juxtapose our findings to those in \cite{BaladronFasoliFaugerasEtAl2012} {(see correction note \cite{BossyEtAl2015})}.

\subsection{Particle corruption}
It is well known that the Euler scheme fails (diverges) when one moves outside the realm of linear growing coefficients, see \cite{HutzenthalerEtAl2011}. We claim that this divergence is worse in the setting of MV-SDEs and associated particle system due to an effect we refer to as \emph{particle corruption}. 

The basic idea is that one particle becomes influential on all other particles, thus we are no longer in the setting of ``weakly interacting''. This is of course not a problem for standard SDE simulation. We show two aspects of particle corruption in a simple example. Firstly, one particle can cause the whole system to crash. Secondly and perhaps more profoundly, the more particles one has the more likely this occurs. This is of course a devastating issue when simulating a MV-SDE since accurately approximating the measure depends on having a large number of interacting particles.

To show this example we take a classical non-globally Lipschitz SDE, the stochastic Ginzburg Landau equation (see \cite{Tien2013}) and add a simple mean field term to it,
\begin{align*}
\dd X_{t} = \Big( \frac{\sigma^{2}}{2} X_{t} - X_{t}^{3} + c \bE[X_{t}] \Big) \dd t + \sigma X_{t} \dd W_{t}, \quad X_{0}=x.
\end{align*}
This MV-SDE clearly satisfies the assumption to have a unique strong solution in $\bS^{p}$ for all $p>1$, hence in theory one could calculate $\varphi(t):=\bE[X_{t}]$ and have a standard SDE with one-sided Lipschitz drift. The analysis carried out in \cite{HutzenthalerEtAl2011} then implies that the Euler scheme diverges here.

\textbf{Showing particle corruption exists.}
 For our example we simulate $N=5000$ particles with a time step $h=0.05$, $T=2$ and $X_{0}=1$, we also take $\sigma=3/2$ and $c=1/2$. We rerun this example until we observed a blow up and plotted the particle paths in Figure \ref{Fig:ParticleCorruptionPart1}.
 
 \begin{figure}[ht] 
 	\centering
 	\includegraphics[width=12cm]{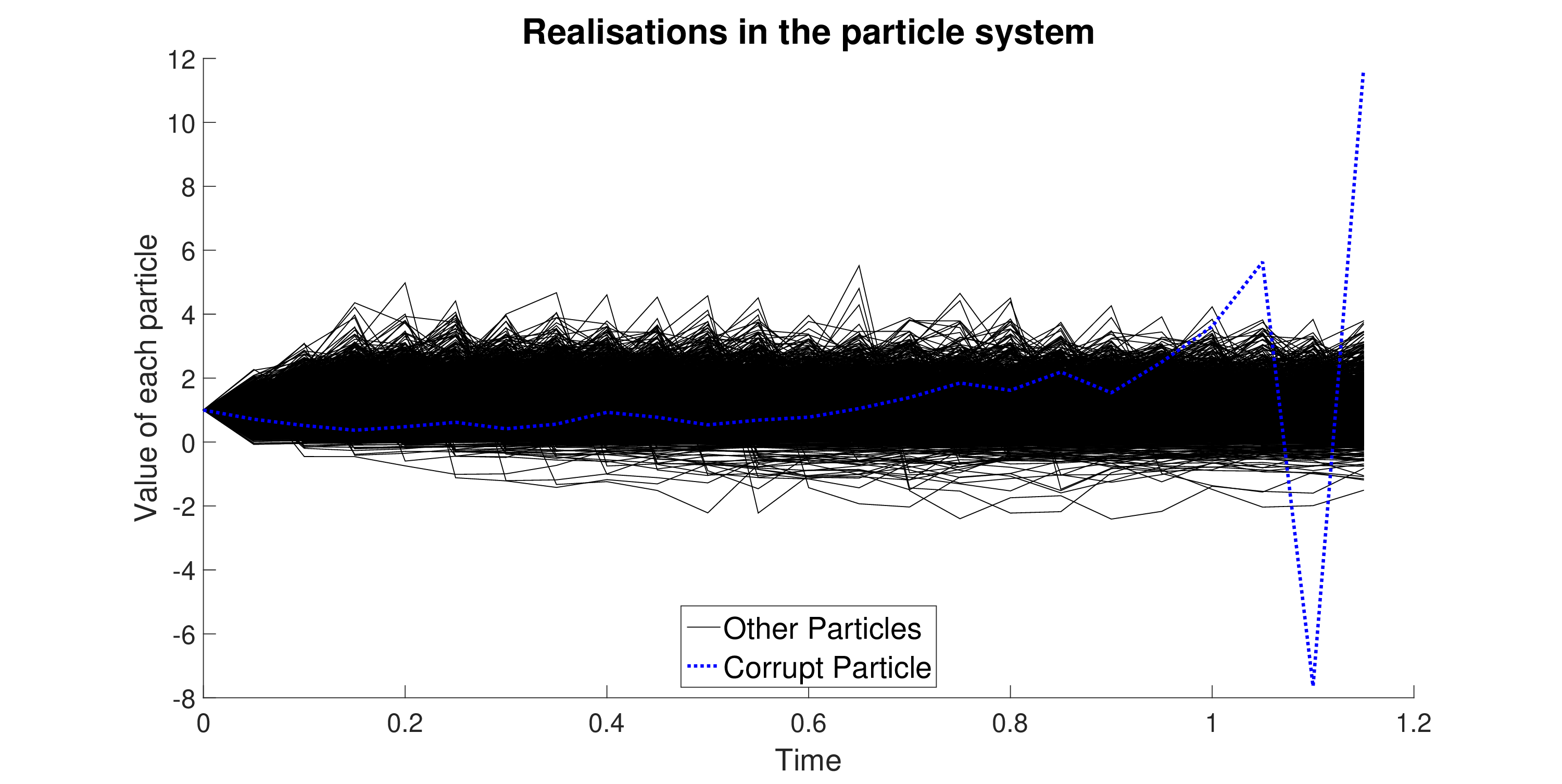}
 	\caption{Showing the realizations of the particles in the system. We note that the particle given by the dashed line is starting to oscillate and is taking larger values than its surrounding particles.}
 	\label{Fig:ParticleCorruptionPart1}
 \end{figure}

Figure \ref{Fig:ParticleCorruptionPart1} shows the first part of the divergence, namely all particles are reasonably well behaved until one starts to oscillate rapidly. We have stopped plotting before the time boundary since this particle diverges shortly after this. We refer to this particle as the \emph{corrupt particle} and it is fairly straightforward to see it will diverge. However, due to the interaction this single particle influences all the remaining particles and the whole system diverges shortly after.

\begin{remark}
	[Why is particle corruption so pronounced?]
	The reason this effect is so dramatic is a simple consequence of the mean-field interaction. Typically, one observes divergence of the Euler scheme via a handful of Monte Carlo simulations that return extremely large (or infinite) values. When one then looks to calculate the expected value of the SDEs at the terminal time for example, these few events completely dominate the other results. This is summed up in a statement of \cite{HutzenthalerEtAl2011}, where an exponentially small probability event has a double exponential impact.
	
	The difference in the MV-SDE (weakly interacting particle) case is that the expectation appears inside the simulation, hence a divergence of a single particle influences multiple particles simultaneously during the simulation and not just at the final time. 
\end{remark}

\textbf{Convergence of Euler and propagation of chaos is impossible.}
The above shows that one particle diverging can cause the whole system to diverge. One may argue that using more particles would reduce the dependency between them and hence influence the system less. In fact as we shall see the opposite is true, the more particles the more likely a divergence is. To test this we use the same example as above but use $N=[1000, ~ 5000, ~ 10000, ~ 20000]$ particles and rerun each case $1000$ times and record the total number of times we observe a divergence over the ensemble.
\begin{table}[htb]
	\centering
	\small
	\begin{tabular}{ c | c | c | c | c}
		Number of particles & $1000$ & $5000$  & $10000$ & $20000$ \\ \hline
		Number of blow ups  & $3$ & $32$ & $43$ & $108$ \\ \hline
	\end{tabular}
	\caption{Number of divergences recorded at each particle level out of $1000$ simulations.}
	\label{Table:Number of blow ups}
\end{table}

The results in Table \ref{Table:Number of blow ups} show conclusively that the more particles the more likely a divergence is to occur. This is a real problem in this setting since in order to minimize the propagation of chaos error one should take $N$ as large as possible, but doing so makes the Euler scheme approximation (more likely to) diverge.

\begin{remark}[Euler cannot work]
	We have shown that naively applying the standard Euler scheme in the MV-SDE setting with non globally Lipschitz coefficient has issues. However, for standard SDEs there are some simple fixes one can apply and still obtain convergence e.g.\ removing paths that leave some ball as considered in \cite{MilsteinTretyakov2005}. Unfortunately, methods of this form cannot work here for the following reason: Either we take the ball ``small'' and therefore our approximation to the law is poor. Or we take a large ball, but then, particles close to the boundary may ``drag'' other particles with them, which again makes the system unstable. 
	
	The dependence on the measure (other particles) implies that the cruder approximation techniques cannot yield the strong convergence results we obtain with the more sophisticated techniques presented in this paper.
	In \cite{BaladronFasoliFaugerasEtAl2012} (see note \cite{BossyEtAl2015}) the authors have a non-globally Lipschitz MV-SDE and simulate using standard Euler scheme. Since no divergence was observed in their simulations they conjectured that the Euler scheme works in their setting. However, their example used a ``small'' diffusion coefficient ($\sigma\in[0,0.5]$) and small particle number (in the order of hundreds). As our analysis points out, this makes divergence unlikely to be observed (but not impossible). For MV-SDEs in general, approaches such as taking a small particle number will yield poor approximation results and therefore not appropriate. Again, our methods provide certainty in terms of convergence (and convergence rate).
\end{remark}

\textbf{Phase transition and particle systems within a bistable potential.}  We have applied our algorithms to the problem highlighted in \cite{gomes2019mean} (see their equation (2.1) and the setup of their Section 3.5) and shortly report that we recover the same findings as above to their problem when dealing with the bistable potential $V_\eta(\eta)=\eta^4/4-\eta^2/2$. Divergence of the explicit Euler scheme in \cite[{Section 3.5}]{gomes2019mean} when using $V'_\eta(\eta)$, while both schemes that we propose behave as we have described. We do not provide the numerical experiments as it would be a repetition of the results above.


\subsection{Timing of implicit vs explicit: size of cloud and spatial dimension}
It is well documented that for a fixed stepsize implicit schemes are slower than explicit ones, mainly because one must solve a fixed point equation at each step. This operation is not ``cheap'' and moreover scales $d^{2}$ in dimension, see \cite{HutzenthalerEtAl2012}. Of course this analysis is carried out for standard SDEs. What we wish to consider is how the particle system affects the timing of both methods.

We consider the same example as previous and again use $h=0.05$, but take $T=1$, we then consider a set of dimensions from $1$ to $200$ and number of particles from $100$ to $20000$. Plotting the time taken for both methods is given in Figure \ref{Fig:Explicit and Implicit Timing}.

\begin{figure}[bht]%
	\centering
	\subfloat[Explicit Scheme]{{\includegraphics[width=7.3cm]{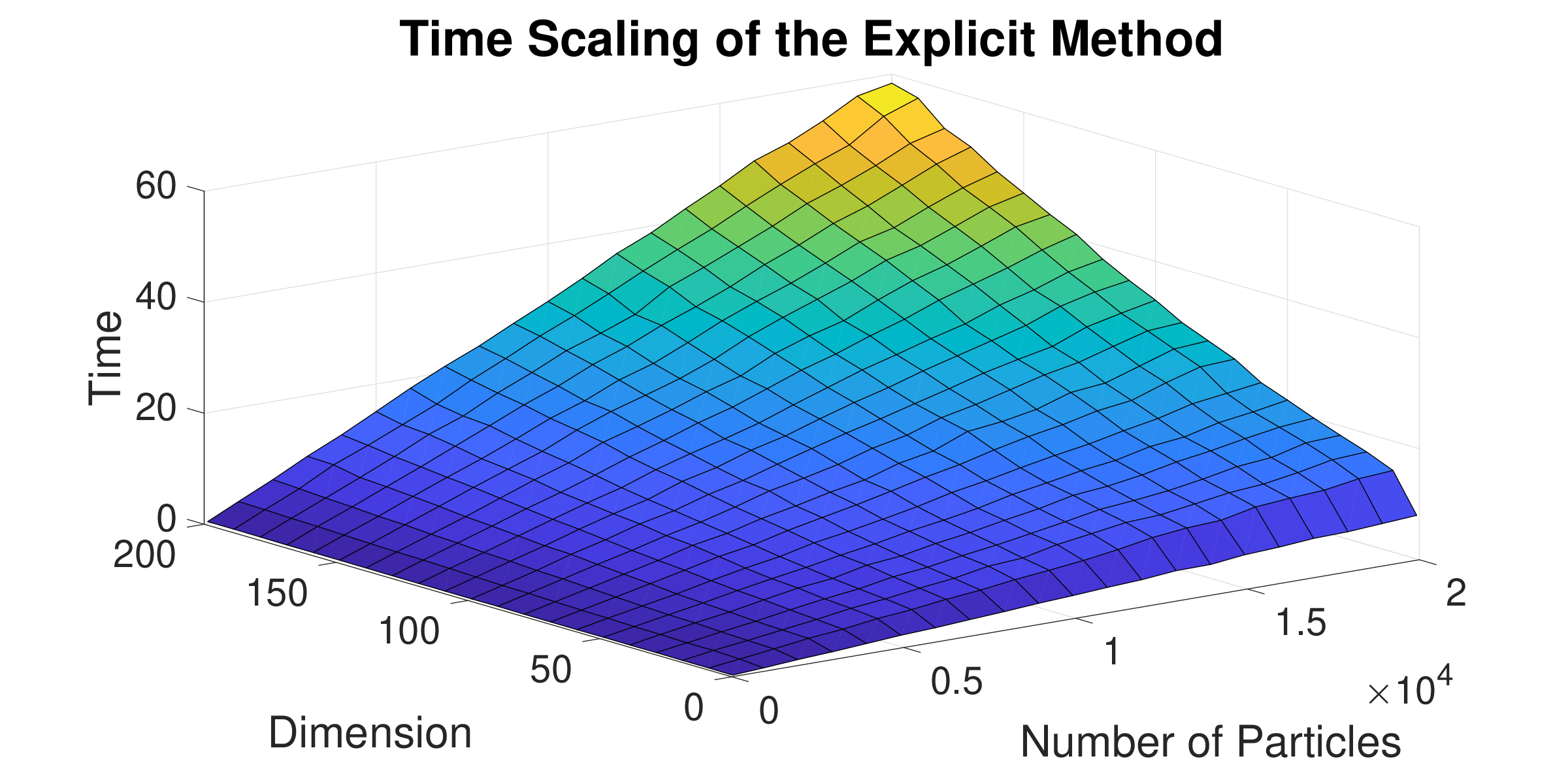} }}%
	\hspace{-0.6cm}
	\subfloat[Implicit Scheme]{{\includegraphics[width=7.3cm]{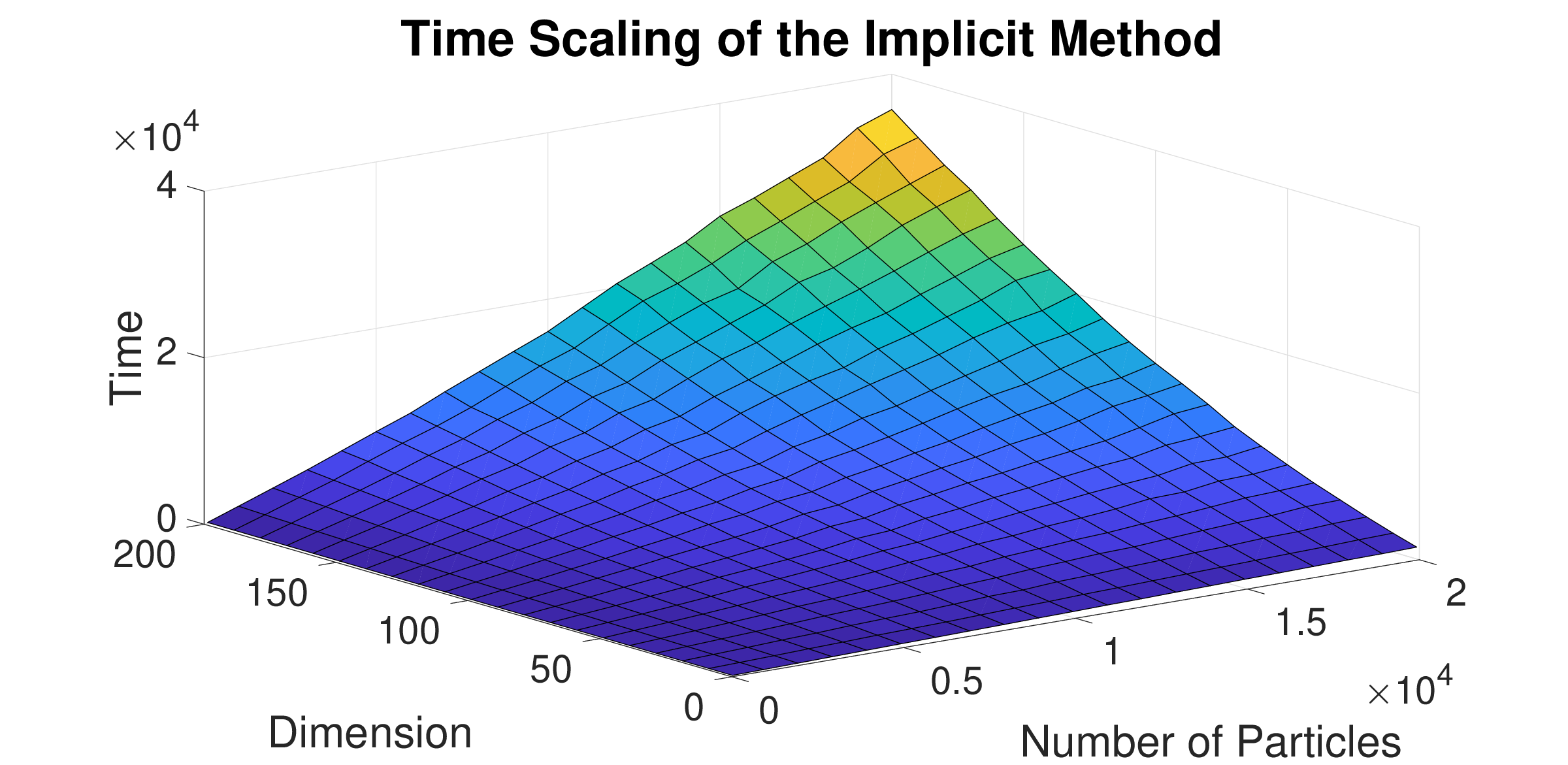} }}%
	\caption{Showing how the CPU-time (in seconds) of the explicit scheme (left; timescale $\approx60$ seconds) and implicit scheme (right; timescale $\approx10^4$ seconds) changes with particles and dimension. }%
	\label{Fig:Explicit and Implicit Timing}%
\end{figure}

Firstly, we observe that the explicit scheme is two to three orders of magnitude faster than the implicit scheme. At the highest dimensional and particle number this difference is very apparent with the tamed scheme taking approximately 1 minute and the implicit 10 hours. Another note to make is the scaling of each method: both methods scale similarly with particle number, but the tamed scheme scales linearly with dimension; this is superior to the $d^{2}$ scaling of the implicit scheme.

Even for the case $d=1$, $N=20000$ the tamed scheme takes approximately $7$ seconds while the implicit scheme takes approximately $23$ minutes. For many practical applications $N=20000$ is not enough for an acceptable level of accuracy, with this in mind and the dimension scaling, this makes the implicit scheme a very expensive method in this setting.


\subsection{Explicit vs implicit convergence: the neuron network model}
\label{sec:NeuRonNumerics}

We compare the convergence of the explicit and the implicit scheme. To this end we use the system in \cite{BaladronFasoliFaugerasEtAl2012} where the authors develop a non globally Lipschitz MV-SDE to model neuron activity.
In our notation their system with $b:[0,T] \times \IR^3 \times \cP_2(\IR^3) \to \IR^3$, $\sigma:[0,T] \times \IR^3 \times \cP_2(\IR^3) \to \IR^{3\times 3}$ reads for $x=(x_1,x_2,x_3),z=(z_1,z_2,z_3) \in \IR^3$ as
\begin{align*}
b \left( t, x, \mu \right) 
&:= \left( \begin{array}{c}
x_1 - (x_1)^3 / 3 - x_2 + I - \int_{\IR^3} J \left( x_1 - V_{rev} \right) z_3 \dd \mu (z) \\
c \left( x_1 + a - b x_2 \right) \\
a_r \frac{ T_{max} (1 - x_3) }{ 1 + \exp ( - \lambda ( x_1 - V_T)) } - a_d x_3
\end{array} \right) \\
\sigma \left( t, x, \mu \right) 
&:= \left( \begin{array}{ccc}
\sigma_{ext} & 0 & - \int_{\IR^3} \sigma_J \left( x_1 - V_{rev} \right) z_3 \dd \mu (z) \\
0 & 0 & 0 \\
0 & \sigma_{32} (x) & 0
\end{array} \right)
\end{align*}
with
\begin{equation*}
\sigma_{32} (x) := \1_{\{x_{3} \in (0,1)\}} \sqrt{ a_r \frac{ T_{max} (1 - x_3) }{ 1 + \exp ( - \lambda ( x_1 - V_T)) } + a_d x_3 } \ \Gamma \exp ( - \Lambda / (1 - (2x_3 - 1)^2)),
\end{equation*}
$T=2$ is chosen as the final time and
\begin{equation*}
X_0 \sim \mathcal{N} \left( \left( \begin{array}{c} V_0 \\ w_0 \\ y_0 \end{array} \right) , \left( \begin{array}{ccc} \sigma_{V_0} & 0 & 0 \\ 0 & \sigma_{w_0} & 0 \\ 0 & 0 & \sigma_{y_0} \end{array} \right) \right),
\end{equation*}
where the parameters have the values
\begin{equation*}
\begin{array}{ccccccc}
V_0 = 0 & \sigma_{V_0} = 0.4   & a = 0.7   & b = 0.8    & c=0.08   & I = 0.5  & \sigma_{ext} = 0.5 
\\
w_0 = 0.5 & \sigma_{w_0} = 0.4 & V_{rev} = 1   & a_r = 1  & a_d = 1   & T_{max} = 1   & \lambda = 0.2  
\\
y_0 = 0.3 & \sigma_{y_0} = 0.05  & J = 1   & \sigma_J = 0.2  & V_T = 2   & \Gamma = 0.1   & \Lambda = 0.5.
\end{array}
\end{equation*}
As the true solution is unknown to compare the convergence rates, we use as proxy the output of the explicit scheme with $2^{23}$ steps. Since the explicit scheme has convergence rate $\sqrt{h}$ we know that $2^{16}$ steps and below yields one order of magnitude larger errors. The simulation for $1000$ particles and average root mean square error of each particle is given in Figure \ref{Fig:ExplicitVsImplicitConvergence}.

One can observe that although initially the implicit scheme has a better rate of convergence, it levels off to yield the expected $1/2$ rate\footnote{\label{footnote01}One can note that the x-axis is written in terms of runtime rather than number of time-steps. As there is a one to one correspondence between the time-steps and the time taken we can still determine the rate. However, this scale allows one to compare both the rate and the time-taken to achieve a given error.}, therefore, making the implicit scheme less computationally efficient than the explicit scheme. Of course our ``true'' was calculated from the explicit scheme, hence we additionally carried out a similar test with a ``true'' from the implicit, and the results were almost identical.

  \begin{figure}[htb]
	\centering
	\includegraphics[width=9cm]{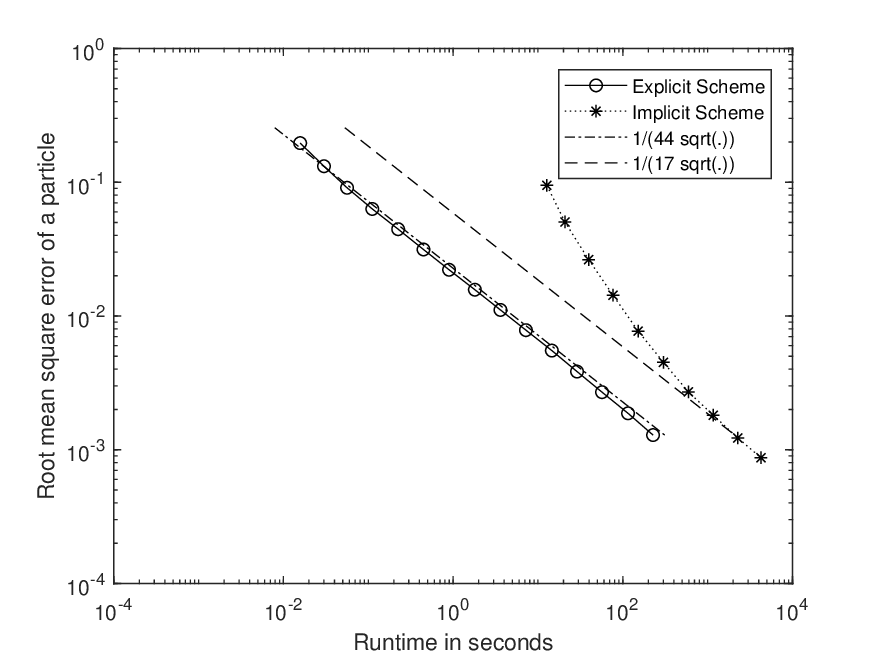}
	\caption{Root mean square error of the explicit and implicit (see footnote \ref{footnote01}). The number of steps, {with $h=T/M$}, of the explicit scheme are $M \in \{2^2, 2^3, \ldots, 2^{16} \}$ and of the implicit scheme are $M \in \{ 2^2, 2^3, \ldots, 2^{11} \}$. We used $1000$ particles and the true is calculated from the explicit with $2^{23}$ steps. Both schemes converge with rate $1/2$. }
	\label{Fig:ExplicitVsImplicitConvergence}
\end{figure}

\begin{remark}[Small diffusion setting]
Above, we have taken $\sigma_{ext}=0.5$, this goes against the example in \cite{BaladronFasoliFaugerasEtAl2012} where $\sigma_{ext}=0$. As it turns out, in the case $\sigma_{ext} = 0$, the implicit scheme has a convergence rate close to $1$ (up to an error of around $10^{-4}$), while the explicit scheme maintains the standard $1/2$ rate. It is our belief that this is due to the fact that when $\sigma_{ext}=0$ the diffusion coefficient makes little difference, hence both scheme revert close to their deterministic convergence rate. The explicit scheme of course still has a rate of order $1/2$, while the implicit is order $1$. It may therefore be that in the setting of small diffusion terms the implicit can yield superior results, of course though this is a special case and is not true in general.
\end{remark}

\subsection*{Obtaining the density}

  \begin{figure}[ht]
	\centering
	\includegraphics[width=10cm]{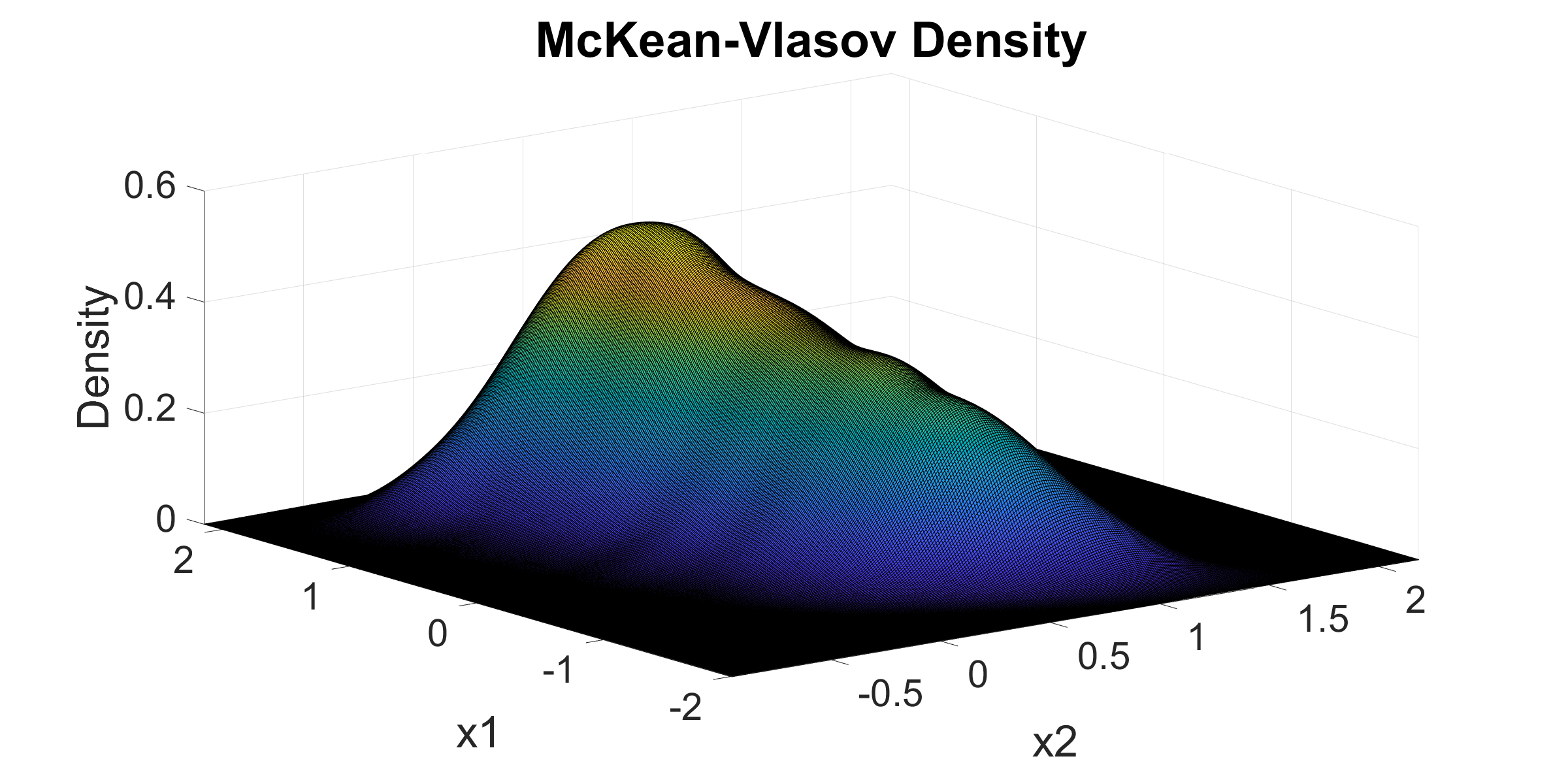}
	\caption{Approximate density of the first and second component of the MV-SDE at time $T=1.2$. We used $10000$ particles, $2^{20}$ steps and a bandwidth of $0.15$ in the kernel smoothing.}
	\label{Fig:DensityApproximation}
\end{figure}

In some applications one may be interested in approximating the distribution (law) of the MV-SDE at a future time. In \cite[Section 4]{BaladronFasoliFaugerasEtAl2012} the authors compare density estimates using both the Fokker-Plank equation and the histogram from the particle system. The approach using PDEs becomes computationally expensive here if one considers multiple populations of MV-SDE and hence the authors take a simple case (see \cite[Section 4.3]{BaladronFasoliFaugerasEtAl2012}). There are of course other drawbacks such as dimension scaling, which often make stochastic techniques more favourable in this setting. 
Moreover, using the PDE approach, one will only obtain the density. While \cite{BaladronFasoliFaugerasEtAl2012} apply a basic histogram approach when using MV-SDEs, this does not yield particularly nice results, namely, the resultant density is not a smooth surface. There are however many statistical techniques one can use to improve this, see \cite[Chapter 18.4]{Keener2011} for further results and discussion. Taking the example in \cite{BaladronFasoliFaugerasEtAl2012} (with $\sigma_{ext}=0$) and applying MATLAB's \verb|ksdensity| function we obtain Figure \ref{Fig:DensityApproximation}.

One can observe the similarity between our result using SDEs and the one obtained in \cite[pg 31]{BaladronFasoliFaugerasEtAl2012} using the (expensive) PDE approach.

%
%
%
\subsection*{Conclusions and future work}

We have shown how one can apply the techniques from SDEs to the MV-SDE setting and some of its pitfalls and challenges that arise. The numerical testing carried out shows that the explicit scheme yields superior results (over the implicit scheme) in general.

Although we have been able to obtain convergence for the implicit scheme it is under stronger assumptions than the explicit scheme (the implicit scheme works very well in Section \ref{sec:NeuRonNumerics}). The reason for these assumptions is that the implicit scheme is more challenging to bound than the explicit. The standard approach around this problem is to use stopping time arguments. However, as described in Remark \ref{rem:No stopping times}, stopping times are harder to handle in the MV-SDE framework. Caution is needed to account for the extra technicalities that arise. 

It is our belief that Assumption \ref{Ass:Extra Implicit Bounds}, although sufficient, is not necessary to guarantee that the implicit scheme converges. As research is carried out into stopping times and MV-SDEs, future theoretical developments in this direction may allow this assumption to be weakened. 
We also leave open a proof for the convergence rate of the implicit scheme.  Showing such a convergence rate in our framework is clearly possible but adds little in scope given the gains of the explicit over the implicit scheme. We leave the question open until a time a more resourceful implicit scheme can be designed.

Another interesting area which we have not discussed is sign preservation and the impact it has on the law. For example a MV-SDE may be known to be positive. However, if the numerical scheme takes the solution into the negative region how does the law dependence influence the remaining particles? One can consider the special case of $L_{b}<0$ in Assumption \ref{Ass:Monotone Assumption}, even though the MV-SDE could have a nonnegative solution, the numerical scheme may not preserve this feature.


\section{Proof of Main Results}
\label{Sec:Proofs}

We shall use $C$ to denote a constant that can change from line to line, but only depends on known quantities, $T$, $d$, the one-sided Lipschitz coefficients {and other constants independent of $h$ or $M$ (with $h=T/M$).}

\subsection{Propagation of chaos}
Let us show the propagation of chaos result.

\begin{proof}[Proposition \ref{Prop:Propagation of Chaos}]
	Let us fix $1 \le i \le N$, we then approach the proof in the usual way for dealing with one-sided Lipschitz coefficients, namely we apply It\^{o}'s formula to the difference (note the $X_{0}^{i}$ cancel out),
	\begin{align}
	|X_{t}^{i} - X_{t}^{i,N}|^{2}
	= &
	\int_{0}^{t} 2 \langle X_{s}^{i} - X_{s}^{i,N}, b(s,X_{s}^{i}, \mu_{s})- b(s, X_{s}^{i,N}, \mu^{X,N}_{s}) \rangle \dd s
	\notag
	\\
	&
	+
	\int_{0}^{t} 2 \langle X_{s}^{i} - X_{s}^{i,N}, (\sigma(s,X_{s}^{i}, \mu_{s})- \sigma(s, X_{s}^{i,N}, \mu^{X,N}_{s})) \dd W_{s}^{i} \rangle
	\notag
	\\
	&
	+
	\sum_{a=1}^{l} \int_{0}^{t} | \sigma_{a}(s,X_{s}^{i}, \mu_{s})- \sigma_{a}(s, X_{s}^{i,N}, \mu^{X,N}_{s}) |^{2} \dd s \, ,
	\label{Eq:Ito Prop Expansion}
	\end{align}
	where $\sigma_{a}$ is the $a$th column of matrix $\sigma$, hence $\sigma_{a}$ is a $d$-dimensional vector. Considering the first integral in \eqref{Eq:Ito Prop Expansion},
	\begin{align*}
	\langle X_{s}^{i} - X_{s}^{i,N}&, b(s,X_{s}^{i}, \mu_{s})- b(s, X_{s}^{i,N}, \mu^{X,N}_{s}) \rangle
\\
&
	=
	\langle X_{s}^{i} - X_{s}^{i,N}, b(s,X_{s}^{i}, \mu_{s})- b(s, X_{s}^{i,N}, \mu_{s}) \rangle
	+
	\langle X_{s}^{i} - X_{s}^{i,N}, b(s, X_{s}^{i,N}, \mu_{s})- b(s, X_{s}^{i,N}, \mu^{X,N}_{s}) \rangle.
	\end{align*}
	Applying the one-sided Lipschitz property in space and $W^{(2)}$ in measure along with Cauchy-Schwarz, we obtain
	\begin{align*}
	\langle X_{s}^{i} - X_{s}^{i,N}, b(s,X_{s}^{i}, \mu_{s})&- b(s, X_{s}^{i,N}, \mu^{X,N}_{s}) \rangle
	\le
	C|X_{s}^{i}-X_{s}^{i,N}|^{2}
	+
	C|X_{s}^{i}-X_{s}^{i,N}| W^{(2)}(\mu_{s}, \mu^{X,N}_s) \, .
	\end{align*}
	{As in \cite{Sznitman1991} or \cite{Meleard1996}}, we introduce the empirical measure  constructed from i.i.d.~samples of the true solution $\mu_{s}^{N} := \frac{1}{N} \sum_{j=1}^{N} \delta_{X_{s}^{j}}$. As $W^{(2)}$ is a metric (see \cite[Chapter 6]{Villani2008}), we have
	\begin{align*}
	W^{(2)}(\mu_{s}, \mu^{X,N}_s)
	\le
	W^{(2)}(\mu_{s}, \mu_{s}^{N})
	+
	W^{(2)}(\mu_{s}^{N}, \mu^{X,N}_s) \, .
	\end{align*}
	Since $\mu_{s}^{N}$, $\mu^{X,N}_s$ are empirical measures a standard result for the Wasserstein metric is 
	\begin{align*}
	W^{(2)}(\mu_{s}^{N}, \mu^{X,N}_s) 
	\le
	\Big( \frac{1}{N} \sum_{j=1}^{N} |X_{s}^{j} - X_{s}^{j,N}|^{2} \Big)^{1/2} \, .
	\end{align*}
	We leave the other $W^{(2)}$ term for the moment and consider the diffusion coefficient in the time integral. Since $\sigma$ is globally Lipschitz and $W^{(2)}$ for each $a$ (by definition $\sigma_{a}= \sigma e_{a}$, with $e_{a}$ the basis vector, global Lipschitz follows from our norm), we get
	\begin{align*}
	&| \sigma_{a}(s, X_{s}^{i}, \mu_{s}) - \sigma_{a}(s, X_{s}^{i,N}, \mu^{X,N}_s)|^{2}
	\\
	&
	\le
	C \big( | \sigma_{a}(s, X_{s}^{i}, \mu_{s}) - \sigma_{a}(s, X_{s}^{i,N}, \mu_{s})|^{2}
	+
	| \sigma_{a}(s, X_{s}^{i,N}, \mu_{s}) - \sigma_{a}(s, X_{s}^{i,N}, \mu^{X,N}_s)|^{2}
	\big)
	\\
	&
	\le C \big( |X_{s}^{i} - X_{s}^{i,N}|^{2} + W^{(2)}( \mu_{s}, \mu^{X,N}_s)^{2} \big)
	\\
	&
	\le C \big( |X_{s}^{i} - X_{s}^{i,N}|^{2} + \frac{1}{N} \sum_{j=1}^{N} |X_{s}^{j}- X_{s}^{j,N}|^{2} + W^{(2)}( \mu_{s}, \mu_{s}^{N})^{2} \big) \, .
	\end{align*} 
	One can note this is independent of $a$.
	The final term to bound is the stochastic integral term. To do this we apply the supremum and expectation operator to \eqref{Eq:Ito Prop Expansion}
	\begin{align}
	\notag
	\bE \Big[ \sup_{t \in [0,T]}|X_{t}^{i} - X_{t}^{i,N}|^{2} \Big]
	&
	\le 
	C \bE \Big[ \sup_{t \in [0,T]}
	\int_{0}^{t} |X_{s}^{i}-X_{s}^{i,N}|^{2}
	+
	|X_{s}^{i}-X_{s}^{i,N}| W^{(2)}(\mu_{s}, \mu^{X,N}_s) \dd s
	\Big]
	\notag
	\\
	&
	+
	\bE \Big[ \sup_{t \in [0,T]}
	\int_{0}^{t} 2 \langle X_{s}^{i} - X_{s}^{i,N}, (\sigma(s,X_{s}^{i}, \mu_{s})- \sigma(s, X_{s}^{i,N}, \mu^{X,N}_s)) \dd W_{s}^{i} \rangle
	\Big]
	\notag
	\\
	&
	+
	C l \bE \Big[ \sup_{t \in [0,T]}
	\int_{0}^{t}   |X_{s}^{i} - X_{s}^{i,N}|^{2} + \frac{1}{N} \sum_{j=1}^{N} |X_{s}^{j}- X_{s}^{j,N}|^{2} + W^{(2)}( \mu_{s}, \mu_{s}^{N})^{2}  \dd s
	\Big] \, .
\label{Eq:Expected Difference Prop}
	\end{align}
	For the stochastic integral,
	\begin{align*}
	&
	\bE \Big[ \sup_{t \in [0,T]}
	\int_{0}^{t} 2 \langle X_{s}^{i} - X_{s}^{i,N}, (\sigma(s,X_{s}^{i}, \mu_{s})- \sigma(s, X_{s}^{i,N}, \mu^{X,N}_s)) \dd W_{s}^{i} \rangle
	\Big]
	\\
	&
	\le
	\bE \Big[ \sup_{t \in [0,T]}
	\Big|
	\int_{0}^{t} 2 \langle X_{s}^{i} - X_{s}^{i,N}, (\sigma(s,X_{s}^{i}, \mu_{s})- \sigma(s, X_{s}^{i,N}, \mu^{X,N}_s)) \dd W_{s}^{i} \rangle
	\Big|
	\Big]
	\\
	&
	\le
	C \bE \Big[ 
	\Big(
	\int_{0}^{T}  \Big(
	\sum_{a=1}^{l} | \sigma_{a}(s,X_{s}^{i}, \mu_{s})- \sigma_{a}(s, X_{s}^{i,N}, \mu^{X,N}_s) |^{2}
	\Big) 
	|X_{s}^{i} - X_{s}^{i,N}|^{2}
	\dd s
	\Big)^{1/2}
	\Big]
	\\
	&
	\le
	\bE \Big[ 
	\Big(
	\sup_{t \in [0,T]}
	|X_{t}^{i} - X_{t}^{i,N}|^{2}
	C \int_{0}^{T}  
	\sum_{a=1}^{l} 
	| \sigma_{a}(s,X_{s}^{i}, \mu_{s})- \sigma_{a}(s, X_{s}^{i,N}, \mu^{X,N}_s) |^{2}
	\dd s
	\Big)^{1/2}
	\Big] \, ,
	\end{align*}
	where we have applied Burkholder-Davis-Gundy to remove the stochastic integral.
	Using Young's inequality $ab \le a^{2}/2 + b^{2}/2$ we can bound this term by
	\begin{align*}
	\bE \Big[ 
	\frac{1}{2}
	\sup_{t \in [0,T]}
	|X_{t}^{i} - X_{t}^{i,N}|^{2}
	+
	\frac{C}{2} \int_{0}^{T} 
	\sum_{a=1}^{l} 
	| \sigma_{a}(s,X_{s}^{i}, \mu_{s})- \sigma_{a}(s, X_{s}^{i,N}, \mu^{X,N}_s) |^{2}
	\dd s
	\Big] \, .
	\end{align*}
	Substituting into \eqref{Eq:Expected Difference Prop} yields,
	\begin{align*}
	\bE \Big[ \sup_{t \in [0,T]}|X_{t}^{i} - X_{t}^{i,N}|^{2} \Big]
	& \le 
	C \bE \Big[ \sup_{t \in [0,T]}
	\int_{0}^{t} |X_{s}^{i}-X_{s}^{i,N}|^{2}
	+
	|X_{s}^{i}-X_{s}^{i,N}| W^{(2)}(\mu_{s}, \mu^{X,N}_s) \dd s
	\Big]
	\\
	&
	+
	\bE \Big[ 
	\frac{1}{2}
	\sup_{t \in [0,T]}
	|X_{t}^{i} - X_{t}^{i,N}|^{2}
	+
	\frac{C}{2} \int_{0}^{T} 
	\sum_{a=1}^{l} 
	| \sigma_{a}(s,X_{s}^{i}, \mu_{s})- \sigma_{a}(s, X_{s}^{i,N}, \mu^{X,N}_s) |^{2}
	\dd s
	\Big]
	\\
	&
	+
	C \bE \Big[ \sup_{t \in [0,T]}
	\int_{0}^{t}   |X_{s}^{i} - X_{s}^{i,N}|^{2} + \frac{1}{N} \sum_{j=1}^{N} |X_{s}^{j}- X_{s}^{j,N}|^{2} + W^{(2)}( \mu_{s}, \mu_{s}^{N})^{2}  \dd s
	\Big] \, .
	\end{align*}
	Taking the $\frac{1}{2}
	\sup_{t \in [0,T]}
	|X_{t}^{i} - X_{t}^{i,N}|^{2}$ to the other side, noting that the supremum value over the integrals is $t=T$ and using the bound for the difference in $\sigma$ we obtain,
	\begin{align*}
	\bE \left[ \sup_{t \in [0,T]}|X_{t}^{i} - X_{t}^{i,N}|^{2} \right]
	&\le
	C \bE \left[ 
	\int_{0}^{T} |X_{s}^{i}-X_{s}^{i,N}|^{2}
	+
	|X_{s}^{i}-X_{s}^{i,N}| W^{(2)}(\mu_{s}, \mu^{X,N}_s) \dd s
	\right]
	\\
	&
	+
	C \bE \left[ 
	\int_{0}^{T}   |X_{s}^{i} - X_{s}^{i,N}|^{2} + \frac{1}{N} \sum_{j=1}^{N} |X_{s}^{j}- X_{s}^{j,N}|^{2} + W^{(2)}( \mu_{s}, \mu_{s}^{N})^{2}  \dd s
	\right] \, .
	\end{align*}
	To deal with the summation term, observe that since all $j$ are identically distributed,
	\begin{align*}
	\bE \Big[
	\frac{1}{N} \sum_{j=1}^{N} |X_{s}^{j}- X_{s}^{j,N}|^{2} 
	\Big]
	=
	\bE \big[
	|X_{s}^{i}- X_{s}^{i,N}|^{2} 
	\big] \, .
	\end{align*} 
	Therefore, applying Young's inequality to $|X_{s}^{i}-X_{s}^{i,N}| W^{(2)}(\mu_{s}, \mu_{s}^{N})$ and taking the supremum over $i$,
	\begin{align*}
	\sup_{1 \le i \le N} \bE \left[  \sup_{t \in [0,T]}|X_{t}^{i} - X_{t}^{i,N}|^{2} \right]
	\le &
	C 
	\int_{0}^{T} \sup_{1 \le i \le N}  \bE  \left[ |X_{s}^{i}-X_{s}^{i,N}|^{2} \right]
	+ \bE [W^{(2)}( \mu_{s}, \mu_{s}^{N})^{2}]  \dd s
	\\
	\le &
	C  
	\int_{0}^{T}
	\bE \left[W^{(2)}( \mu_{s}, \mu_{s}^{N})^{2} \right]
	\dd s
	\, ,
	\end{align*}	
	where the final step follows from Gr\"onwall's inequality. At this point, one could conclude a pathwise propagation of chaos result, see {\cite{Sznitman1991}, \cite{Meleard1996}, \cite[Lemma 1.9]{Carmona2016}}, however, here we are interested in the rate of convergence.  We use the improved version \cite[Theorem 5.8]{CarmonaDelarue2017book1} of the classical convergence result \cite[Chapter 10.2]{RachevRueschendorf1998-Vol2}. Provided $X_{\cdot}^{i} \in L_{\cdot}^{p}( \bR^{d})$ for any $p>4$, which follows from \cite[Theorem 3.3]{dosReisSalkeldTugaut2017} then for any $s$,
	\begin{equation*}
	\bE \left[W^{(2)}( \mu_{s}, \mu_{s}^{N})^{2} \right] 
	\le C
	\begin{cases}
	N^{-1/2} & \text{if } d<4,
	\\
	N^{-1/2} \log(N) \quad & \text{if } d=4,
	\\
	N^{-2/d} & \text{if } d>4.
	\end{cases} 
	\end{equation*}
	Using the result in Theorem \ref{Thm:MV Monotone Existence} with our assumption then completes the proof.
\end{proof}


\subsection{Proof of explicit convergence}
\label{Sec:Proofs of Explicit}
We prove Proposition \ref{Prop:Tamed Scheme Convergence} by establishing first a few auxiliary results. Recall $h=T/M$. To keep expressions compact we introduce 
	\begin{align*}
	\Delta X_{s}^{i,N,M}
	:=
	X_{s}^{i,N} - X_{s}^{i,N,M} \, \quad \textrm{for $s \in [0,T]$.}
	\end{align*}	
	Further, we will use throughout and without mentioning the following result
	\begin{align*}
	\bE \Big[ \frac{1}{N} \sum_{j=1}^N \left\vert \Delta X_{s}^{j,N,M} \right\vert^2 \Big]
	& = 
	\bE \left[ \left\vert \Delta X_{s}^{i,N,M} \right\vert^2 \right]
	=
	 \sup_{1 \leq j \leq N } \bE \left[ \left\vert \Delta X_{s}^{j,N,M} \right\vert^2 \right],
	\end{align*}
	which holds because for every $i$ the RVs are identically distributed.

\begin{lemma} \label{lemma:EstItoPartial}
	Suppose Assumption \ref{Ass:Monotone Assumption} and \ref{Ass:Holder in Time} are fulfilled, then there exists a constant $C$ which is independent of $N$ and $h$ {(or $M$ with $h=T/M$) }such that
	\begin{align*}
		\langle X^{i,N,M}_{t}, b_M \left( t, X^{i,N,M}_{t}, \mu^{X,N,M}_{t} \right) \rangle
		\leq 
		C \left( 1 + \vert X^{i,N,M}_{t} \vert^2 + \frac{1}{N} \sum_{j=1}^N \vert X^{j,N,M}_{t} \vert^2 \right)
	\end{align*}
	and
	\begin{align*}
		\left\vert \sigma \left( t, X^{i,N,M}_{t}, \mu^{X,N,M}_{t} \right) \right\vert^{2}
		\leq 
		C \left( 1 + \vert X^{i,N,M}_{t} \vert^2 + \frac{1}{N} \sum_{j=1}^N \vert X^{j,N,M}_{t} \vert^2 \right)
		.
	\end{align*}
\end{lemma}

\begin{proof}
	First, observe for $x,x' \in \bR^d$ and $\mu \in \cP ( \bR^d )$ that
	\begin{align*}
	\langle x - x', b_M \left( t, x, \mu \right) - b_M (t, x', \mu ) \rangle
	\hspace*{-4cm} & \\
	&= 
	\frac{ \langle x-x', b (t, x, \mu ) - b (t,x',\mu) \rangle }{ 1 + M^{-\alpha} \vert b (t,x,\mu ) \vert } 
		+ \langle x-x', \frac{ b ( t,x',\mu ) ( \vert b(t,x,\mu) \vert - \vert b(t,x',\mu) ) \vert }{ ( M^\alpha + \vert b (t,x,\mu) \vert) ( M^\alpha + \vert b (t,x',\mu) \vert) } \rangle
	\\ &
	\leq \frac{ \langle x-x', b (t, x, \mu ) - b (t,x',\mu) \rangle }{ 1 + M^{-\alpha} \vert b (t,x,\mu ) \vert } 
		+ \vert x-x' \vert^2 
		+ \left\vert \frac{ \vert b (t,x',\mu ) \vert^2 - \vert b (t,x,\mu) \vert \vert b (t,x',\mu ) \vert }{ \vert b (t,x,\mu ) \vert \vert b (t,x',\mu ) \vert } \right\vert^2
		.
	\end{align*}
	Assuming without loss of generality (otherwise just switch $x$ and $x'$) that $\vert b (t,x,\mu)\vert \geq \vert b (t,x',\mu) \vert$ we get by Assumption \ref{Ass:Monotone Assumption}
	\begin{align*}
	\langle x - x', b_M \left( t, x, \mu \right) - b_M (t, x', \mu ) \rangle
	\leq (L_b + 1) \vert x- x' \vert^2 + 1
	.
	\end{align*}
	Similarily we obtain for all $x \in \bR^d$ and $\mu,\mu' \in \cP ( \bR^d )$
	\begin{align*}
	\vert b_M (t,x,\mu) - b_M ( t,x,\mu') \vert
	\leq 
	\vert b (t,x,\mu) - b (t,x,\mu') \vert + 1
	\leq L W^{(2)} (\mu, \mu') + 1
	.
	\end{align*}
	Using this, we have
	\begin{align*}
		\langle X^{i,N,M}_{t}, b_M \left( t, X^{i,N,M}_{t}, \mu^{X,N,M}_{t} \right) \rangle
		\hspace*{-2.2cm} & \\ &
		=
		\langle X^{i,N,M}_{t} - 0, b_M \left( t, X^{i,N,M}_{t}, \mu^{X,N,M}_{t} \right) - b_M \left( t, 0, \delta_0 \right) \rangle + \langle X^{i,N,M}_{t} , b_M \left( t, 0, \delta_0 \right) \rangle
		\\ &
		\leq L_b \vert X^{i,N,M}_{t} \vert^2 + 2 \vert X^{i,N,M}_{t} \vert^2 + L W^{(2)} \left( \mu^{X,N,M}_{t}, \delta_0 \right)^2 + 1 + \vert X^{i,N,M}_{t} \vert^2 + \left\vert b_M \left( t, 0, \delta_0 \right) \right\vert^2
		\\ &
		\leq C \Big( 1 + \vert X^{i,N,M}_{t} \vert^2 + \frac{1}{N} \sum_{j=1}^N \vert X^{j,N,M}_{t} \vert^2 \Big)
	\end{align*}
	by the 1/2-Hölder-continuity in Assumption \ref{Ass:Holder in Time}. Again with Assumption~\ref{Ass:Monotone Assumption} and \ref{Ass:Holder in Time} we get
	\begin{align*}
		\left\vert \sigma \left( t, X^{i,N,M}_{t}, \mu^{X,N,M}_{t} \right) \right\vert^{2}
		&
		\leq \left\vert \sigma \left( t, X^{i,N,M}_{t}, \mu^{X,N,M}_{t} \right) - \sigma \left( t, 0, \delta_0 \right) \right\vert^{2} + \big\vert \sigma \left( t, 0, \delta_0 \right) \big\vert^{2}
		\\ &
		\leq L_\sigma \left( \vert X^{i,N,M}_{t} \vert^2 + W^{(2)} \left( \mu^{X,N,M}_{t}, \delta_0 \right)^2 \right) + \big\vert \sigma \left( t, 0, \delta_0 \right) \big\vert^{2}
		\\&
		\leq C \left( 1 + \vert X^{i,N,M}_{t} \vert^2 + \frac{1}{N} \sum_{j=1}^N \vert X^{j,N,M}_{t} \vert^2 \right)
		.
	\end{align*}
\end{proof}

\begin{lemma}
	\label{Lemma:supEXMsmallerC}
	Suppose Assumption \ref{Ass:Monotone Assumption} and \ref{Ass:Holder in Time} are fulfilled and $X_{0} \in L^{2}(\bR^{d})$, then there exists a constant $C$ which is independent of $N$ and $h$ {(or $M$ with $h=T/M$) } such that
	\begin{equation*}
	\sup_{h>0} \sup_{1 \leq i \leq N} \sup_{0\leq t \leq T} \bE \left[ \vert X^{i,N,M}_t \vert^2 \right] < C.
	\end{equation*}
\end{lemma}

\begin{proof}
	Applying It\^o's formula and restructuring the terms gives
	\begin{align*}
		\left\vert X^{i,N,M}_{t} \right\vert^2 
		=& 
		\left\vert X^i_0 \right\vert ^2 
		+ \int_0^t 2 \langle X^{i,N,M}_{\kappa(s)}, b_M \left( \kappa(s), X^{i,N,M}_{\kappa(s)}, \mu^{X,N,M}_{\kappa(s)} \right) \rangle  
		+ \sum_{a=1}^{l} \left\vert \sigma_{a} \left( \kappa(s), X^{i,N,M}_{\kappa(s)}, \mu^{X,N,M}_{\kappa(s)} \right) \right\vert^{2} \dd s 
		\\
		& + \int_{0}^{t} 2 \langle X^{i,N,M}_{s}, \sigma \left( \kappa(s), X^{i,N,M}_{\kappa(s)}, \mu^{X,N,M}_{\kappa(s)} \right) \dd W_{s}^{i} \rangle 
		\\
		& + \int_0^t 2 \langle X^{i,N,M}_{s} - X^{i,N,M}_{\kappa(s)}, b_M \left( \kappa(s), X^{i,N,M}_{\kappa(s)}, \mu^{X,N,M}_{\kappa(s)} \right) \rangle \dd s.
	\end{align*}
	We start with the expectations of the last term (using $\vert s-\kappa(s)\vert {\leq h=T/M}$ and $\alpha\in(0,1/2]$) 
	\begin{align*}
		& \left\vert \bE \left[ \int_0^t \langle X^{i,N,M}_s - X^{i,N,M}_{\kappa(s)}, b_M \left( \kappa(s), X^{i,N,M}_{\kappa(s)}, \mu^{X,N,M}_{\kappa(s)} \right) \rangle \dd s \right] \right\vert 
		\\
		& = 
		\Big\vert \bE \Big[ \int_0^t \langle \int_{\kappa(s)}^s b_M \left( \kappa(r), X^{i,N,M}_{\kappa(r)}, \mu^{X,N,M}_{\kappa(r)} \right) \dd r 
		\\
		&
		\qquad + \int_{\kappa(s)}^s \sigma \left( \kappa(r), X^{i,N,M}_{\kappa(r)}, \mu^{X,N,M}_{\kappa(r)} \right) \dd W^i_r ,   
		b_M \left( \kappa(s), X^{i,N,M}_{\kappa(s)}, \mu^{X,N,M}_{\kappa(s)} \right) \rangle \dd s \Big] \Big\vert 
		\\
		& = \Big\vert \sum_{k=0}^{M-1}
		\int_{t_{k}}^{t_{k+1}} \1_{ \{s \leq t \} } \bE \Big[ \Big\langle b_M \Big( \kappa(s), X^{i,N,M}_{\kappa(s)}, \mu^{X,N,M}_{\kappa(s)} \Big),  
		\\
		& \hspace*{1cm}  \bE \Big[  
		\int_{t_{k}}^s b_M \Big( \kappa(r), X^{i,N,M}_{\kappa(r)}, \mu^{X,N,M}_{\kappa(r)} \Big) \dd r 
		+
		\int_{t_{k}}^s \sigma \Big( \kappa(r), X^{i,N,M}_{\kappa(r)}, \mu^{X,N,M}_{\kappa(r)} \Big) \dd W^i_r 
		\Big\vert \cF_{t_{k}} \Big] \Big\rangle \Big] \dd s \Big\vert 
		\\
		& \leq \bE \left[ \int_0^t \left\vert b_M \left( \kappa(s), X^{i,N,M}_{\kappa(s)}, \mu^{X,N,M}_{\kappa(s)} \right) \right\vert \int_{\kappa(s)}^s \left\vert b_M \left( \kappa(r), X^{i,N,M}_{\kappa(r)}, \mu^{X,N,M}_{\kappa(r)} \right) \right\vert \dd r \ \dd s \right] 
		\\
		& \leq {C h^{1-2\alpha}}
		\\
		&\leq {C. }
	\end{align*}
	Putting this together and using Lemma \ref{lemma:EstItoPartial} we obtain
	\begin{align*}
		\bE \big[ \big\vert X^{i,N,M}_{t} \big\vert^2 \big]
		& \leq \bE \big[ \big\vert X^i_0 \big\vert^2 \big] 
		+ C \Big( 1 + \bE \Big[ \int_0^t \left\vert X^{i,N,M}_{\kappa(s)} \right\vert^2 + \frac{1}{N} \sum_{j=1}^N \left\vert X^{j,N,M}_{\kappa(s)} \right\vert^2 \dd s \Big] \Big) 
		\\
		& \leq \bE \big[ \vert X^i_0 \vert^2 \big] 
		+ C \Big( 1 + \int_0^t \sup_{0\leq u \leq s} \bE \Big[ \left\vert X^{i,N,M}_{u} \right\vert^2 + \frac{1}{N} \sum_{j=1}^N \left\vert X^{j,N,M}_{u} \right\vert^2 \Big] \dd s  \Big) ,
	\end{align*}
	which furthermore yields
	\begin{align*}
		\sup_{1 \leq i \leq N} \sup_{0 \leq u \leq t} &\bE \left[ \left\vert X^{i,N,M}_{u} \right\vert^2 \right]
		\leq C \left( 1 + \bE \left[ \left\vert X_0 \right\vert^2 \right] + \int_0^t \sup_{1 \leq i \leq N} \sup_{0 \leq u \leq s} \bE \left[ \left\vert X^{i,N,M}_{u} \right\vert^2 \right] \dd s \right) < \infty,
	\end{align*}
	and hence by Gr\"onwall's lemma
	\begin{equation*}
		\sup_{1 \leq i \leq N} \sup_{0 \leq u \leq t} \bE \left[ \left\vert X^{i,N,M}_{u} \right\vert^2 \right] < C,
	\end{equation*}
	where $C$ is a constant which is independent of $N$ and $h$ (and $M$).
\end{proof}

\begin{lemma}
	\label{Lemma:XtMinusKappaLeqCM}
	If Assumption \ref{Ass:Monotone Assumption} and \ref{Ass:Holder in Time} are fulfilled and $X_{0} \in L^{2}(\bR^{d})$, then for all $p \in (0,2]$ we have
	\begin{equation}
	\label{Eq:XtMinusXkappatLeqCM}
	\sup_{1 \leq i \leq N} \sup_{0\leq t \leq T} 
	\bE \left[ \left\vert X^{i,N,M}_t - X^{i,N,M}_{\kappa(t)} \right\vert^p \right] 
	\leq {C h^{p/2}},
	\end{equation}
	and
	\begin{equation}
	\label{Eq:XminusXtimesBMLeqC}
	\sup_{1 \leq i \leq N} \sup_{0\leq t \leq T} \bE \left[ \left\vert X^{i,N,M}_t - X^{i,N,M}_{\kappa(t)} \right\vert^p \left\vert b_M \left( \kappa(t), X^{i,N,M}_{\kappa(t)}, \mu^{X,N,M}_{\kappa(t)} \right) \right\vert^p \right] \leq C,
	\end{equation}
	where $C$ is a positive constant independent of  $N$ and $h$ (and $M$). 
	Furthermore, if for $p>2$ (with $h=T/M$)
	\begin{equation*}
	\sup_{h >0 } \sup_{1 \leq i \leq N} \bE \left[ \sup_{0 \leq t \leq T} \left\vert X^{i,N,M}_t \right\vert^p \right] < \infty,
	\end{equation*}
	then the estimates \eqref{Eq:XtMinusXkappatLeqCM} and \eqref{Eq:XminusXtimesBMLeqC} hold for those $p$ as well.
\end{lemma}
\begin{proof}[Proof of Lemma \ref{Lemma:XtMinusKappaLeqCM}]
	We obtain for any $p \geq 2$
	\begin{align}
		&	\left\vert \int_{\kappa(t)}^t b_M \left( \kappa(s), X^{i,N,M}_{\kappa(s)}, \mu^{X,N,M}_{\kappa(s)} \right) \dd s \right\vert^p
		\leq {T^{p/2} h^{p/2}}, \label{eq:est_int_bM}
	\end{align}
	since $\vert b_M \vert \leq M^\alpha {= T^\alpha h^{-\alpha}}$ and $\alpha \leq 1/2$.
	It is easy to see that in the case of $p \in (0,2]$
	\begin{align*}
		&\bE \left[ \left \vert X^{i,N,M}_t - X^{i,N,M}_{\kappa(t)} \right\vert^p \right] 
		\\
		&  
		\leq
		\bE \Big[ \Big\vert \int_{\kappa(t)}^t b_M \big( \kappa(s), X^{i,N,M}_{\kappa(s)}, \mu^{X,N,M}_{\kappa(s)} \big) \dd s + \int_{\kappa(t)}^t \sigma \big( \kappa(s), X^{i,N,M}_{\kappa(s)}, \mu^{X,N,M}_{\kappa(s)} \big) \dd W^i_s \Big\vert^2 \Big]^{\frac{p}{2}} 
		\\
		& \leq
		2^{p/2} \bE \Big[ \Big\vert \int_{\kappa(t)}^t b_M \big( \kappa(s), X^{i,N,M}_{\kappa(s)}, \mu^{X,N,M}_{\kappa(s)} \big) \dd s \Big\vert^2 
		+ \Big\vert \int_{\kappa(t)}^t \sigma \big( \kappa(s), X^{i,N,M}_{\kappa(s)}, \mu^{X,N,M}_{\kappa(s)} \big) \dd W^i_s \Big\vert^2 \Big]^{\frac{p}{2}},
	\end{align*}
	and due to It\^o's isometry, Lemma \ref{lemma:EstItoPartial} and Lemma \ref{Lemma:supEXMsmallerC} for $C$ independent of $h$ (and $M$) and $i$
	\begin{align*}
		&
		\bE \left[ \left\vert \int_{\kappa(t)}^t \sigma \left( \kappa(s), X^{i,N,M}_{\kappa(s)}, \mu^{X,N,M}_{\kappa(s)} \right) \dd W^i_s \right\vert^2 \right]
		\\
		& \leq \bE \left[ \int_{\kappa(t)}^t K \left( 1 + \left\vert X^{i,N,M}_{\kappa(s)} \right\vert^2 + \frac{1}{N} \sum_{j=1}^N \left\vert X^{j,N,M}_{\kappa(s)} \right\vert^2 \right) \dd s \right] \\
		& \leq \sup_{1 \leq i \leq N} \sup_{s \in [\kappa(t),t]} 
		\bE \left[ {h} 
		K \left( 1 + \left\vert X^{i,N,M}_s \right\vert^2 + \left\vert X^{i,N,M}_{s} \right\vert^2 \right) \right] 
		\leq {C h} 
	\end{align*}
	which gives, combined with \eqref{eq:est_int_bM}, that
	\begin{equation*}
		\sup_{0\leq t \leq T} \bE \left[ \left\vert X^{i,N,M}_t - X^{i,N,M}_{\kappa(t)} \right\vert^p \right] \leq C { h^{p/2}},
\quad \textrm{for all $p\in (0,2]$.}
	\end{equation*}
	If additionally 
	$\sup_{M \geq 1} \sup_{1 \leq i \leq N} \bE \big[ \sup_{0 \leq t \leq T} \vert X^{i,N,M}_t \vert^p \big] < \infty$
	for some $p>2$, then (using \eqref{eq:est_int_bM})
	\begin{align*}
		&
		\bE \left[ \left \vert X^{i,N,M}_t - X^{i,N,M}_{\kappa(t)} \right\vert^p \right] 
		\\
		&   
		\leq
		C \bE \Bigg[ \Big\vert \int_{\kappa(t)}^t \hspace*{-0.05cm} b_M \big( \kappa(s), X^{i,N,M}_{\kappa(s)}, \mu^{X,N,M}_{\kappa(s)} \big) \dd s \Big\vert^p 
		+ 
		\Big\vert \int_{\kappa(t)}^t \hspace*{-0.05cm} \sigma \big( \kappa(s), X^{i,N,M}_{\kappa(s)}, \mu^{X,N,M}_{\kappa(s)} \big) \dd W^i_s \Big\vert^p \Bigg] 
		\\
		&   \leq 
		C \bE \left[ 
		{ T^{p/2} h^{p/2}} 
			+ \Big\vert \int_{\kappa(t)}^t \sigma \big( \kappa(s), X^{i,N,M}_{\kappa(s)}, \mu^{X,N,M}_{\kappa(s)} \big)^2 \dd s \Big\vert^{p/2} \right],
	\end{align*}
	by the estimate \eqref{eq:est_int_bM} and the Burkholder-Davis-Gundy inequality. Since furthermore,
	\begin{align*}
		& \bE \left[ \left\vert \int_{\kappa(t)}^t \sigma \left( \kappa(s), X^{i,N,M}_{\kappa(s)}, \mu^{X,N,M}_{\kappa(s)} \right)^2 \dd s \right\vert^{p/2} \right] \\
		& \hspace*{1cm} \leq \bE \left[ 
		{h^{p/2}} 
		\sup_{s \in [\kappa(t),t]} K \left( 1 + \left\vert X^{i,N,M}_{s} \right\vert^p + \left( \frac{1}{N} \sum_{j=1}^N \left\vert X^{j,N,M}_s \right\vert ^2 \right)^{p/2} \right) \right]  \\
		& \hspace*{1cm} 
		\leq { h^{p/2}} 
		K \left( 1 + \bE \left[ \sup_{0 \leq t \leq T} \left\vert X^{i,N,M}_t \right\vert^p \right] + \sup_{1 \leq j \leq N} \bE \left[ \sup_{0 \leq t \leq T} \left\vert X^{j,N,M}_t \right\vert^p \right] \right) \\
		& \hspace*{1cm} \leq C{ h^{ p/2} },
	\end{align*}
	we get the desired result here as well.
	
		Finally, using the above results and that $\alpha \le 1/2$, we obtain for any $p \ge 0$ for which \linebreak $\bE [ \vert X^{i,N,M}_t - X^{i,N,M}_{\kappa(t)} \vert^p ] \le C { h^{p/2}}$, that
	\begin{align*}
	&\bE \left[ \left\vert X^{i,N,M}_t - X^{i,N,M}_{\kappa(t)} \right\vert^p \left\vert b_M \big( \kappa(t), X^{i,N,M}_{\kappa(t)}, \mu^{X,N,M}_{\kappa(t)} \big) \right\vert^p \right]
	\leq \bE \left[ \left\vert X^{i,N,M}_t - X^{i,N,M}_{\kappa(t)} \right\vert^p \right] 
	{ T^{p\alpha} h^{-p\alpha} }
	\leq C,
	\end{align*}
	holds for any $t \in [0,T]$ and $1 \leq i \leq N$, which completes the proof.
\end{proof}

\begin{lemma}
	\label{Lemma:EsupXMsmallerC}
	Suppose that Assumption \ref{Ass:Monotone Assumption} and \ref{Ass:Holder in Time} are fulfilled, then for every $p \geq 2$ with $X_{0} \in L^{p}(\bR^{d})$ there exists a constant $C$ such that
	\begin{equation*}
	\sup_{h>0 } \sup_{1 \leq i \leq N} \bE \left[ \sup_{0\leq t \leq T} \left\vert X^{i,N,M}_t \right\vert^p \right] < C.
	\end{equation*}
\end{lemma}

\begin{proof}
Define $\hat{p} \ge 2$ such that $\bE[|X_{0}|^{\hat{p}}] < \infty$ and note that for $p<2$ Lemma \ref{Lemma:XtMinusKappaLeqCM} yields immediately the result.

We use an inductive argument and start with $p = 2$. In every step we set $q= 2p \wedge \hat{p}$. 
By It\^o's formula and Lemma \ref{lemma:EstItoPartial} we have
\begin{align*}
\left\vert X^{i,N,M}_t \right\vert^2
\leq & \
\left\vert X^{i,N,M}_0 \right\vert^2 
+ \int_0^t \left\vert X^{i,N,M}_s - X^{i,N,M}_{\kappa(s)} \right\vert \left\vert b_M \big( \kappa(s), X^{i,N,M}_{\kappa(s)}, \mu^{X,N,M}_{\kappa(s)} \right\vert \dd s
\\ & 
+ \int_0^t C \Big( 1 + \left\vert X^{i,N,M}_{\kappa(s)}\right\vert^2 + \frac{1}{N} \sum_{j=1}^N \left\vert X^{j,N,M}_{\kappa(s)}\right\vert^2 \Big) \dd s
+ \left\vert \int_0^t X^{i,N,M}_u \sigma \big( \kappa(u), X^{i,N,M}_{\kappa(u)}, \mu^{X,N,M}_{\kappa(u)} \big) \dd W^i_u \right\vert
.
\end{align*}
With the inequality $\vert a+b \vert^{q/2} \leq C ( \vert a \vert^{q/2} + \vert b \vert^{q/2})$ and Jensen's inequality we therefore obtain
\begin{align*}
\bE \big[ \sup_{0 \leq s \leq t} \left\vert X^{i,N,M}_s \right\vert^q \big]
&	
\leq  C \Big( 1 + \bE \Big[ \big\vert X^{i,N,M}_0 \big\vert^q \Big] + \int_0^t \bE \Big[ \big\vert X^{i,N,M}_{\kappa(s)} \big\vert^q \Big] \dd s
\\
&  + \int_0^t \bE \Big[ \big\vert X^{i,N,M}_s - X^{i,N,M}_{\kappa(s)} \big\vert^{q/2} \big\vert b_M \big( \kappa(s), X^{i,N,M}_{\kappa(s)}, \mu^{X,N,M}_{\kappa(s)} \big) \big\vert^{q/2} \Big] \dd s 
\\
&  + \bE \Big[ \sup_{0 \leq s \leq t} \big\vert \int_0^s X^{i,N,M}_u \sigma \big( \kappa(u), X^{i,N,M}_{\kappa(u)}, \mu^{X,N,M}_{\kappa(u)} \big) \dd W^i_u \big\vert^{q/2} \Big] \Big).
\end{align*}
The application of the Burkholder-Davis-Gundy inequality and Lemma \ref{Lemma:XtMinusKappaLeqCM} with\footnote{Observe that Lemma \ref{Lemma:XtMinusKappaLeqCM} holds for the current value of $p$ and since $q=2p \wedge \hat{p}$ it implies that it holds for $q/2$.} $q/2$ yields
	\begin{align*}
		\bE \left[ \sup_{0 \leq s \leq t} \left\vert X^{i,N,M}_s \right\vert^q \right]
		\leq & C \left( 1 + \bE \left[ \left\vert X^{i,N,M}_0 \right\vert^q \right] + \int_0^t \bE \left[ \sup_{0 \leq u \leq s} \left\vert X^{i,N,M}_{u} \right\vert^q \right] \dd s \right. \\
		& \left. + \bE \left[ \left( \int_0^t \left\vert X^{i,N,M}_s \right\vert^2 \left\vert \sigma \left( \kappa(s), X^{i,N,M}_{\kappa(s)}, \mu^{X,N,M}_{\kappa(s)} \right) \right\vert^2 \dd s \right)^{q/4} \right] \right),
	\end{align*}
	where $C$ denotes in each case a constant that is independent of $h$ (or $M$). With Young's inequality in the form $ab \leq \frac{1}{2C}a^2 + \frac{C}{2}b^2$, H\"older's inequality and the estimate for $\sigma$ we have
	\begin{align*}
	\bE 
	\left[ \sup_{0 \leq s \leq t} \left\vert X^{i,N,M}_s \right\vert^q \right]
	\leq &
	C \left( 1 + \bE \left[ \left\vert X^{i,N,M}_0 \right\vert^q \right] + \int_0^t \bE \left[ \sup_{0 \leq u \leq s} \left\vert X^{i,N,M}_{u} \right\vert^q \right] \dd s + \frac{1}{2C} \bE \left[ \sup_{0 \leq s \leq t} \left\vert X^{i,N,M}_s \right\vert^q \right] \right.
	\\
	& \qquad \left. + \frac{C}{2} \bE \left[ \int_0^t \left\vert \sigma \left( \kappa(s), X^{i,N,M}_{\kappa(s)}, \mu^{X,N,M}_{\kappa(s)} \right) \right\vert^q \dd s \right] \right) 
	\\
	\leq & 
	C \Big( 1 + \bE \left[ \left\vert X^{i,N,M}_0 \right\vert^q \right] + \int_0^t \bE \left[ \sup_{0 \leq u \leq s} \left\vert X^{i,N,M}_{u} \right\vert^q \right] \dd s + \frac{1}{2C} \bE \left[ \sup_{0 \leq s \leq t} \left\vert X^{i,N,M}_s \right\vert^q \right]  
	\\
	& \qquad  + \frac{C}{2} \int_{0}^{t} \bE \Big[ \sup_{0 \leq u \leq s} K \Big( 1 + \left\vert X^{i,N,M}_{u} \right\vert^q + \Big( \frac{1}{N} \sum_{j=1}^N \left\vert X^{j,N,M}_{u} \right\vert^2 \Big)^{q/2} \Big) \Big] \dd s \Big).
	\end{align*}
	Taking the $\frac{1}{2}\bE[\sup_{0 \leq s \leq t} \vert X^{i,N,M}_s \vert^q]$ term to the LHS taking the $\sup$ over $i$ on both sides we obtain
	\begin{align*}
		\sup_{1 \leq i \leq N} \bE \left[ \sup_{0 \leq s \leq t} \left\vert X^{i,N,M}_s \right\vert^q \right]
	\leq C \left( 1 + \bE \left[ \big\vert X^{i,N,M}_0 \big\vert^q \right] + \int_0^t \sup_{1 \leq i \leq N} \bE \left[ \sup_{0 \leq u \leq s} \left\vert X^{i,N,M}_{u} \right\vert^q \right] \dd s \right) < \infty,
	\end{align*}
	and thus the application of Gr\"onwall's lemma yields that
	\begin{equation}
	\label{eq:EsupXMleqC}
	\sup_{1 \leq i \leq N } \bE \Big[ \sup_{0 \leq t \leq T } \big\vert X^{i,N,M}_t \big\vert^q \Big] < C,
	\end{equation}
	for some positive constant $C$ which depends on $\bE[|X_{0}^{i}|^{q}]$ but is independent of $N$ and { $h$ (or $M$)}.
	
	Since \eqref{eq:EsupXMleqC} is proven for $q$ we can set $p=q$ and use this result in the next step of the iteration. Since the new $q$ is at most twice as much as $p$, Lemma \ref{Lemma:XtMinusKappaLeqCM} can again be applied for $q/2$. This iteration gets repeated until $q = \hat{p}$.
\end{proof}

Now we can complete the proof of Proposition \ref{Prop:Tamed Scheme Convergence}.

\begin{proof}[Proof of Proposition \ref{Prop:Tamed Scheme Convergence}]
	Using It\^o's formula we observe,
	\begin{align*}
	\left\vert \Delta X_{t}^{i,N,M} \right\vert^2 
	 & = \int_0^t 2 \langle\Delta X_{s}^{i,N,M}, \left( b \left( s, X^{i,N}_s, \mu^{X,N}_s \right) - b_M \left( \kappa(s), X^{i,N,M}_{\kappa(s)}, \mu^{X,N,M}_{\kappa(s)} \right) \right) \rangle \dd s 
	\\
	&
	\quad
	 + \sum_{a=1}^{l}
	 \int_{0}^{t} \vert \sigma_{a} \left( s, X^{i,N}_s, \mu^{X,N}_s \right) - \sigma_{a} \left( \kappa(s), X^{i,N,M}_{\kappa(s)}, \mu^{X,N,M}_{\kappa(s)} \right) \vert^{2} \dd s 
	 \\
	& \quad
	 + \int_{0}^{t} 2 \langle\Delta X_{s}^{i,N,M}, \left( \sigma \left( s, X^{i,N}_s, \mu^{X,N}_s \right) - \sigma \left( \kappa(s), X^{i,N,M}_{\kappa(s)}, \mu^{X,N,M}_{\kappa(s)} \right) \right) \dd W_{s}^{i} \rangle.
	\end{align*}
	Furthermore observe that
	\begin{align*}
	\langle  X^{i,N}_{s} & - X^{i,N,M}_{s} , b \big( s, X^{i,N}_s, \mu^{X,N}_s \big) - b_M \big( \kappa(s), X^{i,N,M}_{\kappa(s)}, \mu^{X,N,M}_{\kappa(s)} \big) \rangle 
	\\
	= & 
	\langle\Delta X_{s}^{i,N,M} , b \big( s, X^{i,N}_s, \mu^{X,N}_s \big) - b \big( s, X^{i,N,M}_s, \mu^{X,N}_s \big) \rangle 
	\\
	& + \langle\Delta X_{s}^{i,N,M} , b \big( s, X^{i,N,M}_s, \mu^{X,N}_s \big) - b \big( s, X^{i,N,M}_s, \mu^{X,N,M}_s \big) \rangle 
	\\
	& + \langle\Delta X_{s}^{i,N,M} , b \big( s, X^{i,N,M}_s, \mu^{X,N,M}_s \big) - b \big( \kappa(s), X^{i,N,M}_s, \mu^{X,N,M}_s \big) \rangle 
	\\
	& + \langle\Delta X_{s}^{i,N,M} , b \big( \kappa(s), X^{i,N,M}_s, \mu^{X,N,M}_s \big) - b \big( \kappa(s), X^{i,N,M}_{\kappa(s)}, \mu^{X,N,M}_s \big) \rangle 
	\\
	& + \langle\Delta X_{s}^{i,N,M} , b \big( \kappa(s), X^{i,N,M}_{\kappa(s)}, \mu^{X,N,M}_s \big) - b \big( \kappa(s), X^{i,N,M}_{\kappa(s)}, \mu^{X,N,M}_{\kappa(s)} \big) \rangle 
	\\
	& + \langle\Delta X_{s}^{i,N,M} , b \big( \kappa(s), X^{i,N,M}_{\kappa(s)}, \mu^{X,N,M}_{\kappa(s)} \big) - b_M \big( \kappa(s), X^{i,N,M}_{\kappa(s)}, \mu^{X,N,M}_{\kappa(s)} \big) \rangle,
	\end{align*}
	where we estimate every term on the right hand side as follows. 
	Due to Assumption \ref{Ass:Monotone Assumption} we have
	\begin{equation*}
	\langle\Delta X_{s}^{i,N,M} , b \left( s, X^{i,N}_s, \mu^{X,N}_s \right) - b \left( s, X^{i,N,M}_s, \mu^{X,N}_s \right) \rangle 
	\leq L_b \left\vert\Delta X_{s}^{i,N,M} \right\vert^2,
	\end{equation*}
	and
	\begin{align*}
	 \langle \Delta X_{s}^{i,N,M}  , b \left( s, X^{i,N,M}_s, \mu^{X,N}_s \right) - b \left( s, X^{i,N,M}_s, \mu^{X,N,M}_s \right) \rangle 
	& \leq \left\vert\Delta X_{s}^{i,N,M} \right\vert \vert W^{(2)} \left( \mu^{X,N}_s, \mu^{X,N,M}_s \right) \vert 
	\\
	&  \leq \left\vert\Delta X_{s}^{i,N,M} \right\vert \frac{1}{\sqrt{N}} \Big( \sum_{j=1}^N \vert \Delta X_{s}^{j,N,M} \vert ^2 \Big)^{1/2} 
	\\
	& \leq \frac{1}{2} \left\vert\Delta X_{s}^{i,N,M} \right\vert^2 + \frac{1}{2}\frac{1}{N} \sum_{j=1}^N \vert \Delta X_{s}^{j,N,M} \vert ^2,
	\end{align*}
	and with Assumption \ref{Ass:Holder in Time}
	\begin{align*}
	 \langle\Delta X_{s}^{i,N,M} & , b \left( s, X^{i,N,M}_s, \mu^{X,N,M}_s \right) - b \left( \kappa(s), X^{i,N,M}_s, \mu^{X,N,M}_s \right) \rangle 
	\\
	&  
	\qquad \qquad \leq C \left\vert\Delta X_{s}^{i,N,M} \right\vert \vert s - \kappa(s) \vert ^{1/2} 
		\leq \frac{1}{2} \left\vert \Delta X_{s}^{i,N,M} \right\vert^2 +  { C h }.
	\end{align*}
	Further,
	\begin{align*}
	& \big\langle \Delta X_{s}^{i,N,M} , b \big( \kappa(s), X^{i,N,M}_s, \mu^{X,N,M}_s \big) - b \big( \kappa(s), X^{i,N,M}_{\kappa(s)}, \mu^{X,N,M}_s \big) \big\rangle
	 \\
	& \leq \frac{1}{2} \left\vert \Delta X_{s}^{i,N,M} \right\vert^2 + \frac{1}{2} \left\vert b \left( \kappa(s), X^{i,N,M}_s, \mu^{X,N,M}_s \right) - b \left( \kappa(s), X^{i,N,M}_{\kappa(s)}, \mu^{X,N,M}_s \right) \right\vert^2,
	\end{align*}
	which we can furthermore dominate by using the polynomial growth of $b$ with rate $q$, Cauchy-Schwarz,  Lemma \ref{Lemma:EsupXMsmallerC} and Lemma \ref{Lemma:XtMinusKappaLeqCM}, to have 
	\begin{align*}
	& \bE \left[ \sup_{u \in [0,t]} \int_0^u \left\vert b \left( \kappa(s), X^{i,N,M}_s, \mu^{X,N,M}_s \right) - b \left( \kappa(s), X^{i,N,M}_{\kappa(s)}, \mu^{X,N,M}_s \right) \right\vert^2 \dd s \right] \\
	& \leq \int_0^{t} \bE \left[ L \left( 1 + \left\vert X^{i,N,M}_s \right\vert^q + \left\vert X^{i,N,M}_{\kappa(s)} \right\vert^q \right)^2 \left\vert X^{i,N,M}_s - X^{i,N,M}_{\kappa(s)} \right\vert^2 \right] \dd s \\
	& \leq \int_0^{t} \sqrt{ \bE \left[ L \left( 1 + \left\vert X^{i,N,M}_s \right\vert^q + \left\vert X^{i,N,M}_{\kappa(s)} \right\vert^q \right)^4 \right] \bE \left[ \left\vert X^{i,N,M}_s - X^{i,N,M}_{\kappa(s)} \right\vert^4 \right] } \dd s 
	\leq \int_0^t \sqrt{ {C h^{2} } } \dd s 
	 \leq { C h } 
	\end{align*}  
	since 
	$$
	\sup_{{h>0}} 
	\sup_{1 \leq i \leq N} \bE \left[ \sup_{0\leq t \leq T} \left\vert X^{i,N,M}_t \right\vert^{4q} \right] 
	\leq 1 + 
	\sup_{{h>0}} 
	\sup_{1 \leq i \leq N} \bE \left[ \sup_{0\leq t \leq T} \left\vert X^{i,N,M}_t \right\vert^{4(1+q)} \right] < \infty.$$
	Again Assumption \ref{Ass:Monotone Assumption} yields
	\begin{align*}
	& \langle \Delta X_{s}^{i,N,M} , b \left( \kappa(s), X^{i,N,M}_{\kappa(s)}, \mu^{X,N,M}_s \right) - b \left( \kappa(s), X^{i,N,M}_{\kappa(s)}, \mu^{X,N,M}_{\kappa(s)} \right) \rangle \\
	& \leq 
	\left\vert \Delta X_{s}^{i,N,M} \right\vert \frac{1}{\sqrt{N}} \Big( \sum_{j=1}^N \vert X^{j,N,M}_{s} - X^{j,N,M}_{\kappa(s)} \vert ^2 \Big)^{1/2} 
	\leq 
	\frac{1}{2} \left\vert \Delta X_{s}^{i,N,M} \right\vert^2 + \frac{1}{2} \frac{L^2}{N} \sum_{j=1}^N \left\vert X^{j,N,M}_{s} - X^{j,N,M}_{\kappa(s)} \right\vert ^2,
	\end{align*}
	and the definition of $b_M$ together with $\vert a - \frac{a}{1+ M^{-\alpha} \vert a \vert } \vert = \vert a \frac{ M^{-\alpha} \vert a \vert}{ 1 + M^{-\alpha} \vert a \vert } \vert \leq \vert a \vert^2 M^{-\alpha}={\vert a \vert^2 T^{-\alpha} h^{\alpha}} $ that
	\begin{align*}
	& \langle \Delta X_{s}^{i,N,M} , b \big( \kappa(s), X^{i,N,M}_{\kappa(s)}, \mu^{X,N,M}_{\kappa(s)} \big) - b_M \big( \kappa(s), X^{i,N,M}_{\kappa(s)}, \mu^{X,N,M}_{\kappa(s)} \big) \rangle 
	\\
	& \leq \frac{1}{2} \left\vert \Delta X_{s}^{i,N,M} \right\vert^2 + \frac{1}{2} \left\vert b \big( \kappa(s), X^{i,N,M}_{\kappa(s)}, \mu^{X,N,M}_{\kappa(s)} \big) - b_M \big( \kappa(s), X^{i,N,M}_{\kappa(s)}, \mu^{X,N,M}_{\kappa(s)} \big) \right\vert^2 
	\\
	& \leq \frac{1}{2} \left\vert \Delta X_{s}^{i,N,M} \right\vert^2 
	+ \frac{1}{2} 
	{ T^{-2\alpha} h^{2\alpha}}
	\big\vert b \left( \kappa(s), X^{i,N,M}_{\kappa(s)}, \mu^{X,N,M}_{\kappa(s)} \right) \big\vert^4 
	\\
	& \leq 
	\frac{1}{2} \left\vert \Delta X_{s}^{i,N,M} \right\vert^2 
	+
	{ C h^{2\alpha}}
	\Big( 1 +  \left\vert X^{i,N,M}_{\kappa(s)} \right\vert^{4(1+q)} 
	+ \Big( \frac{1}{N} \sum_{j=1}^N \left\vert X^{j,N,M}_{\kappa(s)} \right\vert^2 \Big)^2 \Big),
	\end{align*}
	where $q$ is again the polynomial growth rate of $b$.
	Also the Burkholder-Davis-Gundy inequality yields
	\begin{align*}
	& \bE \left[ \sup_{u \in [0,t]} \int_{0}^{u} 2 
	\langle \Delta X_{s}^{i,N,M}, \left( \sigma \big( s, X^{i,N}_s, \mu^{X,N}_s \big) - \sigma \big( \kappa(s), X^{i,N,M}_{\kappa(s)}, \mu^{X,N,M}_{\kappa(s)} \big) \right) \dd W_{s}^{i} \rangle \right] 
	\\
	& \leq 
	\bE \Bigg[ \Big( C \int_{0}^{t} 
	\Big( \sum_{a=1}^{l} \vert \sigma_{a}(s,X^{i,N}_s, \mu^{X,N}_s)- \sigma_{a}(s, X^{i,N,M}_{\kappa(s)}, \mu^{X,N,M}_{\kappa(s)} ) \vert^{2} \Big) \vert \Delta X_{s}^{i,N,M} \vert^{2} 
	\dd s \Big)^{\frac{1}{2}} \Bigg] 
	\\
	& \leq 
	\bE \Big[ \frac{1}{2} \sup_{u \in [0,t]} \vert \Delta X_{u}^{i,N,M} \vert^{2} 
	+ 
	C \int_{0}^{t} \sum_{a=1}^{l} \left\vert \sigma_{a}(s,X^{i,N}_s, \mu^{X,N}_s)- \sigma_{a}( \kappa(s), X^{i,N,M}_{\kappa(s)}, \mu^{X,N,M}_{\kappa(s)}) \right\vert^{2} \dd s \Big].
	\end{align*}
	and
	\begin{align*}
	& \left\vert \sigma_{a} \left( s, X^{i,N}_s, \mu^{X,N}_s \right) - \sigma_{a} \left( \kappa(s), X^{i,N,M}_{\kappa(s)}, \mu^{X,N,M}_{\kappa(s)} \right) \right\vert^{2} 
	\\
	& \leq 
	C \left\vert s - \kappa(s) \right\vert + C \left\vert X^{i,N}_s - X^{i,N,M}_{\kappa(s)} \right\vert^2 + C W^{(2)} \left(\mu^{X,N}_s, \mu^{X,N,M}_{\kappa(s)} \right)^2 
	\\
	& \leq 
	C M^{- 1} + C\left\vert X^{i,N}_s - X^{i,N,M}_{\kappa(s)} \right\vert^2 + \frac{C}{N} \sum_{j=1}^N \left\vert X^{j,N}_s - X^{j,N,M}_{\kappa(s)} \right\vert ^2
	 \\
	&  \leq
	 C M^{- 1} + C\left\vert X^{i,N}_s - X^{i,N,M}_{\kappa(s)} \right\vert^2
	+ \frac{C}{N} \sum_{j=1}^N \left( \left\vert \Delta X_{s}^{j,N,M} \right\vert^2 + \left\vert X^{j,N,M}_s - X^{j,N,M}_{\kappa(s)} \right\vert ^2 \right).
	\end{align*}
	By putting those estimates together we obtain
	\begin{align*}
	\bE \left[ \sup_{0 \leq u \leq t} \left\vert \Delta X_{u}^{i,N,M} \right\vert^2  \right] 
	& \leq 
	C \bE \Bigg[ 
	{ h}
	+ \int_0^t \left\vert \Delta X_{s}^{i,N,M} \right\vert^2
	+ \frac{1}{N} \sum_{j=1}^N \left\vert X^{j,N,M}_{s} - X^{j,N,M}_{\kappa(s)} \right\vert^2 
	+ { h} 	
	+ \frac{1}{N} \sum_{j=1}^N \left\vert \Delta X_{s}^{j,N,M} \right\vert^2 
	\\
	&
	  + \left\vert X^{i,N,M}_{s} - X^{i,N,M}_{\kappa(s)} \right\vert^2
	+ { h^{2\alpha}} 
	\Big( 1 +  \left\vert X^{i,N,M}_{\kappa(s)} \right\vert^{4(1+q)} \Big) 
	 + { h^{2\alpha}} 
	\Big( \frac{1}{N} \sum_{j=1}^N \left\vert X^{j,N,M}_{\kappa(s)} \right\vert^2 \Big)^2 \dd s 
	\Bigg]
	\\
	&
	+ 
	\bE \Big[ \frac{1}{2} \sup_{u \in [0,t]} \vert \Delta X_{u}^{i,N,M} \vert^{2} \Big]
	\end{align*}
	and therefore
	\begin{align*}
	\bE \left[ \sup_{0 \leq u \leq t} \left\vert \Delta X_{u}^{i,N,M} \right\vert^2  \right] 
	\leq \ &
	C  \Big( 
	\int_0^t \bE \Big[ \sup_{0 \leq u \leq s} \big\vert \Delta X_{u}^{i,N,M} \big\vert^2 \Big] \dd s  
	+ { h^{2\alpha}+h} 
	\Big),
	\end{align*}
	by Lemma \ref{Lemma:EsupXMsmallerC} and since $X^{i,N}$ are identically distributed and $X^{i,N,M}$ are identically distributed for all $i \in \{ 1, \ldots, N \}$. This estimate holds for every $i$ hence we can insert $\sup_{1 \leq i \leq N}$ on both sides giving
	\begin{align*}
	\sup_{1 \leq i \leq N} \bE \left[ \sup_{0 \leq u \leq t} \left\vert \Delta X_{u}^{i,N,M} \right\vert^2  \right] 
	& 
	 \leq C \Big( \int_0^t \sup_{1 \leq i \leq N} \bE \Big[ \sup_{0 \leq u \leq s} \left\vert \Delta X_{u}^{i,N,M} \right\vert^2 \Big] \dd s 
	+ { h^{2\alpha}+h} 
	\Big)
	 < \infty,
	\end{align*}
	and finally by Gr\"onwall's lemma (using that $\alpha = 1/2$),
	\begin{align*}
&	\sup_{1 \leq i \leq N} \bE \left[ \sup_{0 \leq u \leq t} \left\vert X^{i,N}_{u} - X^{i,N,M}_{u} \right\vert^2  \right]
	\leq 
	{ C h}. 
	\end{align*}
\end{proof}


\subsection{Proof of implicit convergence}
\label{Sec:Proofs of Implicit}
The main goal here is to prove Proposition \ref{Prop:Strong Implicit Scheme Convergence}. We loosely follow \cite{MaoSzpruch2013}, however, due to the extra dependencies on time and measure and further allowing for random initial conditions we require more refined arguments. We  take $N$ as some fixed positive integer.
Before considering the implicit scheme, let us make a remark and show a result on the particle system \eqref{Eq:MV-SDE Propagation}.

\begin{remark}[Monotone growth]
\label{rem:MonotoneGrowth}
	The combination of Assumption \ref{Ass:Monotone Assumption}, \ref{Ass:Holder in Time} and H\ref{Ass:Bounded Measure}, imply the monotone growth condition. Namely, there exist constants $\alpha,\beta\in \bR$ such that $\forall\hspace*{0.2mm} t \in [0,T], \mu \in \cP_{2}(\bR^{d})$ with $l$ being the dimension of the Brownian motion,
	\begin{align*}
	\langle x, b(t,x,\mu) \rangle
	+
	\frac{1}{2 }
	\sum_{a=1}^{l}
	|\sigma_{a}(t,x,\mu)|^{2} 
	\le
	\alpha + \beta |x|^{2}
	\quad
	\forall x \in \bR^{d}.
	\end{align*}
\end{remark}

\begin{proposition}
	\label{Prop:Particle System Bounds}
	Let Assumption \ref{Ass:Monotone Assumption}, \ref{Ass:Holder in Time} and H\ref{Ass:Bounded Measure} (in Assumption \ref{Ass:Extra Implicit Bounds}) hold, further, let $X_{0} \in L^{2}(\bR^{d})$. Then the following bounds hold,
	\begin{align*}
	\sup_{1 \le i \le N} \bE[|X_{T}^{i,N}|^{2}] \le \big(\bE[|X_{0}|^{2}] + 2 \alpha T \big) \exp (2 \beta T ),
	\end{align*}
	and for $\tau_{m}^{i} = \inf \{t \ge 0 : ~ |X_{t}^{i,N}| >m \}$ we have 
	\begin{align*}
	\sup_{1 \le i \le N} \bP( \tau_{m}^{i} \le T) 
	\le
	\frac{1}{m^{2}} \big(\bE[|X_{0}|^{2}] + 2 \alpha T \big) \exp (2 \beta T ) \, .
	\end{align*}
\end{proposition}

\begin{proof}
	Firstly, let us consider the stopped process $X_{T \wedge \tau_{m}^{i}}^{i,N}$. Applying It\^{o} to the square of this process and taking expectations yields
	\begin{align*}
		\bE[ |X_{T \wedge \tau_{m}^{i}}^{i,N}|^{2}]
		&
		=
		\bE[|X_{0}^{i}|^{2}]
		+
		\bE \Big[ \int_{0}^{T \wedge \tau_{m}^{i}} 2 \langle X_{s}^{i,N}, b(s, X_{s}^{i,N}, \mu_{s}^{X,N}) \rangle
		+
		\sum_{a=1}^{l} |\sigma_{a}(s, X_{s}^{i,N}, \mu_{s}^{X,N})|^{2} \dd s 
		\Big]
		\\
		&
		\le
		\bE[|X_{0}^{i}|^{2}]
		+
		2 \alpha T
		+
		\int_{0}^{T} 2 \beta \bE[ \vert X_{s \wedge \tau_{m}^{i}}^{i,N} \vert^2 ] \dd s 
		\le 
		\big( \bE[|X_{0}^{i}|^{2}]
		+
		2 \alpha T \big) e^{2 \beta T}
		,
	\end{align*}
	where we have used the growth and stopping condition to remove the martingale term, then Remark \ref{rem:MonotoneGrowth}, uniform boundedness of $b$ in the measure component and Gronwall's inequality to obtain the result.
	
	Noting that the following lower bound also holds,
	\begin{align*}
		\bE[ |X_{T \wedge \tau_{m}^{i}}^{i,N}|^{2}]
		\ge
		m^{2} \bP( \tau_{m}^{i} \le T) \, ,
\quad\textrm{we obtain}\quad
%
	%
		\bP( \tau_{m}^{i} \le T) 
		\le
		\frac{1}{m^{2}} \big(\bE[|X_{0}^{i}|^{2}] + 2 \alpha T \big) \exp (2 \beta T ) \, .
	\end{align*}
	Further, since $\lim_{m \rightarrow \infty} |X_{T \wedge \tau_{m}^{i}}^{i,N}| = |X_{T}^{i,N}|$, we obtain by Fatou's lemma,
	\begin{align*}
		\bE[|X_{T}^{i,N}|^{2}]
		\le
		\liminf_{m \rightarrow \infty}
		\bE[|X_{T \wedge \tau_{m}^{i}}^{i,N}|^{2}]
		\le
		\big(\bE[|X_{0}^{i}|^{2}] + 2 \alpha T \big) \exp (2 \beta T ) \, .
	\end{align*}
	The result then follows by noting that $\bE[|X_{0}^{i}|^{2}]=\bE[|X_{0}|^{2}]$ and hence the bounds are independent of $i$, so we obtain the result for the supremum over $i$.
\end{proof}

Let us now return to the implicit scheme. At each time step $t_i$ and for each particle $i$ one needs to solve the fixed point equation { (with $h=T/M$)}
\begin{align*}
\tilde{X}_{t_{k+1}}^{i,N,M} 
-
b\Big(t_{k},\tilde{X}_{t_{k+1}}^{i,N,M}, \tilde{\mu}^{X,N,M}_{t_{k}} \Big) h 
= 
\tilde{X}_{t_{k}}^{i,N,M}
+
\sigma\Big(t_{k},\tilde{X}_{t_{k}}^{i,N,M} , \tilde{\mu}^{X,N,M}_{t_{k}} \Big) \Delta W_{t_{k}}^{i}.
\end{align*}
This leads us to consider a function $F$
\begin{equation}
\label{Eq:implicitToSolve}
F(t,x,\mu) := x - b(t,x,\mu)h.
\end{equation}
For the implicit scheme to have a solution the function $F$ must have a unique inverse.
The following lemma is crucial in proving convergence of the implicit scheme. 
\begin{lemma}
	\label{Lemma:Results for F}
	Let Assumption \ref{Ass:Monotone Assumption}, \ref{Ass:Holder in Time} and  H\ref{Ass:Bounded Measure} (in Assumption \ref{Ass:Extra Implicit Bounds}) hold and fix $h^{*}< 1/\max(L_{b}, 2 \beta)$. Further, let $0 < h \le h^{*}$ and take any $ t \in [0,T]$ and $\mu \in \cP_{2}(\bR^{d})$ fixed. Then for all $y \in \bR^{d}$, there exists a unique $x$ such that $F(t,x,\mu)=y$. Hence the fixed point problem in \eqref{Eq:implicitScheme} is well defined.
	
	Moreover, for all $t \in [0,T]$ and $\mu \in \cP_{2}(\bR^{d})$ the following bound holds,
	\begin{align} \label{est:FixEquation}
	|x|^{2}
	\le
	(1- 2 h \beta)^{-1}(|F(t,x,\mu)|^{2} + 2 h \alpha) \, ,
	\end{align}
	and for any $k \ge 1$ the following recursive bound holds,
	\begin{align}
	|F(t_{k} ,\tilde{X}_{t_{k+1}}^{i,N,M}, \tilde{\mu}_{t_{k}}^{X,N,M})|^{2}
	\notag
	&
	\le
	|F(t_{k-1},\tilde{X}_{t_{k}}^{i,N,M}, \tilde{\mu}_{t_{k-1}}^{X,N,M})|^{2}
	+
	\Big( \sum_{a=1}^{l}|\sigma_{a}(t_{k},\tilde{X}_{t_{k}}^{i,N,M}, \tilde{\mu}_{t_{k}}^{X,N,M})| |\big(\Delta W_{t_{k}}^{i} \big)_{a}|\Big)^{2}
	\notag
	\\
	\label{Eq:Recurrence Bound F}
	&
	+
	2h \alpha
	+ 
	2 h \beta |\tilde{X}_{t_{k}}^{i,N,M}|^{2}
	+
	2
	\langle \tilde{X}_{t_{k}}^{i,N,M},
	\sigma(t_{k},\tilde{X}_{t_{k}}^{i,N,M}, \tilde{\mu}_{t_{k}}^{X,N,M}) \Delta W_{t_{k}}^{i} \rangle
	\, ,	
	\end{align}
	where $\big(\Delta W_{t_{k}}^{i} \big)_{a}$ is the $a$th entry of the vector.
\end{lemma}
\begin{proof}
	Let us first prove there exists a unique solution to \eqref{Eq:implicitToSolve}, in the sense that for all $ t \in [0,T]$ and $\mu \in \cP_{2}(\bR^{d})$ fixed, then there exists a unique $x \in \bR^{d}$ such that $F(t,x,\mu)=y$ for a given $y \in \bR^{d}$, provided $0 < h < h^{*}$. This is a classical problem considered in \cite[p.557]{Zeidler1990-IIB} or see \cite[p.2596]{LionnetReisSzpruch2015}, which requires $F$ to be continuous, monotone and coercive (in $x$). The continuity of $b$ yields that of $F$. For the monotonicity of $F$, we have 
	\begin{align*}
		\langle x- x', F(t,x,\mu) -F(t,x',\mu) \rangle
		&=
		|x-x'|^{2} - \langle x- x', b(t,x,\mu)h -b(t,x',\mu)h \rangle
		\\
		&
		\ge
		|x-x'|^{2} (1- L_{b}h) \, ,
	\end{align*}
	and provided $h< 1/L_{b}$, the final constant is strictly positive. Coercivity follows similarly by the monotone growth condition in $b$,
	\begin{align*}
		\langle x,  F(t,x,\mu) \rangle
		\ge
		|x|^{2}
		- h(\alpha + \beta |x|^{2}) \, ,
	\end{align*}
	therefore, 
	\begin{align*}
		\lim_{|x| \rightarrow \infty}
		\frac{\langle x,  F(t,x,\mu) \rangle}{|x|}
		=
		\infty,
		\quad
		\text{for } h < 1/\beta.
	\end{align*}
	Hence $F(t,x,\mu)=y$ has a unique solution for $F$ defined in \eqref{Eq:implicitToSolve} and therefore the numerical scheme \eqref{Eq:implicitScheme} is well defined.

	To show $x$ is bounded by $F(\cdot, x, \cdot)$, again fix some  $ t \in [0,T]$ and $\mu \in \cP_{2}(\bR^{d})$. Then,
	\begin{align*}
		|F (t,x,\mu)|^{2} 
		&= |x|^{2} - 2 \langle x, b(t,x,\mu) \rangle h + |b(t,x,\mu)|^{2} h^{2}
		\\
		&
		\ge
		|x|^{2} - 2 \langle x, b(t,x,\mu) \rangle h
		\ge
		(1-2h \beta)|x|^{2} - 2h \alpha,
	\end{align*}
	by Remark \ref{rem:MonotoneGrowth}.
	Since $h < 1/(2 \beta)$, we obtain
	\begin{align*}
		|x|^{2}
		\le
		(1- 2 h \beta)^{-1}(|F(t,x,\mu)|^{2} + 2 h \alpha) \, .
	\end{align*}
	This result is useful since it holds for all $ t \in [0,T]$ and $\mu \in \cP_{2}(\bR^{d})$. For the recursive bound it is useful to note
	\begin{align}
		F(t_{k},\tilde{X}_{t_{k+1}}^{i,N,M}, \tilde{\mu}_{t_{k}}^{X,N,M})
		&
		=
		\tilde{X}_{t_{k+1}}^{i,N,M}
		-
		b(t_{k},\tilde{X}_{t_{k+1}}^{i,N,M}, \tilde{\mu}_{t_{k}}^{X,N,M}) h
		\notag
		\\
		\nonumber
		\label{Eq:Useful F Relation}
		&
		=
		\tilde{X}_{t_{k}}^{i,N,M}
		+
		\sigma(t_{k},\tilde{X}_{t_{k}}^{i,N,M}, \tilde{\mu}_{t_{k}}^{X,N,M}) \Delta W_{t_{k}}^{i}
		\\
		&
		=
		F(t_{k-1},\tilde{X}_{t_{k}}^{i,N,M}, \tilde{\mu}_{t_{k-1}}^{X,N,M})
		+
		b(t_{k-1},\tilde{X}_{t_{k}}^{i,N,M}, \tilde{\mu}_{t_{k-1}}^{X,N,M}) h
		\\
		& \quad
		+
		\sigma(t_{k},\tilde{X}_{t_{k}}^{i,N,M}, \tilde{\mu}_{t_{k}}^{X,N,M}) \Delta W_{t_{k}}^{i}. \notag
	\end{align}
	This recursion is only valid for $k \ge 1$ due to the appearance of $t_{k-1}$. Using this relation observe the following,
	\begin{align*}
		|F(t_{k}  ,\tilde{X}_{t_{k+1}}^{i,N,M}, \tilde{\mu}_{t_{k}}^{X,N,M})|^{2}
		&
		=
		|F(t_{k-1},\tilde{X}_{t_{k}}^{i,N,M}, \tilde{\mu}_{t_{k-1}}^{X,N,M})|^{2}
		+
		|b(t_{k-1},\tilde{X}_{t_{k}}^{i,N,M}, \tilde{\mu}_{t_{k-1}}^{X,N,M})|^{2} h^{2}
		\\
		&
		+
		|\sigma(t_{k},\tilde{X}_{t_{k}}^{i,N,M}, \tilde{\mu}_{t_{k}}^{X,N,M}) \Delta W_{t_{k}}^{i}|^{2}
		\\
		&
		+
		2 \langle F(t_{k-1},\tilde{X}_{t_{k}}^{i,N,M}, \tilde{\mu}_{t_{k-1}}^{X,N,M})
		,
		b(t_{k-1},\tilde{X}_{t_{k}}^{i,N,M}, \tilde{\mu}_{t_{k-1}}^{X,N,M}) \rangle h
		\\
		&
		+
		2 \langle F(t_{k-1},\tilde{X}_{t_{k}}^{i,N,M}, \tilde{\mu}_{t_{k-1}}^{X,N,M})
		\\
		&
		\qquad \quad
				+
		b(t_{k-1},\tilde{X}_{t_{k}}^{i,N,M}, \tilde{\mu}_{t_{k-1}}^{X,N,M}) h
		,
		\sigma(t_{k},\tilde{X}_{t_{k}}^{i,N,M}, \tilde{\mu}_{t_{k}}^{X,N,M}) \Delta W_{t_{k}}^{i} \rangle
		.
	\end{align*}
	We now look to bound these various terms. By definition of $F$,
	\begin{align*}
		 2 \langle F(t_{k-1},\tilde{X}_{t_{k}}^{i,N,M}, \tilde{\mu}_{t_{k-1}}^{X,N,M})
		&,
		b(t_{k-1},\tilde{X}_{t_{k}}^{i,N,M}, \tilde{\mu}_{t_{k-1}}^{X,N,M})  \rangle h
		+
		|b(t_{k-1},\tilde{X}_{t_{k}}^{i,N,M}, \tilde{\mu}_{t_{k-1}}^{X,N,M})|^{2} h^{2}
		\\
		&
		\le
		2 \langle \tilde{X}_{t_{k}}^{i,N,M}
		,
		b(t_{k-1},\tilde{X}_{t_{k}}^{i,N,M}, \tilde{\mu}_{t_{k-1}}^{X,N,M}) \rangle \ h
		\le
		2 h \alpha
		+ 2 h \beta |\tilde{X}_{t_{k}}^{i,N,M}|^{2}.
	\end{align*}
	Similarly,
	\begin{align*}
		2 \langle F(t_{k-1},\tilde{X}_{t_{k}}^{i,N,M}, \tilde{\mu}_{t_{k-1}}^{X,N,M})
		+
		&
		b(t_{k-1},\tilde{X}_{t_{k}}^{i,N,M}, \tilde{\mu}_{t_{k-1}}^{X,N,M}) h
		,
		\sigma(t_{k},\tilde{X}_{t_{k}}^{i,N,M}, \tilde{\mu}_{t_{k}}^{X,N,M}) \Delta W_{t_{k}}^{i} \rangle
		\\
		&
		=
		2 \langle \tilde{X}_{t_{k}}^{i,N,M}
		,
		\sigma(t_{k},\tilde{X}_{t_{k}}^{i,N,M}, \tilde{\mu}_{t_{k}}^{X,N,M}) \Delta W_{t_{k}}^{i} \rangle .
	\end{align*}
	In order to obtain the desired form we note
	\begin{align*}
		\sigma(t,x,\mu) \Delta W_{t} = \sum_{a=1}^{l} \sigma_{a}(t,x,\mu) (\Delta W_{t})_{a} \, .
	\end{align*} 
	Crucially $ (\Delta W_{t})_{a}$ is a scalar and standard properties of norms yield,
	\begin{align*}
		|\sigma(t_{k},\tilde{X}_{t_{k}}^{i,N,M}, \tilde{\mu}_{t_{k}}^{X,N,M}) \Delta W_{t_{k}}^{i}|
		\le
		\sum_{a=1}^{l}
		|\sigma_{a}(t_{k},\tilde{X}_{t_{k}}^{i,N,M}, \tilde{\mu}_{t_{k}}^{X,N,M})|
		| \big( \Delta W_{t_{k}}^{i} \big)_{a}| \, .
	\end{align*}
	The bound on $F$ then follows immediately from these results.
\end{proof}

Let us now show the first moment bound result. As is standard with implicit schemes we firstly do this up to a stopping time, hence we define
\begin{align}
\label{Eq:Lambda Stopping Time}
\lambda_{m}^{i} = \inf \{ k: ~ |\tilde{X}_{t_{k}}^{i,N,M}| >m\} .
\end{align}	
One should note that this stopping time does not actually bound $\tilde{X}$ at that point $i$, the best one can do is bound the previous point i.e.\ for $\lambda_{m}^{i} >0$, we have $|\tilde{X}_{\lambda_{m}^{i}-1}^{i,N,M}| \le m$.

\begin{lemma}
	\label{Lem:Moment Bound for Stopped X}
	Let Assumption \ref{Ass:Monotone Assumption}, \ref{Ass:Holder in Time} and  H\ref{Ass:Bounded Measure} (in Assumption \ref{Ass:Extra Implicit Bounds}) hold and fix $h^{*}< 1/\max(L_{b}, 2 \beta)$. Then for any $p \ge 2$ such that $\bE[|X_{0}|^{p}] = C(p) < \infty$, we also have,
	\begin{align*}
	\sup_{1 \le i \le N}
	\bE \big[ | \tilde{X}_{t_{k}}^{i,N,M}|^{p} \1_{\{k \le \lambda_{m}^{i}\}}
	\big]
	\le 
	C(p,m)
	\quad
	\forall k \le M
	~ ~ \text{and } 0 < h \le h^{*}.
	\end{align*} 
\end{lemma}
Using standard notation, $C(a)$ denotes a constant that can depend on variable $a$.

\begin{proof}
	As it turns out the function $F$ in \eqref{Eq:implicitToSolve} gives us a useful bound. From \eqref{Eq:Useful F Relation} we obtain,
	\begin{align*}
		|F(t_{k},\tilde{X}_{t_{k+1}}^{i,N,M}, \tilde{\mu}_{t_{k}}^{X,N,M})|^{p}
		\le
		2^{p-1}
		\big(
		|\tilde{X}_{t_{k}}^{i,N,M}|^{p}
		+
		|\sigma(t_{k},\tilde{X}_{t_{k}}^{i,N,M}, \tilde{\mu}_{t_{k}}^{X,N,M}) \Delta W_{t_{k}}^{i}|^{p}
		\big).
	\end{align*}
	Hence, multiplying with the indicator and taking expected values yields,
	\begin{align*}
		\bE[|F(t_{k},\tilde{X}_{t_{k+1}}^{i,N,M}, \tilde{\mu}_{t_{k}}^{X,N,M})|^{p} \1_{\{k+1 \le \lambda_{m}^{i}\}} ]
		\le
		C(p) \Big( m^{p} 
		+
		\bE \big[|\sigma(t_{k},\tilde{X}_{t_{k}}^{i,N,M}, \tilde{\mu}_{t_{k}}^{X,N,M})\Delta W_{t_{k}}^{i}|^{p} \1_{\{k+1 \le \lambda_{m}^{i}\}} \big] \Big) \, .
	\end{align*} 
	Then we estimate
	\begin{align*}
		& 
		\bE \big[|\sigma(t_{k},\tilde{X}_{t_{k}}^{i,N,M}, \tilde{\mu}_{t_{k}}^{X,N,M})\Delta W_{t_{k}}^{i}|^{p} \1_{\{k+1 \le \lambda_{m}^{i}\}} \big]
		\\
		&
		\qquad
		\le
		\sum_{a=1}^{l} \bE[ | \sigma_{a}(t_{k},\tilde{X}_{t_{k}}^{i,N,M}, \tilde{\mu}_{t_{k}}^{X,N,M})|^{2p} \1_{\{ k+1 \le \lambda_{m}^{i}\}}] 
		+
		\bE[|(\Delta W_{t_{k}}^{i})_{a}|^{2p}] \, .
	\end{align*}
	Using the bounds on each coefficient of $\sigma$, it is straightforward to observe,
	\begin{align*}
		| \sigma_{a}(t_{k},\tilde{X}_{t_{k}}^{i,N,M}, \tilde{\mu}_{t_{k}}^{X,N,M})|^{2p}
		\le
		C(p) \big( 1+ |\tilde{X}_{t_{k}}^{i,N,M}|^{2p} \big) \, .
	\end{align*}
	Using this bound we obtain,
	\begin{align*}
		\bE[|F(t_{k},\tilde{X}_{t_{k+1}}^{i,N,M}, \tilde{\mu}_{t_{k}}^{X,N,M})|^{p} \1_{\{k+1 \le \lambda_{m}^{i}\}} ]
		\le
		C(p,m) \, .
	\end{align*}
	Rewriting the quantity we wish to bound as
	\begin{align*}
		\bE \big[ | \tilde{X}_{t_{k}}^{i,N,M}|^{p} \1_{\{k \le \lambda_{m}^{i}\}}
		\big]
		&=
		\bE \big[ | \tilde{X}_{t_{k}}^{i,N,M}|^{p} \1_{\{k \le \lambda_{m}^{i}, k>0\}}
		\big]
		+
		\bE \big[ | \tilde{X}_{t_{0}}^{i,N,M}|^{p} \1_{\{k = 0, ~ \lambda_{m}^{i}=0\}}
		\big] 
		\le
		C(p,m) \, ,
	\end{align*}
	where the inequality follows from Estimate \eqref{est:FixEquation}, our bound on $F$, and the assumption that $X_{0} \in L^{p} (\bR^{d})$. Again, the corresponding bound is independent of the choice of $i$ and  hence the result holds for the supremum over $i$.
\end{proof}

Although the previous bound is useful, the presence of the stopping time is inconvenient. We therefore remove it and show the second moment is bounded. 
\begin{proposition}
	\label{Prop:Bounded Implicit Second Moment}
	Let Assumption \ref{Ass:Monotone Assumption}, \ref{Ass:Holder in Time} and H\ref{Ass:Bounded Measure} (in Assumption \ref{Ass:Extra Implicit Bounds}) hold and fix $h^{*}< 1/\max(L_{b}, 2 \beta)$. Further assume that $X_{0} \in L^{4}(\bR^{d})$. Then,
	\begin{align*}
	\sup_{1 \le i \le N} \sup_{h \le h^{*}} \sup_{0 \le k \le M} 
	\bE[ |\tilde{X}_{t_{k}}^{i,N,M}|^{2}]
	\le 
	C \, .
	\end{align*}
\end{proposition}

\begin{proof}
	Firstly let us take a nonnegative integer $K$, such that $Kh \le T$.
	Now let us consider \eqref{Eq:Recurrence Bound F}. One can note that this bound still holds where the $F$ terms are multiplied by $\1_{\{ \lambda_{m}^{i}>0 \}}$ (since both sides are nonnegative and the indicator is bounded above by one). Summing both sides from $k=1$ to $K \wedge \lambda_{m}^{i}$, noting that the $F$ terms cancel, we obtain,
	\begin{align*}
		& |F(t_{K \wedge \lambda_{m}^{i}},\tilde{X}_{t_{(K \wedge \lambda_{m}^{i})+1}}^{i,N,M}, \tilde{\mu}_{t_{K \wedge \lambda_{m}^{i}}}^{X,N,M})|^{2} 
		\1_{\{ \lambda_{m}^{i}>0 \}}
		\\
		&
		\le
		|F(t_{0},\tilde{X}_{t_{1}}^{i,N,M}, \tilde{\mu}_{t_{0}}^{X,N,M})|^{2}
		\1_{\{ \lambda_{m}^{i}>0 \}}
		+
		\sum_{k=1}^{K \wedge \lambda_{m}^{i}} \big(
		2h \alpha
		+ 
		2 h \beta |\tilde{X}_{t_{k}}^{i,N,M}|^{2}
		\1_{\{ \lambda_{m}^{i}>0 \}}
		\big)
		\\
		&
		~
		+
		\sum_{k=1}^{K \wedge \lambda_{m}^{i}}
		\Big( \sum_{a=1}^{l}|\sigma_{a}(t_{k},\tilde{X}_{t_{k}}^{i,N,M}, \tilde{\mu}_{t_{k}}^{X,N,M})| |\big(\Delta W_{t_{k}}^{i} \big)_{a}|\Big)^{2}
		\1_{\{ \lambda_{m}^{i}>0 \}}
		\\
		&
		~
		+
		\sum_{k=1}^{K \wedge \lambda_{m}^{i}}
		2
		\langle \tilde{X}_{t_{k}}^{i,N,M},
		\sigma(t_{k},\tilde{X}_{t_{k}}^{i,N,M}, \tilde{\mu}_{t_{k}}^{X,N,M}) \Delta W_{t_{k}}^{i} \rangle
		\1_{\{ \lambda_{m}^{i}>0 \}}
		\, ,	
	\end{align*}
	where we use the convention $\sum_{k=1}^{0} \cdot =0$.
	Although the stopping time is useful it is not ideal that it appears on the sum. However, for nonnegative terms it is straightforward to take the stopping time into the coefficients and the stochastic term can be rewritten as
	\begin{align*}
		& 
		\sum_{k=1}^{K \wedge \lambda_{m}^{i}}
		2
		\langle \tilde{X}_{t_{k}}^{i,N,M},
		\sigma(t_{k},\tilde{X}_{t_{k}}^{i,N,M}, \tilde{\mu}_{t_{k}}^{X,N,M}) \Delta W_{t_{k}}^{i} \rangle \1_{\{ \lambda_{m}^{i}>0 \}}
		=
		\sum_{k=1}^{K}
		2
		\langle \tilde{X}_{t_{k}}^{i,N,M},
		\sigma(t_{k},\tilde{X}_{t_{k}}^{i,N,M}, \tilde{\mu}_{t_{k}}^{X,N,M}) \Delta W_{t_{k}}^{i} \rangle \1_{\{k \le \lambda_{m}^{i}\}}.
	\end{align*}
Taking expectations and noting, by Lemma \ref{Lem:Moment Bound for Stopped X}, that $\tilde{X}_{t_{k}}^{i,N,M} \1_{\{k \le \lambda_{m}^{i}\}} \in L_{t_{k}}^{4}(\bR^{d})$ we conclude this term to be a martingale. We therefore obtain the following bound,
	\begin{align*}
		& \bE[|F(t_{K \wedge \lambda_{m}^{i}},\tilde{X}_{t_{(K \wedge \lambda_{m}^{i})+1}}^{i,N,M}, \tilde{\mu}_{t_{K \wedge \lambda_{m}^{i}}}^{X,N,M})|^{2} 
		\1_{\{ \lambda_{m}^{i}>0 \}}]
		\\
		&
		\le
		\bE \big[
		|F(t_{0},\tilde{X}_{t_{1}}^{i,N,M}, \tilde{\mu}_{t_{0}}^{X,N,M})|^{2} \big]
		+
		2 \alpha T
		+
		\sum_{k=1}^{K} 
		2 h \beta \bE \big[|\tilde{X}_{t_{k\wedge \lambda_{m}^{i}}}^{i,N,M}|^{2} \1_{\{ \lambda_{m}^{i}>0 \}} \big]
		\\
		&
		+
		\sum_{k=1}^{K}
		\bE \Big[
		\Big( \sum_{a=1}^{l}|\sigma_{a}(t_{k \wedge \lambda_{m}^{i}},\tilde{X}_{t_{k \wedge \lambda_{m}^{i}}}^{i,N,M}, \tilde{\mu}_{t_{k \wedge \lambda_{m}^{i}}}^{X,N,M})| |\big(\Delta W_{t_{k \wedge \lambda_{m}^{i}}}^{i} \big)_{a}|\Big)^{2} 
		\1_{\{ \lambda_{m}^{i}>0 \}}
		\Big]
		\, .
	\end{align*}
	The idea is to apply the discrete version of Gr\"onwall's inequality to this (see for example \cite[pg. 436]{MitrinovicEtAl2012} or \cite[Lemma 3.4]{MaoSzpruch2013}), which requires our bound to be in terms of $F$. Using arguments similar to previous ones  
	\begin{align*}
		& \bE \Big[
		\Big( \sum_{a=1}^{l}|\sigma_{a}(t_{k \wedge \lambda_{m}^{i}},\tilde{X}_{t_{k \wedge \lambda_{m}^{i}}}^{i,N,M}, \tilde{\mu}_{t_{k \wedge \lambda_{m}^{i}}}^{X,N,M})| |\big(\Delta W_{t_{k \wedge \lambda_{m}^{i}}}^{i} \big)_{a}|\Big)^{2} 
		\1_{\{ \lambda_{m}^{i}>0 \}}
		\Big]
		\\
		&
		\le
		C
		\sum_{a=1}^{l}\bE \Big[|\sigma_{a}(t_{k \wedge \lambda_{m}^{i}},\tilde{X}_{t_{k \wedge \lambda_{m}^{i}}}^{i,N,M}, \tilde{\mu}_{t_{k \wedge \lambda_{m}^{i}}}^{X,N,M})|^{2}
		|\big(\Delta W_{t_{k \wedge \lambda_{m}^{i}}}^{i} \big)_{a}|^{2}
		\1_{\{ \lambda_{m}^{i}>0 \}} \Big]
		\le
		C
		\sum_{a=1}^{l}h \big(1+ 
		\bE \big[|\tilde{X}_{t_{k \wedge \lambda_{m}^{i}}}^{i,N,M}|^{2} \1_{\{ \lambda_{m}^{i}>0 \}} \big]
		\big) \, ,
	\end{align*}
	where we have used independence of $\sigma(\cdot) \1_{\{ \lambda_{m}^{i}>0 \}}$ and $\Delta W$ along with the growth bounds on $\sigma$ to obtain the final inequality. Combining this with our previous bounds and appealing again to Lemma \ref{Lemma:Results for F} (to bound $\tilde{X}$ by $F$) we obtain,
	\begin{align*}
		& \bE[|F(t_{K \wedge \lambda_{m}^{i}},\tilde{X}_{t_{(K \wedge \lambda_{m}^{i})+1}}^{i,N,M}, \tilde{\mu}_{t_{K \wedge \lambda_{m}^{i}}}^{X,N,M})|^{2} 
		\1_{\{ \lambda_{m}^{i}>0 \}}]
		\\
		&
		\le
		\bE \big[
		|F(t_{0},\tilde{X}_{t_{1}}^{i,N,M}, \tilde{\mu}_{t_{0}}^{X,N,M})|^{2} \big]
		+
		C
		+
		\sum_{k=1}^{K} 
		C h \bE \big[|\tilde{X}_{t_{k\wedge \lambda_{m}^{i}}}^{i,N,M}|^{2}
		\1_{\{ \lambda_{m}^{i}>0 \}} \big]
		\\
		&
		\le
		\bE \big[
		|F(t_{0},\tilde{X}_{t_{1}}^{i,N,M}, \tilde{\mu}_{t_{0}}^{X,N,M})|^{2} \big]
		+
		C(1+\frac{h}{1-2h \beta})
		\\
		&
		\qquad
		+
		\sum_{k=1}^{K} 
		C \frac{h}{1-2h \beta} \bE \big[ 
		|F(t_{(k \wedge \lambda_{m}^{i})-1},\tilde{X}_{t_{k \wedge \lambda_{m}^{i}}}^{i,N,M}, \tilde{\mu}_{t_{(k \wedge \lambda_{m}^{i})-1}}^{X,N,M})|^{2} 
		\1_{\{ \lambda_{m}^{i}>0 \}}
		\big]
		\, .
	\end{align*}
	Applying a discrete version of the Gr\"onwall inequality and noting $\sum_{k=1}^{K}1 \le T/h$ yields
	\begin{align*}
		& \bE[|F(t_{K \wedge \lambda_{m}^{i}},\tilde{X}_{t_{(K \wedge \lambda_{m}^{i})+1}}^{i,N,M}, \tilde{\mu}_{t_{K \wedge \lambda_{m}^{i}}}^{X,N,M})|^{2} 
		\1_{\{ \lambda_{m}^{i}>0 \}}]
		\\
		&
		\le
		\Big(
		\bE \big[
		|F(t_{0},\tilde{X}_{t_{1}}^{i,N,M}, \tilde{\mu}_{t_{0}}^{X,N,M})|^{2} \big]
		+
		C\big( 1+\frac{h}{1-2h \beta} \big)
		\Big)
		\exp\Big( \frac{C}{1-2h \beta} \Big)
		\, .
	\end{align*}
	Recalling \eqref{Eq:Useful F Relation}, we can apply the same arguments as before to obtain the bound
	\begin{align*}
		\bE \big[
		|F(t_{0},\tilde{X}_{t_{1}}^{i,N,M}, \tilde{\mu}_{t_{0}}^{X,N,M})|^{2} \big]
		\le
		C(1+ (1+h) \bE[|\tilde{X}_{t_{0}}^{i,N,M}|^{2}] ) \, .
	\end{align*}
	Noting that our bound for $F$ is now independent of $m$, we can use Fatou's lemma to take the limit and obtain (for $K  \ge 1$),
	\begin{align*}
		\bE[|F(t_{K},\tilde{X}_{t_{K+1}}^{i,N,M}, \tilde{\mu}_{t_{K}}^{X,N,M})|^{2} 
		]
		\le
		C \Big(
		1+ (1+h) \bE[|\tilde{X}_{t_{0}}^{i,N,M}|^{2}]
		+
		\frac{h}{1-2h \beta}
		\Big)
		\exp\Big( \frac{C}{1-2h \beta} \Big)
		\, .
	\end{align*}
	Again by Lemma \ref{Lemma:Results for F}, the LHS of the latter inequality bounds $\tilde{X}_{t_{K+1}}^{i,N,M}$ (with some constant), hence we obtain a bound for $\tilde{X}_{t_{k}}^{i,N,M}$ for $k \ge 2$. By assumption $\tilde{X}_{t_{0}}^{i,N,M}$ has second moment therefore we need to obtain a bound for $\tilde{X}_{t_{1}}^{i,N,M}$. This is not difficult to obtain using again that we can bound $\tilde{X}$ as follows,
	\begin{align*}
		\bE \big[
		|\tilde{X}_{t_{1}}^{i,N,M}|^{2} \big]
		\le
		\big(1-2h \beta \big)^{-1} \Big( 
		2 h \alpha + \bE \big[
		|F(t_{0},\tilde{X}_{t_{1}}^{i,N,M}, \tilde{\mu}_{t_{0}}^{X,N,M})|^{2} \big] \Big) ,
	\end{align*}
	then we can apply the same bound on $F$ as above.
	
	In order to complete the proof, we also need to show that this bound exists for all $i$ and $0 < h \le h^{*}$. One can see immediately that all bounds decrease as $h$ decreases, hence the supremum value is to set $h=h^{*}$, which is also finite since $h^{*}<1/(2 \beta)$. The supremum over $i$ follows from the fact that all bounds are independent of $i$.
\end{proof}

Now that we have established a bound on the second moment, we look to show convergence of this scheme to the true particle system solution. As always with discrete schemes it is beneficial to introduce their continuous counterpart. As it turns out doing it naively for implicit schemes leads to measurability problems, hence one introduces the so-called forward backward scheme
\begin{equation*}
\hat{X}^{i,N,M}_{t_{k+1}} 
= 
\hat{X}^{i,N,M}_{t_{k}} 
+ 
b \left( t_{k-1 \vee 0}, \tilde{X}_{t_{k}}^{i,N,M}, \tilde{\mu}_{t_{k-1 \vee 0}}^{X,N,M} \right) h 
+ 
\sigma \left( t_{k}, \tilde{X}_{t_{k}}^{i,N,M}, \tilde{\mu}_{t_{k}}^{i,N,M} \right) \Delta W_{t_{k}}^{i}, 
\end{equation*}
where $\hat{X}_{0}^{i,N,M} = X_{0}^{i}$ and  $\vee$ denotes the maximum. The scheme's continuous time version is
\begin{align}
\label{Eq:ForwardBackward Scheme}
\hat{X}^{i,N,M}_{t} = 
X_{0}^{i}
&+
\int_{0}^{t} b \left( (\kappa(s)-h) \vee 0, \tilde{X}_{\kappa(s)}^{i,N,M}, \tilde{\mu}_{(\kappa(s) - h) \vee 0}^{X,N,M} \right) \dd s
+
\int_{0}^{t} \sigma \left( \kappa(s), \tilde{X}_{\kappa(s)}^{i,N,M}, \tilde{\mu}_{\kappa(s)}^{i,N,M} \right) \dd W_{s}^{i}.
\end{align}
The first result we present is that the discrete and continuous versions stay close to one another, up to the stopping time \eqref{Eq:Lambda Stopping Time}.
\begin{lemma}
	\label{Lem:Discrete and Continuous Implicit}
	Let Assumption \ref{Ass:Monotone Assumption}, \ref{Ass:Holder in Time} and H\ref{Ass:Bounded Measure} (in Assumption \ref{Ass:Extra Implicit Bounds}) hold and fix $h^{*}< 1/\max(L_{b}, 2 \beta)$. Further assume $X_{0} \in L^{4(q+1)}( \bR^{d})$. Then for $1 \le p \le 4$ the following holds for $0 < h \le h^{*}$,
	\begin{align*}
	\sup_{1 \le i \le N}
	\sup_{0 \le k \le M}
	\bE \big[ |\hat{X}_{t_{k}}^{i,N,M} - \tilde{X}_{t_{k}}^{i,N,M}|^{p} 
	\1_{\{ k \le \lambda_{m}^{i}\}} \big]
	\le 
	C(m,p)h^{p} \, .
	\end{align*}
	Moreover, we also have the following relation between $\hat{X}$ and $F$ for all $1 \le k \le M$,
	\begin{align}
	\label{Eq:ForwardBackward F Relation}
	|\hat{X}_{t_{k}}^{i,N,M}|^{2} 
	\ge
	\frac{1}{2} |F(t_{k-1},\tilde{X}_{t_{k}}^{i,N,M}, \tilde{\mu}_{t_{k-1}}^{X,N,M})|^{2}
	-
	|b(t_{0},\tilde{X}_{t_{0}}^{i,N,M}, \tilde{\mu}_{t_{0}}^{X,N,M}) h|^{2} \, .
	\end{align}
\end{lemma}

\begin{proof}
	To show the first part we start by noting the following useful relation between \eqref{Eq:implicitScheme} and \eqref{Eq:ForwardBackward Scheme}, namely for $1 \le k \le M$,
	\begin{align*}
		\hat{X}_{t_{k}}^{i,N,M}
		-
		\tilde{X}_{t_{k}}^{i,N,M}
		=
		\big(
		b(t_{0},\tilde{X}_{t_{0}}^{i,N,M}, \tilde{\mu}_{t_{0}}^{X,N,M})
		-
		b(t_{k-1},\tilde{X}_{t_{k}}^{i,N,M}, \tilde{\mu}_{t_{k-1}}^{X,N,M}) 
		\big) h\, .
	\end{align*}
	Noting that one can bound
	\begin{align*}
		| b(t_{0},\tilde{X}_{t_{0}}^{i,N,M}, \tilde{\mu}_{t_{0}}^{X,N,M})
		-
		b(t_{k-1},\tilde{X}_{t_{k}}^{i,N,M}, \tilde{\mu}_{t_{k-1}}^{X,N,M}) |
		\le
		C \big(1 + |t_{k}|^{1/2} 
		+ |\tilde{X}_{t_{0}}^{i,N,M}|^{q+1} 
		+ |\tilde{X}_{t_{k}}^{i,N,M}|^{q+1}
		\big) \, ,
	\end{align*}
	where we have used the polynomial growth, Hölder-continuity on the coefficient $b$ and in particular Assumption H\ref{Ass:Bounded Measure}. Hence,
	\begin{align*}
		& \bE \big[ |\hat{X}_{t_{k}}^{i,N,M} - \tilde{X}_{t_{k}}^{i,N,M}|^{p} 
		\1_{\{ k \le \lambda_{m}^{i}\}} \big]
		\\
		&
		\quad
		\le
		C(p)h^{p} \big(1 + |t_{k}|^{p/2} 
		+ \bE \big[ |\tilde{X}_{t_{0}}^{i,N,M}|^{p(q+1)} 
		\1_{\{ k \le \lambda_{m}^{i}\}} \big]
		+
		\bE \big[ |\tilde{X}_{t_{k}}^{i,N,M}|^{p(q+1)}
		\1_{\{ k \le \lambda_{m}^{i}\}} \big]
		\big)\, .
	\end{align*}
	One observes that the terms on the RHS are bounded by $C(p,m)$ for $p \le 4$ since $X_{0} \in L^{4(q+1)}( \bR^{d})$ and Lemma \ref{Lem:Moment Bound for Stopped X}. This  completes the first part of the proof.

	For the second part, recall from the relation between \eqref{Eq:implicitScheme} and \eqref{Eq:ForwardBackward Scheme}, one has,
	\begin{align*}
		\hat{X}_{t_{k}}^{i,N,M}
		&
		=
		b(t_{0},\tilde{X}_{t_{0}}^{i,N,M}, \tilde{\mu}_{t_{0}}^{X,N,M}) h
		+
		\tilde{X}_{t_{k}}^{i,N,M}
		-
		b(t_{k-1},\tilde{X}_{t_{k}}^{i,N,M}, \tilde{\mu}_{t_{k-1}}^{X,N,M}) h
		\\
		&
		=
		b(t_{0},\tilde{X}_{t_{0}}^{i,N,M}, \tilde{\mu}_{t_{0}}^{X,N,M}) h
		+
		F(t_{k-1},\tilde{X}_{t_{k}}^{i,N,M}, \tilde{\mu}_{t_{k-1}}^{X,N,M})
		\, .
	\end{align*}
	Using the reverse triangle inequality we obtain,
	\begin{align*}
	|\hat{X}_{t_{k}}^{i,N,M}|^{2}
	\ge
	-|b(t_{0},\tilde{X}_{t_{0}}^{i,N,M}, \tilde{\mu}_{t_{0}}^{X,N,M}) h|
	+
	|F(t_{k-1},\tilde{X}_{t_{k}}^{i,N,M}, \tilde{\mu}_{t_{k-1}}^{X,N,M})|
	\, .	
	\end{align*}
	The result follows from squaring both sides and applying the generalisation of Young's inequality, namely,
	\begin{align*}
	& |b(t_{0},\tilde{X}_{t_{0}}^{i,N,M}, \tilde{\mu}_{t_{0}}^{X,N,M}) h|
	|F(t_{k-1},\tilde{X}_{t_{k}}^{i,N,M}, \tilde{\mu}_{t_{k-1}}^{X,N,M})|
	\\
	&
	\le
	|b(t_{0},\tilde{X}_{t_{0}}^{i,N,M}, \tilde{\mu}_{t_{0}}^{X,N,M}) h|^{2}
	+\frac{1}{4}
	|F(t_{k-1},\tilde{X}_{t_{k}}^{i,N,M}, \tilde{\mu}_{t_{k-1}}^{X,N,M})|^{2}
	\, .	
	\end{align*}
\end{proof}

The next result we wish to present is that both schemes do not blow up in finite time, for this we define a new stopping time,
\begin{align*}
\eta_{m}^{i}
:=
\inf \big\{
t \ge 0: |\hat{X}_{t}^{i,N,M}| \ge m \, ,
~ ~
\text{or}
~ ~
|\tilde{X}_{\kappa(t)}^{i,N,M}| >m
\big\} \, .
\end{align*}
Note in particular that $\eta_{m}^{i}$ is smaller than or equal to $\lambda_{m}^{i}$ in \eqref{Eq:Lambda Stopping Time}.
\begin{lemma}
	\label{Lem:Bounded Stopping Time}
	Let Assumption \ref{Ass:Monotone Assumption}, \ref{Ass:Holder in Time} and  H\ref{Ass:Bounded Measure} (in Assumption \ref{Ass:Extra Implicit Bounds}) hold, fix $h^{*}< 1/\max(L_{b}, 2 \beta)$ and assume $X_{0} \in L^{4(q+1)}( \bR^{d})$. Then, for any $\epsilon >0$, there exists a $m^{*}$ such that, for any $m \ge m^{*}$ we can find a $h_{0}^{*}(m)$ (note the dependence on $m$) so that
	\begin{align*}
	\sup_{1 \le i \le N}
	\bP( \eta_{m}^{i} < T) \le \epsilon \, , 
	~ ~
	\text{for any } 0< h \le h_{0}^{*}(m).
	\end{align*}
\end{lemma}

\begin{proof}
	Note due to the initial condition being random we must be careful with how we set $m$, we shall come back to this later. Let us start by applying It\^{o} to the stopped version of \eqref{Eq:ForwardBackward Scheme}, 
	\begin{align*}
	|\hat{X}^{i,N,M}_{T \wedge \eta_{m}^{i}}|^{2}
	= 
	&
	|X_{0}^{i}|^{2}
	+
	\int_{0}^{T \wedge \eta_{m}^{i}} 2 \langle \hat{X}^{i,N,M}_{s}, 
	b \left( (\kappa(s)-h) \vee 0, \tilde{X}_{\kappa(s)}^{i,N,M}, \tilde{\mu}_{(\kappa(s) - h) \vee 0}^{X,N,M} \right)
	\rangle
	\\
	&
	+
	\sum_{a=1}^{l}
	|\sigma_{a} \left( \kappa(s), \tilde{X}_{\kappa(s)}^{i,N,M}, \tilde{\mu}_{\kappa(s)}^{i,N,M} \right)|^{2} \dd s
+
	\int_{0}^{T \wedge \eta_{m}^{i}} 
	2 \langle \hat{X}^{i,N,M}_{s},
	\sigma \left( \kappa(s), \tilde{X}_{\kappa(s)}^{i,N,M}, \tilde{\mu}_{\kappa(s)}^{i,N,M} \right) \dd W_{s}^{i}
	\rangle.
	\end{align*}
	We now look to bound the various integrands. Firstly one can observe
	\begin{align*}
	& \langle \hat{X}^{i,N,M}_{t}, 
	b \big( (\kappa(s)-h) \vee 0, \tilde{X}_{\kappa(s)}^{i,N,M}, \tilde{\mu}_{(\kappa(s) - h) \vee 0}^{X,N,M} \big)
	\rangle
	+
	\sum_{a=1}^{l}
	|\sigma_{a} \big( \kappa(s), \tilde{X}_{\kappa(s)}^{i,N,M}, \tilde{\mu}_{\kappa(s)}^{i,N,M} \big)|^{2}
	\\
	&
	=
	\langle \hat{X}^{i,N,M}_{t} - \tilde{X}_{\kappa(s)}^{i,N,M}, 
	b \big( (\kappa(s)-h) \vee 0, \tilde{X}_{\kappa(s)}^{i,N,M}, \tilde{\mu}_{(\kappa(s) - h) \vee 0}^{X,N,M} \big)
	\rangle 
	\\
	&
	~
	+
	\langle \tilde{X}_{\kappa(s)}^{i,N,M}, 
	b \big( (\kappa(s)-h) \vee 0, \tilde{X}_{\kappa(s)}^{i,N,M}, \tilde{\mu}_{(\kappa(s) - h) \vee 0}^{X,N,M} \big)
	\rangle
	+
	\sum_{a=1}^{l}
	|\sigma_{a} \big( \kappa(s), \tilde{X}_{\kappa(s)}^{i,N,M}, \tilde{\mu}_{\kappa(s)}^{i,N,M} \big)|^{2}
	\\
	&
	\le
	C |\hat{X}^{i,N,M}_{t} - \tilde{X}_{\kappa(s)}^{i,N,M}|(1 + |\tilde{X}_{\kappa(s)}^{i,N,M}|^{q+1})
	+
	2 \alpha + \beta |\tilde{X}_{\kappa(s)}^{i,N,M}|^{2} \, ,
	\end{align*}
	where we used Cauchy-Schwarz, polynomial growth bound, Hölder-continuity and monotone growth to obtain the final inequality. 
	
	Taking expectations and noting that due to the stopping time the stochastic integral is square integrable and hence a martingale, we obtain,
	\begin{align*}
	& \bE[ |\hat{X}^{i,N,M}_{T \wedge \eta_{m}^{i}}|^{2} ]
\le
	\bE[ |X_{0}^{i}|^{2}]
	+
	\bE \Big[ 
	\int_{0}^{T \wedge \eta_{m}^{i}}
	\hspace*{-0.075cm}
	C |\hat{X}^{i,N,M}_{s} - \tilde{X}_{\kappa(s)}^{i,N,M}|(1 + |\tilde{X}_{\kappa(s)}^{i,N,M}|^{q+1})
	+ 
	2 \alpha + \beta |\tilde{X}_{\kappa(s)}^{i,N,M}|^{2}
	\dd s
	\Big] .
	\end{align*}
	To proceed we note the following,
	$|\tilde{X}_{\kappa(s)}^{i,N,M}|^{2}
	\le
	2(
	|\tilde{X}_{\kappa(s)}^{i,N,M}- \hat{X}_{s}^{i,N,M}|^{2}
	+
	|\hat{X}_{s}^{i,N,M}|^{2} )$
 	and also that
	\begin{align*}
	\int_{0}^{T \wedge \eta_{m}^{i}}
	|\hat{X}_{s}^{i,N,M} - \tilde{X}_{\kappa(s)}^{i,N,M}|^{2}
	\dd s
	\le
	C(m)
	\int_{0}^{T \wedge \eta_{m}^{i}}
	|\hat{X}_{s}^{i,N,M} - \tilde{X}_{\kappa(s)}^{i,N,M}|
	\dd s \, ,
	\end{align*}
	where we used the fact that the stopping time ensures $\tilde{X}$ and $\hat{X}$ are $\le m$ for $s < \eta_{m}^{i}$ and $s=\eta_{m}^{i}$ has measure zero. The same reasoning also implies,
	\begin{align*}
	\int_{0}^{T \wedge \eta_{m}^{i}}
	C |\hat{X}^{i,N,M}_{s} - \tilde{X}_{\kappa(s)}^{i,N,M}|&(1 + |\tilde{X}_{\kappa(s)}^{i,N,M}|^{q+1})
	\dd s
	\le
	C(m)
	\int_{0}^{T \wedge \eta_{m}^{i}}
	|\hat{X}_{s}^{i,N,M} - \tilde{X}_{\kappa(s)}^{i,N,M}|
	\dd s .
	\end{align*}
	Hence the following result holds,
	\begin{align*}
	&\bE[ |\hat{X}^{i,N,M}_{T \wedge \eta_{m}^{i}}|^{2} ]
	\le
	\bE[ |X_{0}^{i}|^{2}]
	+
	C
	\bE \Big[ 
	\int_{0}^{T \wedge \eta_{m}^{i}}
	C(m) |\hat{X}^{i,N,M}_{s} - \tilde{X}_{\kappa(s)}^{i,N,M}|
	+
	1 + \beta |\hat{X}_{s}^{i,N,M}|^{2}
	\dd s
	\Big] \, .
	\end{align*}
	The next step is of course to take the expectation inside the integral. Let us start by noting the difference term can be bounded as
	\begin{align*}
	\bE \Big[
	\int_{0}^{T \wedge \eta_{m}^{i}}
	|\hat{X}_{s}^{i,N,M}  - \tilde{X}_{\kappa(s)}^{i,N,M}|
	\dd s
	\Big]
	&
	\le
	\bE \Big[
	\int_{0}^{T \wedge \eta_{m}^{i}}
	|\hat{X}_{s}^{i,N,M} - \hat{X}_{\kappa(s)}^{i,N,M}|
	\dd s
	+
	\int_{0}^{T \wedge \eta_{m}^{i}}
	|\hat{X}_{\kappa(s)}^{i,N,M} - \tilde{X}_{\kappa(s)}^{i,N,M}|
	\dd s
	\Big]
	\\
	&
	\le
	\bE \Big[
	h
	\int_{0}^{T \wedge \eta_{m}^{i}}
	|b \left( (\kappa(s)-h) \vee 0, \tilde{X}_{\kappa(s)}^{i,N,M}, \tilde{\mu}_{(\kappa(s) - h) \vee 0}^{X,N,M} \right)|
	\dd s
	\Big]
	\\
	&
	~
	+
	\bE \Big[
	\int_{0}^{T \wedge \eta_{m}^{i}}
	|\sigma \left( \kappa(s), \tilde{X}_{\kappa(s)}^{i,N,M}, \tilde{\mu}_{\kappa(s)}^{i,N,M} \right) (W_{s}^{i}- W_{\kappa(s)}^{i}) |
	\dd s
	\Big]
	+
	C(m)h,		
	\end{align*}
	where we have used Lemma \ref{Lem:Discrete and Continuous Implicit} for the final inequality. For the other terms, one can note due to the growth assumptions on $b$ and Lemma \ref{Lem:Moment Bound for Stopped X}, that
	\begin{align*}
	\bE \Big[
	h
	\int_{0}^{T \wedge \eta_{m}^{i}}
	|b \left( (\kappa(s)-h) \vee 0, \tilde{X}_{\kappa(s)}^{i,N,M}, \tilde{\mu}_{(\kappa(s) - h) \vee 0}^{X,N,M} \right)|
	\dd s
	\Big]
	\le
	 C(m) \,h .
	\end{align*}
	The term involving $\sigma$ is more complex. However, we can bound it as follows:
	\begin{align*}
	& \bE \Big[
	\int_{0}^{T \wedge \eta_{m}^{i}}
	|\sigma \left( \kappa(s), \tilde{X}_{\kappa(s)}^{i,N,M}, \tilde{\mu}_{\kappa(s)}^{i,N,M} \right) (W_{s}^{i}- W_{\kappa(s)}^{i}) |
	\dd s
	\Big]
	\\
	&
	\le
	C
	\int_{0}^{T}
	\sum_{a=1}^{l}
	\bE \Big[
	|\sigma_{a} \left( \kappa(s), \tilde{X}_{\kappa(s)}^{i,N,M}, \tilde{\mu}_{\kappa(s)}^{i,N,M} \right)| | (W_{s}^{i}- W_{\kappa(s)}^{i})_{a} |
	\1_{ \{\kappa(s) \le t_{\lambda_{m}^{i}} \}}
	\Big]
	\dd s
	\\
	&
	\le
	C
	\int_{0}^{T}
	\sum_{a=1}^{l}
	h^{1/2}
	(1 + \bE[ | \tilde{X}_{\kappa(s) \wedge t_{\lambda_{m}^{i}}}|^{2}])
	\dd s
	\le C(m) h^{1/2}.	
	\end{align*}
	Further, since $|\hat{X}_{s}^{i,N,M}|\ge 0$, we obtain
	\begin{align*}
	\bE \Big[ 
	\int_{0}^{T \wedge \eta_{m}^{i}}
	|\hat{X}_{s}^{i,N,M}|^{2}
	\dd s
	\Big]
	\le
	\int_{0}^{T}
	\bE \big[
	|\hat{X}_{s  \wedge \eta_{m}^{i}}^{i,N,M}|^{2}
	\big]
	\dd s .
	\end{align*}
	Hence,
	\begin{align}
	\bE[ |\hat{X}^{i,N,M}_{T \wedge \eta_{m}^{i}}|^{2} ]
	\le
	&
	\bE[ |X_{0}^{i}|^{2}]
	+
	C(m) h^{1/2}
	+
	C 
	\int_{0}^{T}
	1 + \beta \bE \big[
	|\hat{X}_{s  \wedge \eta_{m}^{i}}^{i,N,M}|^{2}
	\big]
	\dd s
	\notag
	\\
	\le
	&
	\big( \bE[ |X_{0}^{i}|^{2}]
	+
	C 
	+
	C(m)h^{1/2} \big)
	\exp(C \beta T)
	\label{Eq:Bound on Stopped X hat}
	\, ,
	\end{align}
	where the final inequality follows from Gr\"onwall.
	
	In order to obtain an upper bound on the probability of the stopping time occurring we look to obtain a lower bound for \eqref{Eq:ForwardBackward Scheme} at the stopping time. For the moment let us take $\vert X_{0}^{i} \vert < m$, hence $\eta_{m}^{i}>0$. There are now two possible ways the stopping time can be reached: if $\hat{X}$ hits the boundary first, then we have $|\hat{X}_{\eta_{m}^{i}}^{i,N,M}| =m$ and if $\tilde{X}$ hits the boundary first we have  $|\tilde{X}_{\eta_{m}^{i}}^{i,N,M}| > m$.
	
	In the case that $\hat{X}$ hits the boundary first, the lower bound is obvious, namely $|\hat{X}_{\eta_{m}^{i}}^{i,N,M}| =m$. For the second case it is less obvious. Recalling \eqref{Eq:ForwardBackward F Relation} and \eqref{est:FixEquation} we obtain  lower bound
	\begin{align*}
	|\hat{X}_{t_{k}}^{i,N,M}|^{2} 
	\ge
	\frac{1}{2} \big((1-2h \beta)|\tilde{X}_{t_{k}}^{i,N,M}|^{2} -2h \alpha \big)
	-
	|b(t_{0},\tilde{X}_{t_{0}}^{i,N,M}, \tilde{\mu}_{t_{0}}^{X,N,M}) h|^{2} \, ,
	\end{align*}
	where again we are taking $k \ge 1$ here, but this is not a problem since we are assuming for the moment $\vert X_{0}^{i} \vert <m$. Observing that this lower bound holds independently of which process triggers the stopping condition we have on $\big\{ \vert \tilde{X}_{\eta_{m}^{i}}^{i,N,M} \vert > m \big\}$ that
	\begin{align*}
	m^{2}
	&\ge
	|\hat{X}_{\eta_{m}^{i}}^{i,N,M}|^{2} 
	\1_{\{ |X_{0}^{i}|<m\}}
	\ge
	\frac{1}{2} \big((1-2h \beta)m^{2} -2h \alpha \big)
	\1_{\{ |X_{0}^{i}|<m\}}
	-
	|b(t_{0},\tilde{X}_{t_{0}}^{i,N,M}, \tilde{\mu}_{t_{0}}^{X,N,M}) h|^{2}
	\1_{\{ |X_{0}^{i}|<m\}}
	\, .
	\end{align*}
	Thus, for constants $C_1,C_2 >0$,
	\begin{align*}
	|\hat{X}_{\eta_{m}^{i}}^{i,N,M}|^{2} 
	\1_{\{ |X_{0}^{i}|<m\}}
	\ge
	(C_{1} m^{2} -C_{2}h)
	\1_{\{ |X_{0}^{i}|<m\}}
	-
	C(m)h^{2}
	\1_{\{ |X_{0}^{i}|<m\}}
	\, ,
	\end{align*}
	where  $|b(t_{0},\tilde{X}_{t_{0}}^{i,N,M}, \tilde{\mu}_{t_{0}}^{X,N,M})|\1_{\{ |X_{0}^{i}|<m\}} \le C(m) \1_{\{ |X_{0}^{i}|<m \} }$ via the growth condition on $b$. Let us now combine these results to obtain an upper bound for the probability of the stopping time. Notice that
	\begin{align*}
	\bE[ |\hat{X}_{T \wedge \eta_{m}^{i}}^{i,N,M}|^{2}]
	& \ge
	\bE[ |X_{0}^{i}|^{2} \1_{\{|X_{0}^{i}| \ge m \}}]
	+
	\bE[ |\hat{X}_{\eta_{m}^{i}}^{i,N,M}|^{2}
	\1_{ \{ \vert X^i_0 \vert < m \} }
	\1_{\{0< \eta_{m}^{i} <T \}}]
	\\
	&
	\ge
	\bP(\eta_{m}^{i}=0)
	+
	\left(
	(C_{1} m^{2} -C_{2}h)
	-
	C(m)h^{2}
	\right)
	\bP(\{ |X_{0}^{i}|<m\} \cap \{0< \eta_{m}^{i} <T \} )
	\, .
	\end{align*}
	Leaving the second term for the moment, and noting that $X_{0}^{i}$ is uniformly integrable, then for any $\epsilon >0$ there exists an $m^*>0$ such that for all $m \geq m^*$
	\begin{align*}
	\bP( \eta_{m}^{i}=0)
	\le m \bP(|X_{0}^{i}| \ge m)
	\le \bE[ |X_{0}^{i}| \1_{\{ |X_{0}^{i}| \ge m \}}]
	\le \frac{\epsilon}{3} \, .
	\end{align*}
	It is also useful to note that 
	$$\bP(\{ |X_{0}^{i}|<m\} \cap \{0< \eta_{m}^{i} <T \} ) = \bP(\{0< \eta_{m}^{i} <T \} ).$$ 
	From our previous analysis it is clear that for $m$ large enough and some contant $C(m)$, by using \eqref{Eq:Bound on Stopped X hat}, the probability can be bounded by
	\begin{align*}
	\bP(0< \eta_{m}^{i} <T)
	&\le
	\frac{ 
		\bE[ |\hat{X}_{T \wedge \eta_{m}^{i}}^{i,N,M}|^{2}] }
	{ (C_{1} m^{2} -C_{2}h
		-
		C(m)h^{2}) }
	\le
	\frac{ 
		\big( \bE[ |X_{0}^{i}|^{2}]
		+
		C 
		+
		C(m)h^{1/2} \big)
		\exp(C \beta T) }
	{ C_{1} m^{2} -C_{2}h
		-
		C(m)h^{2} }
	\, .
	\end{align*}
	Now the goal is to bound this by $2 \epsilon/3$. We already have taken $m$ sufficiently large to obtain the last inequality. Now consider for any given $m$ a factor $h_{01}^{*}(m)$ such that \linebreak
	$C_{2} h_{01}^{*}(m) + C(m)h_{01}^{*}(m)^{2} \le 1$. It is clear for $0 < h < h_{01}^{*}(m)$ the same bound holds. Then for the same $\epsilon$ as before choose $m$ large enough such that,
	\begin{align*}
	\frac{ 
		\big( \bE[ |X_{0}^{i}|^{2}]
		+
		C  \big)
		\exp(C \beta T) }
	{ C_{1} m^{2} -1 }
	\le
	\frac{\epsilon}{3}
	\, .
	\end{align*}
	Redefine $m^{*}$ as the corresponding maximum of this $m$ and $m^{*}$. Now for any $m \ge m^{*}$, define $h_{02}^{*}(m)$ such that,
	\begin{align*}
	\frac{ 
		C(m)(h_{02}^{*})^{1/2} 
		\exp(C \beta T) }
	{ C_{1} m^{2} -1 }
	\le
	\frac{\epsilon}{3} \, .
	\end{align*}
	Again for $0 < h < h_{02}^{*}(m)$ the above inequality holds. Hence for any $m \ge m^{*}$ and any \linebreak $0 < h < \min ( h_{01}^{*}(m), h_{02}^{*}(m))$, we have,
	$
	\bP( \eta_{m}^{i}<T)
	\le
	\bP( \eta_{m}^{i}=0)
	+
	\bP( 0<\eta_{m}^{i}<T)
	\le
	\epsilon.
	$
\end{proof}

We now look towards showing our strong convergence result, firstly by showing convergence between \eqref{Eq:ForwardBackward Scheme} and \eqref{Eq:MV-SDE Propagation} and then \eqref{Eq:implicitScheme} and \eqref{Eq:MV-SDE Propagation}. From this point onwards we require H\ref{Ass:Sigma Indepedent of Measure} (in Assumption~\ref{Ass:Extra Implicit Bounds}).
%
%
%
Recalling the stopping time in Proposition \ref{Prop:Particle System Bounds}, we now define $\theta_{m}^{i} := \tau_{m}^{i} \wedge \eta_{m}^{i}$ and have the following convergence result.
\begin{lemma}
	\label{Lem:Strong ForwardBackward Particle Conv}
	Let Assumption \ref{Ass:Monotone Assumption}, \ref{Ass:Holder in Time}, the full Assumption  \ref{Ass:Extra Implicit Bounds} hold, fix $h^{*}< 1/\max(L_{b}, 2 \beta)$ and assume $X_{0} \in L^{4(q+1)}( \bR^{d})$. Then, for all $h \in ( 0, h^*)$,
	\begin{align*}
	\sup_{1 \le i \le N}
	\bE[ \sup_{0 \le t \le T} |\hat{X}_{t \wedge \theta_{m}^{i}}^{i,N,M} - X_{t \wedge \theta_{m}^{i}}^{i,N}|^{2}]
	\le
	C(m)h 
	+
	C\bE\big[\1_{\{ T > \theta_{m}^{i} \}} \big]^{1/2}\, .
	\end{align*}
\end{lemma}

\begin{proof}
	For ease of presentation we denote by $\overline{\kappa}(s):= (\kappa(s)-h) \vee 0$. 
	As is standard we start by applying It\^{o} to the difference to obtain
	\begin{align*}
		|X_{t \wedge \theta_{m}^{i}}^{i,N} - \hat{X}_{t \wedge \theta_{m}^{i}}^{i,N,M} |^{2}
		&
		=
		\int_{0}^{t \wedge \theta_{m}^{i}}
		2 \langle
		X_{s}^{i,N} - \hat{X}_{s}^{i,N,M} 
		,
		b(s, X_{s}^{i,N},\mu_{s}^{X,N}) - b(\overline{\kappa}(s), \tilde{X}_{\kappa(s)}^{i,N,M}, \tilde{\mu}_{\overline{\kappa}(s)}^{X,N,M}) \rangle
		\\
		&
		\qquad\qquad + \sum_{a=1}^{l}
		|\sigma_a(s, X_{s}^{i,N},\mu_{s}^{X,N}) - \sigma_a(\kappa(s), \tilde{X}_{\kappa(s)}^{i,N,M}, \tilde{\mu}_{\kappa(s)}^{X,N,M}) |^{2} \dd s
		\\
		&
		\quad + 
		\int_{0}^{t \wedge \theta_{m}^{i}}
		2 \langle
		X_{s}^{i,N} - 
		\hat{X}_{s}^{i,N,M}
		,
		\big(\sigma(s, X_{s}^{i,N},\mu_{s}^{X,N}) - \sigma(\kappa(s), \tilde{X}_{\kappa(s)}^{i,N,M}, \tilde{\mu}_{\kappa(s)}^{X,N,M}) \big) \dd W_{s}^{i} \rangle \, .
	\end{align*}
	By writing out the drift term we have that
	\begin{align*}
		& \langle
		X_{s}^{i,N} - \hat{X}_{s}^{i,N,M} 
		,
		b(s, X_{s}^{i,N},\mu_{s}^{X,N}) - b(\overline{\kappa}(s), \tilde{X}_{\kappa(s)}^{i,N,M}, \tilde{\mu}_{\overline{\kappa}(s)}^{X,N,M}) \rangle
		\\
		&
		=
		\langle
		X_{s}^{i,N} - \hat{X}_{s}^{i,N,M} 
		,
		b(s, X_{s}^{i,N},\mu_{s}^{X,N}) - b(s, \hat{X}_{s}^{i,N,M}, \mu_{s}^{X,N}) \rangle
		\\
		&
		~	+
		\langle
		X_{s}^{i,N} - \hat{X}_{s}^{i,N,M} 
		,
		b(s, \hat{X}_{s}^{i,N,M}, \mu_{s}^{X,N}) - b(\overline{\kappa}(s), \hat{X}_{s}^{i,N,M}, \mu_{s}^{X,N}) \rangle
		\\
		&
		~	+
		\langle
		X_{s}^{i,N} - \hat{X}_{s}^{i,N,M} 
		,
		b(\overline{\kappa}(s), \hat{X}_{s}^{i,N,M}, \mu_{s}^{X,N})
		-
		b(\overline{\kappa}(s), \hat{X}_{\kappa(s)}^{i,N,M}, \mu_{s}^{X,N}) \rangle
		\\
		&
		~	+
		\langle
		X_{s}^{i,N} - \hat{X}_{s}^{i,N,M} 
		,
		b(\overline{\kappa}(s), \hat{X}_{\kappa(s)}^{i,N,M}, \mu_{s}^{X,N})
		-
		b(\overline{\kappa}(s), \tilde{X}_{\kappa(s)}^{i,N,M}, \mu_{s}^{X,N}) \rangle
		\\
		&
		~	+
		\langle
		X_{s}^{i,N} - \hat{X}_{s}^{i,N,M} 
		,
		b(\overline{\kappa}(s), \tilde{X}_{\kappa(s)}^{i,N,M}, \mu_{s}^{X,N})
		-
		b(\overline{\kappa}(s), \tilde{X}_{\kappa(s)}^{i,N,M}, \tilde{\mu}_{\overline{\kappa}(s)}^{X,N,M}) \rangle
		\\
		&
		\le
		C \Big( 
		|X_{s}^{i,N} - \hat{X}_{s}^{i,N,M}|^{2} 
		+ h
		+ \big(C \wedge CW^{(1)}(\mu_{s}^{X,N}, \tilde{\mu}_{\overline{\kappa}(s)}^{X,N,M}) \big)^{2}
		\\
		& \qquad 
		+
		(1+ |\hat{X}_{s}^{i,N,M}|^{2q} + |\hat{X}_{\kappa(s)}^{i,N,M}|^{2q})
		|\hat{X}_{s}^{i,N,M} - \hat{X}_{\kappa(s)}^{i,N,M}|^{2}
		\\
		&
		\qquad \qquad 
		+
		(1+ |\hat{X}_{\kappa(s)}^{i,N,M}|^{2q} + |\tilde{X}_{\kappa(s)}^{i,N,M}|^{2q})
		|\hat{X}_{\kappa(s)}^{i,N,M} - \tilde{X}_{\kappa(s)}^{i,N,M}|^{2} \Big)\, ,
	\end{align*}
	where we have used the growth bounds on $b$ along with several applications of Cauchy-Schwarz and Young's inequality. In particular we have used the fact that $b$ is both globally and $W^{(1)}$ bounded in measure to obtain the $ C \wedge CW^{(1)}(\mu_{s}^{X,N}, \tilde{\mu}_{\overline{\kappa}(s)}^{X,N,M})$ bound.
	Using similar arguments to earlier proofs and to the drift term above, we get the following bound for the diffusion
	\begin{align*}
		&	|\sigma_{a}(s, X_{s}^{i,N}, \mu_{s}^{X,N}) - \sigma_{a}(\kappa(s), \tilde{X}_{\kappa(s)}^{i,N,M}, \tilde{\mu}_{\kappa(s)}^{X,N,M}) |
		\\
		&\qquad 
		\le
		C \big( h^{1/2} + |X_{s}^{i,N}- \hat{X}_{s}^{i,N,M}| + |\hat{X}_{s}^{i,N,M} - \hat{X}_{\kappa(s)}^{i,N,M}| + |\hat{X}_{\kappa(s)}^{i,N,M} - \tilde{X}_{\kappa(s)}^{i,N,M}| + 1 \wedge W^{(1)}(\mu_{s}^{X,N}, \tilde{\mu}_{\kappa(s)}^{X,N,M}) \big) .
	\end{align*}
	Ultimately we need to take supremum and expected values, hence we wish to bound
	\begin{align*}
		\bE \Big[
		\sup_{0 \le r \le t \wedge \theta_{m}^{i}}
		\int_{0}^{r}
		2 \langle
		X_{s}^{i,N} - 
		\hat{X}_{s}^{i,N,M}
		,
		\big(\sigma(s, X_{s}^{i,N}, \mu_{s}^{X,N}) - \sigma(\kappa(s), \tilde{X}_{\kappa(s)}^{i,N,M}, \tilde{\mu}_{\kappa(s)}^{X,N,M}) \big) \dd W_{s}^{i} \rangle
		\Big] \, .
	\end{align*}
	We use Burkholder Davis Gundy inequality, however care is needed since the terminal time is a stopping time. It turns out the usual upper bound still holds (see for example \cite[pg. 226]{Protter2005}), hence we obtain by using Young's inequality
	\begin{align*}
		&
		\bE \Big[
		\sup_{0 \le r \le t \wedge \theta_{m}^{i}}
		\int_{0}^{r}
		2 \langle
		X_{s}^{i,N} - 
		\hat{X}_{s}^{i,N,M}
		,
		\big(\sigma(s, X_{s}^{i,N}, \mu_{s}^{X,N}) - \sigma(\kappa(s), \tilde{X}_{\kappa(s)}^{i,N,M}, \tilde{\mu}_{\kappa(s)}^{X,N,M}) \big) \dd W_{s}^{i} \rangle
		\Big] 
		\\
		&
		\le
		C
		\bE \Big[
		\Big(
		\int_{0}^{t \wedge \theta_{m}^{i}}
		|
		X_{s}^{i,N} - 
		\hat{X}_{s}^{i,N,M}|^{2}
		\sum_{a=1}^{l} |\sigma_a(s, X_{s}^{i,N}, \mu_{s}^{X,N}) - \sigma_a(\kappa(s), \tilde{X}_{\kappa(s)}^{i,N,M}, \tilde{\mu}_{\kappa(s)}^{X,N,M})|^{2}  \dd s
		\Big)^{1/2}
		\Big] 
		\\
		&
		\le
		\frac{1}{2}
		\bE \Big[
		\sup_{0 \le s \le t \wedge \theta_{m}^{i}}
		|
		X_{s}^{i,N} - 
		\hat{X}_{s}^{i,N,M}|^{2}
		\Big] 
		+
		C
		\bE \Big[
		\int_{0}^{t \wedge \theta_{m}^{i}}
		\sum_{a=1}^{l} |\sigma_a(s, X_{s}^{i,N}, \mu_{s}^{X,N}) - \sigma_a(\kappa(s), \tilde{X}_{\kappa(s)}^{i,N,M}, \tilde{\mu}_{\kappa(s)}^{X,N,M})|^{2}  \dd s
		\Big] 
		.
	\end{align*}
	Taking supremum over time and expectations of our original difference and using these bounds we obtain the inequality
	\begin{align*}
		& 
		\frac{1}{2}
		\bE \left[ \sup_{0 \le t \le T} |X_{t \wedge \theta_{m}^{i}}^{i,N} - \hat{X}_{t \wedge \theta_{m}^{i}}^{i,N,M} |^{2}
		\right]
		\\
		&
		\le
		\bE \Big[
		\int_{0}^{T \wedge \theta_{m}^{i}}
		C \Big( 
		|X_{s}^{i,N} - \hat{X}_{s}^{i,N,M}|^{2} 
		+ 
		\big(1 \wedge W^{(1)}(\mu_{s}^{X,N},\tilde{\mu}_{\overline{\kappa}(s)}^{X,N,M}) \big)^{2}
		\\
		& 
		+
		(1+ |\hat{X}_{s}^{i,N,M}|^{2q} + |\hat{X}_{\kappa(s)}^{i,N,M}|^{2q})
		|\hat{X}_{s}^{i,N,M} - \hat{X}_{\kappa(s)}^{i,N,M}|^{2}
		+ h
		+
		(1+ |\hat{X}_{\kappa(s)}^{i,N,M}|^{2q} + |\tilde{X}_{\kappa(s)}^{i,N,M}|^{2q})
		|\hat{X}_{\kappa(s)}^{i,N,M} - \tilde{X}_{\kappa(s)}^{i,N,M}|^{2}
		\Big)
		\\
		& 
		+ C \sum_{a=1}^{l} \left(
		h + |X_{s}^{i,N}- \hat{X}_{s}^{i,N,M}|^{2} + |\hat{X}_{s}^{i,N,M} - \hat{X}_{\kappa(s)}^{i,N,M}|^{2} + |\hat{X}_{\kappa(s)}^{i,N,M} - \tilde{X}_{\kappa(s)}^{i,N,M}|^{2} 
		+
		\big(1 \wedge W^{(1)}(\mu_{s}^{X,N},\tilde{\mu}_{\kappa(s)}^{X,N,M}) \big)^{2}		
		\right) \dd s
		\Big]
		\, .
	\end{align*}

	Let us now concentrate on the measure terms
	$ 1 \wedge W^{(1)}(\mu_{s}^{X,N},\tilde{\mu}_{\overline{\kappa}(s)}^{X,N,M})
	$
	and
	$ 1 \wedge W^{(1)}(\mu_{s}^{X,N},\tilde{\mu}_{\kappa(s)}^{X,N,M}).
	$
	 The goal in the end is to use a Gr\"onwall type inequality. Hence, we want to obtain terms of a similar form. The standard argument in this case is to remove the average sum of other particles using the fact that they are identically distributed, unfortunately the presence of the stopping time breaks this argument and forces us to argue a different way. We start by noting the following bound
	\begin{align*}
	W^{(1)}(\mu_{s}^{X,N}, \tilde{\mu}_{\overline{\kappa}(s)}^{X,N,M})
	\le &
	\frac{1}{N}\sum_{j=1}^{N}|X_{s}^{j,N}-\tilde{X}_{\overline{\kappa}(s)}^{j,N,M}| \1_{\{s \le \theta_{m}^{j}\}}
	+
	\frac{1}{N}\sum_{j=1}^{N}|X_{s}^{j,N}-\tilde{X}_{\overline{\kappa}(s)}^{j,N,M}| \1_{\{s > \theta_{m}^{j}\}}.
	\end{align*}
	By using the fact that for $a$, $b$, $c>0$, $\min(a,b+c) \le \min(a,b)+\min(a,c)$ and $\min(a,b) \le \sqrt{a}\sqrt{b}$ alongside H\"older inequality for sums, we obtain
	\begin{align*}
	1 \wedge W^{(1)}(\mu_{s}^{X,N}, \tilde{\mu}_{\overline{\kappa}(s)}^{X,N,M})
	\le &
	\sqrt{\frac{1}{N}\sum_{j=1}^{N}|X_{s}^{j,N}-\tilde{X}_{\overline{\kappa}(s)}^{j,N,M}|^{2} \1_{\{s \le \theta_{m}^{j}\}}}
	+
	\sqrt{\frac{1}{N}\sum_{j=1}^{N}|X_{s}^{j,N}-\tilde{X}_{\overline{\kappa}(s)}^{j,N,M}| \1_{\{s > \theta_{m}^{j}\}}} \, .
	\end{align*}
	Let us further define $\hat{\mu}_{s}^{X,N,M} := \frac{1}{N}\sum_{j=1}^{N} \delta_{\hat{X}_{s}^{j,N,M}}$. Then using the triangle inequality we get
	\begin{align*}
	\frac{1}{N}\sum_{j=1}^{N}|X_{s}^{j,N}-\tilde{X}_{\overline{\kappa}(s)}^{j,N,M}|^{2} \1_{\{s \le \theta_{m}^{j}\}}
	\le &
	\frac{C}{N}\sum_{j=1}^{N}|X_{s}^{j,N}-\hat{X}_{s}^{j,N,M}|^{2} \1_{\{s \le \theta_{m}^{j}\}} 
	+
	\frac{C}{N}\sum_{j=1}^{N}|\hat{X}_{s}^{j,N,M}-\hat{X}_{\kappa(s)}^{j,N,M}|^{2} \1_{\{s \le \theta_{m}^{j}\}}
	\\
	&
	+
	\frac{C}{N}\sum_{j=1}^{N}|\hat{X}_{\kappa(s)}^{j,N,M}-\hat{X}_{\overline{\kappa}(s)}^{j,N,M}|^{2} \1_{\{s \le \theta_{m}^{j}\}}
	+
	\frac{C}{N}\sum_{j=1}^{N}|\hat{X}_{\overline{\kappa}(s)}^{j,N,M}-\tilde{X}_{\overline{\kappa}(s)}^{j,N,M}|^{2} \1_{\{s \le \theta_{m}^{j}\}}
	\, .
	\end{align*}
	Hence, we can bound the measure terms by
	\begin{align*}
	& \bE \bigg[
	\int_{0}^{T \wedge \theta_{m}^{i}}
	\big(1 \wedge W^{(1)}(\mu_{s}^{X,N},\tilde{\mu}_{\overline{\kappa}(s)}^{X,N,M}) \big)^{2}
	 \dd s
	\bigg]
	\\
	 & \le
	\bE \bigg[
	\int_{0}^{T}
	\big(1 \wedge W^{(1)}(\mu_{s}^{X,N},\tilde{\mu}_{\overline{\kappa}(s)}^{X,N,M}) \big)^{2}
	\dd s
	\bigg]
	\\
	& \le
	\bE \bigg[
	\int_{0}^{T}
	\frac{C}{N} 
	\bigg( \sum_{j=1}^{N}|X_{s}^{j,N}-\hat{X}_{s}^{j,N,M}|^{2} \1_{\{s \le \theta_{m}^{j}\}} 
	+
	\sum_{j=1}^{N}|\hat{X}_{s}^{j,N,M}-\hat{X}_{\kappa(s)}^{j,N,M}|^{2} \1_{\{s \le \theta_{m}^{j}\}}
	\\
	&
	~ ~ ~
	+
	\sum_{j=1}^{N}|\hat{X}_{\kappa(s)}^{j,N,M}-\hat{X}_{\overline{\kappa}(s)}^{j,N,M}|^{2} \1_{\{s \le \theta_{m}^{j}\}}
	+
	\sum_{j=1}^{N}|\hat{X}_{\overline{\kappa}(s)}^{j,N,M}-\tilde{X}_{\overline{\kappa}(s)}^{j,N,M}|^{2} \1_{\{s \le \theta_{m}^{j}\}}
	+
	\sum_{j=1}^{N}|X_{s}^{j,N}-\tilde{X}_{\overline{\kappa}(s)}^{j,N,M}| \1_{\{s > \theta_{m}^{j}\}}
	\dd s
	\bigg) \bigg]
	\, 
	\end{align*}
	and likewise also
	\begin{align*}
	& \bE \bigg[
	\int_{0}^{T \wedge \theta_{m}^{i}}
	\big(1 \wedge W^{(1)}(\mu_{s}^{X,N},\tilde{\mu}_{\kappa(s)}^{X,N,M}) \big)^{2}
	 \dd s
	\bigg]
	\\
	&
	\qquad \le
	\bE \bigg[
	\int_{0}^{T}
	\frac{C}{N} 
	\bigg( \sum_{j=1}^{N}|X_{s}^{j,N}-\hat{X}_{s}^{j,N,M}|^{2} \1_{\{s \le \theta_{m}^{j}\}} 
	+
	\sum_{j=1}^{N}|\hat{X}_{s}^{j,N,M}-\hat{X}_{\kappa(s)}^{j,N,M}|^{2} \1_{\{s \le \theta_{m}^{j}\}}
	\\
	&
	\qquad \qquad
	+
	\sum_{j=1}^{N}|\hat{X}_{\kappa(s)}^{j,N,M}-\tilde{X}_{\kappa(s)}^{j,N,M}|^{2} \1_{\{s \le \theta_{m}^{j}\}}
	+
	\sum_{j=1}^{N}|X_{s}^{j,N}-\tilde{X}_{\kappa(s)}^{j,N,M}| \1_{\{s > \theta_{m}^{j}\}}
	\dd s
	\bigg) \bigg]
	\, .
	\end{align*}
	
	Therefore, taking the expectation inside the integral and supremum over the particle index; noting particles are identically distributed, we obtain
	\begin{align*}
		\hspace*{2cm} & \hspace*{-2cm} 
		\sup_{1 \le i \le N} \bE \left[ \sup_{0 \le t \le T \wedge \theta_{m}^{i}} |X_{t \wedge \theta_{m}^{i}}^{i,N} - \hat{X}_{t \wedge \theta_{m}^{i}}^{i,N,M} |^{2}
		\right]
		\\
		\le &
		C \Big( hT
		+
		\int_{0}^{T }
		\sup_{1 \le i \le N} \bE \Big[
		\sup_{0 \le r \le s}
		|X_{r\wedge \theta_{m}^{i}}^{i,N} - \hat{X}_{r\wedge \theta_{m}^{i}}^{i,N,M}|^{2} 
		\Big] 
		+ 
		\sup_{1 \le i \le N}
		\bE \Big[
		|\hat{X}_{\kappa(s)}^{i,N,M}-\hat{X}_{\overline{\kappa}(s)}^{i,N,M}|^{2} \1_{\{s \le \theta_{m}^{i}\}}
		\Big]
		\\
		&
		\quad
		+ 
		\sup_{1 \le i \le N}
		\bE \Big[
		|\hat{X}_{\overline{\kappa}(s)}^{i,N,M}-\tilde{X}_{\overline{\kappa}(s)}^{i,N,M}|^{2} \1_{\{s \le \theta_{m}^{i}\}}
		\Big]
		\\
		&
		\quad
		+
		\sup_{1\le i \le N} \bE \Big[
		|X_{s}^{i,N}-\tilde{X}_{\overline{\kappa}(s)}^{i,N,M}| \1_{\{s > \theta_{m}^{i}\}}
		\Big]
		+
		\sup_{1\le i \le N} \bE \Big[
		|X_{s}^{i,N}-\tilde{X}_{\kappa(s)}^{i,N,M}| \1_{\{s > \theta_{m}^{i}\}}
		\Big]
		\\
		&
		\quad 
		+
		\sup_{1 \le i \le N}
		\bE \Big[
		(1+ |\hat{X}_{s}^{i,N,M}|^{2q} + |\hat{X}_{\kappa(s)}^{i,N,M}|^{2q})
		|\hat{X}_{s}^{i,N,M} - \hat{X}_{\kappa(s)}^{i,N,M}|^{2}
		\1_{\{s \le \theta_{m}^{i}\}}
		\Big]
		\\
		&
		\quad 
		+
		\sup_{1 \le i \le N}
		\bE \Big[
		(1+ |\hat{X}_{\kappa(s)}^{i,N,M}|^{2q} + |\tilde{X}_{\kappa(s)}^{i,N,M}|^{2q})
		|\hat{X}_{\kappa(s)}^{i,N,M} - \tilde{X}_{\kappa(s)}^{i,N,M}|^{2}
		\1_{\{s \le \theta_{m}^{i}\}}
		\Big]
		\dd s
		\Big)
		\, ,
	\end{align*}
	where we have further used that if $Y_{\cdot}\ge 0$, then $Y_{\cdot}\1_{\{\cdot \le t\}} \le Y_{\{\cdot \wedge t\}}$.
	Noting $\1_{\{\cdot\}}= \1_{\{ \cdot\}}^{2}$, we obtain via Cauchy-Schwarz inequality
	\begin{align*}
		\bE \Big[
		(1+ |\hat{X}_{s}^{i,N,M}|^{2q} + |\hat{X}_{\kappa(s)}^{i,N,M}|^{2q})
		|\hat{X}_{s}^{i,N,M} &- \hat{X}_{\kappa(s)}^{i,N,M}|^{2}
		\1_{\{s \le \theta_{m}^{i}\}}
		\Big]
		\le
		C(m)
		\bE \Big[
		|\hat{X}_{s}^{i,N,M} - \hat{X}_{\kappa(s)}^{i,N,M}|^{4}
		\1_{\{s \le \theta_{m}^{i}\}}
		\Big]^{1/2}.
	\end{align*}
	Noting that
	\begin{align*}
		&|\hat{X}_{s}^{i,N,M} - \hat{X}_{\kappa(s)}^{i,N,M}|
		\le
		|b \left( \overline{\kappa}(s), \tilde{X}_{\kappa(s)}^{i,N,M}, \tilde{\mu}_{(\kappa(s) - h) \vee 0}^{X,N,M} \right)|h
		+
		|\sigma \left( \kappa(s), \tilde{X}_{\kappa(s)}^{i,N,M}, \tilde{\mu}_{\kappa(s)}^{X,N,M}) \right) (W_{s}^{i} - W_{\kappa(s)}^{i}) | \, ,
	\end{align*}
	which implies
	\begin{align*}
		\hspace*{0.5cm} & \hspace*{-0.5cm}
		\bE \Big[
		|\hat{X}_{s}^{i,N,M} - \hat{X}_{\kappa(s)}^{i,N,M}|^{4}
		\1_{\{s \le \theta_{m}^{i}\}}
		\Big]
		\\
		&
		\le
		Ch^{4} \bE \Big [(1 + | \tilde{X}_{\kappa(s)}^{i,N,M}|^{4(q+1)}) \1_{\{s \le \theta_{m}^{i}\}}  \Big] 
		+
		C \bE \Big [(1 + | \tilde{X}_{\kappa(s)}^{i,N,M}|^{4}) \1_{\{s \le \theta_{m}^{i}\}}  \Big] 
		\bE \Big [(W_{s}^{i} - W_{\kappa(s)}^{i})^{4}  \Big] 
		\le
		C(m)h^{2} \, ,
	\end{align*}
	where we used Lemma \ref{Lem:Moment Bound for Stopped X} to obtain the final inequality (note by assumption $X_{0} \in L^{4(q+1)}(\bR^{d})$).
	Arguing in the exact same fashion along with Lemma \ref{Lem:Discrete and Continuous Implicit} also yields
	\begin{align*}
		\bE \Big[
		(1+ |\hat{X}_{\kappa(s)}^{i,N,M}|^{2q} + |\tilde{X}_{\kappa(s)}^{i,N,M}|^{2q})
		|\hat{X}_{\kappa(s)}^{i,N,M} - \tilde{X}_{\kappa(s)}^{i,N,M}|^{2}
		\1_{\{s \le \theta_{m}^{i}\}}
		\Big]
		\le
		C(m)h \, .
	\end{align*}
	The remaining terms can be bounded using the same arguments as above. Substituting these bounds then implies
	\begin{align*}
		& \sup_{1 \le i \le N} \bE \Big[ \sup_{0 \le t \le T} |X_{t \wedge \theta_{m}^{i}}^{i,N} - \hat{X}_{t \wedge \theta_{m}^{i}}^{i,N,M} |^{2}
		\Big]
		\\ 
		&
		~ ~ \le
		C(m) h
		+
		C
		\int_{0}^{T }
		\sup_{1\le i \le N} \bE \Big[
		|X_{s}^{i,N}-\tilde{X}_{\overline{\kappa}(s)}^{i,N,M}| \1_{\{s > \theta_{m}^{i}\}}
		\Big]
		\dd s
		+
		C
		\int_{0}^{T }
		\sup_{1\le i \le N} \bE \Big[
		|X_{s}^{i,N}-\tilde{X}_{\kappa(s)}^{i,N,M}| \1_{\{s > \theta_{m}^{i}\}}
		\Big]
		\dd s
		\\
		& 
		\qquad
		+
		C
		\int_{0}^{T }
		\sup_{1 \le i \le N}\bE \Big[
		\sup_{0 \le r \le s}
		|X_{r\wedge \theta_{m}^{i}}^{i,N} - \hat{X}_{r\wedge \theta_{m}^{i}}^{i,N,M}|^{2} 
		\Big]
		\dd s
		\, .
	\end{align*}
	Hence, by Gr\"onwall's inequality we obtain,
	\begin{align*}
		&		
		\sup_{1 \le i \le N}\bE \Big[ \sup_{0 \le t \le T} |X_{t \wedge \theta_{m}^{i}}^{i,N} 
		- \hat{X}_{t \wedge \theta_{m}^{i}}^{i,N,M} |^{2}
		\Big]
		\\
		&
		\le
		C(m)h
		+
		C
		\int_{0}^{T }
		\sup_{1\le i \le N} \bE \Big[
		|X_{s}^{i,N}-\tilde{X}_{\overline{\kappa}(s)}^{i,N,M}| \1_{\{s > \theta_{m}^{i}\}}
		\Big]
		+
		\sup_{1\le i \le N} \bE \Big[
		|X_{s}^{i,N}-\tilde{X}_{\kappa(s)}^{i,N,M}| \1_{\{s > \theta_{m}^{i}\}}
		\Big]
		\dd s
		\, .
	\end{align*}
	We then complete the proof by applying Cauchy-Schwarz to the expectations in the integrand along with Propositions \ref{Prop:Particle System Bounds} and \ref{Prop:Bounded Implicit Second Moment}.
\end{proof}

We now can prove our main implicit scheme result.

\begin{proof}[Proof of Proposition \ref{Prop:Strong Implicit Scheme Convergence}] Recall that $s\in[1,2)$. Define the error term as 
	$E_{r}(T)^{i} = X_{T}^{i,N} - \tilde{X}_{T}^{i,N,M}$ 
	and also let us note a more general version of Young's inequality
	\begin{align*}
	x^{s} y \le \frac{\delta s}{2} x^{2} + \frac{2-s}{2 \delta^{s/(2-s)}} y^{2/(2-s)} \, , \quad
	\forall ~ x, 
	~ y, ~ \delta >0 \, .
	\end{align*}
	Hence,
	\begin{align*}
		\bE[ |X_{T}^{i,N} - \tilde{X}_{T}^{i,N,M}|^{s} ]
	&\le
	2^{s-1} \big( \bE[ |X_{T}^{i,N} - \hat{X}_{T}^{i,N,M}|^{s} \1_{\{ \tau_{m}^{i} > T, ~ \eta_{m}^{i} >T \}}] 
	+ \bE[ |\hat{X}_{T}^{i,N,M} - \tilde{X}_{T}^{i,N,M}|^{s} \1_{\{ \tau_{m}^{i} > T, ~ \eta_{m}^{i} >T \}}] \big)
	\\
	& \qquad
	+
	\frac{\delta s}{2} \bE[ |E_{r}(T)^{i}|^{2}] 
	+ \frac{2-s}{2 \delta^{s/(2-s)}} \bE[ \1_{\{ \tau_{m}^{i} \le T ~ \text{or} ~ \eta_{m}^{i} \le T \}}] \, .
	\end{align*}
	From Lemma \ref{Lem:Discrete and Continuous Implicit} we obtain
	\begin{align*}
	\bE[ |\hat{X}_{T}^{i,N,M} - \tilde{X}_{T}^{i,N,M}|^{s} \1_{\{ \tau_{m}^{i} > T, ~ \eta_{m}^{i} >T \}}]
	\le
	C(m,s) h^{s} \, .
	\end{align*}
	Also let us note,
	$
	\bE[ | E_{r}(T)^{i}|^{2}]
	\le
	2 \bE[ |X_{T}^{i,N}|^{2} + | \tilde{X}_{T}^{i,N,M}|^{2}]
	\le 
	2C \, , 
	$
	where we have used Propositions \ref{Prop:Particle System Bounds} and \ref{Prop:Bounded Implicit Second Moment}. Hence for any $\epsilon >0$, we can choose $ \delta$ such that,
	\begin{align*}
	\frac{\delta s}{2} \bE[ |E_{r}(T)^{i}|^{2}] 
	\le \frac{\epsilon}{3} \, .
	\end{align*} 
	By subadditivity of measures
	$\bE[ \1_{\{ \tau_{m}^{i} \le T ~ \text{or} ~ \eta_{m}^{i} \le T \}}]
	\le
	\bP( \tau_{m}^{i} \le T) + \bP(\eta_{m}^{i} \le T)$.
	By Proposition~\ref{Prop:Particle System Bounds}, there exists $m^{*}$ (dependent on $\delta$), such that for $m \ge m^{*}$
	\begin{align*}
	\frac{2-s}{2 \delta^{s/(2-s)}} \bP( \tau_{m}^{i} \le T )
	\le
	\frac{\epsilon}{3} \, .
	\end{align*}
	Then, noting by Lemma \ref{Lem:Strong ForwardBackward Particle Conv}
	\begin{align*}
	\bE[
	| \hat{X}_{T}^{i,N,M} 
	- X_{T}^{i,N} |^{s} \1_{\{ \tau_{m}^{i} > T, ~ \eta_{m}^{i} >T \}} ]
	\le
	\bE \Big[
	\sup_{0 \le t \le T} | \hat{X}_{t \wedge \theta_{m}^{i}}^{i,N,M} 
	- X_{t \wedge \theta_{m}^{i}}^{i,N} |^{2} \Big]^{s/2}
	\le
	C(m)h^{s/2} + C\bE\big[\1_{\{ T > \theta_{m}^{i} \}} \big]^{s/4} \, .
	\end{align*}
	Again by subadditivity of measures we can bound
	\begin{align*}
	\bE\big[\1_{\{ T > \theta_{m}^{i} \}} \big]^{s/4}
	\le
	\bP( \tau_{m}^{i} \le T)^{s/4} + \bP(\eta_{m}^{i} \le T)^{s/4} \, .
	\end{align*}
	By the same argument as before we can define a new $m^{*}$, greater than or equal to the previous such that $C\bP( \tau_{m}^{i} \le T)^{s/4}$ is sufficiently small. By Lemma \ref{Lem:Bounded Stopping Time}, by taking $h$ small enough for any $\tilde{\epsilon} > 0$, 
	$\bP( \eta_{m}^{i} <T) \le \tilde{\epsilon}$,
	and by extension, there exists an $h$ small enough such that $\bP( \eta_{m}^{i} <T)^{s/4} \le \tilde{\epsilon}$.
	Hence, for any $m$, we can take $h$ small enough such that
	\begin{align*}
	& 2^{s-1} \big( \bE[ |X_{T}^{i,N} - \hat{X}_{T}^{i,N,M}|^{s} \1_{\{ \tau_{m}^{i} > T, ~ \eta_{m}^{i} >T \}}] 
	\\
	&\qquad 
	+ \bE[ |\hat{X}_{T}^{i,N,M} - \tilde{X}_{T}^{i,N,M}|^{s} \1_{\{ \tau_{m}^{i} > T, ~ \eta_{m}^{i} >T \}}] \big)
	+
	\frac{2-s}{2 \delta^{s/(2-s)}} \bP( \eta_{m}^{i} \le T )
	\le
	\frac{\epsilon}{3} \,
	\end{align*}
	and hence $ \bE[ |X_{T}^{i,N} - \tilde{X}_{T}^{i,N,M}|^{s} ] \leq \epsilon$. Since $\epsilon > 0$ was arbitrary, we have the result.
\end{proof}

\paragraph{Funding}
G. dos Reis acknowledges support from the \emph{Funda{\c c}$\tilde{\text{a}}$o para a Ci$\hat{e}$ncia e a Tecnologia} (Portuguese Foundation for Science and Technology) through the project UIDB/00297/2020 (Centro de Matem\'atica e Aplica\c c$\tilde{\text{o}}$es CMA/FCT/UNL).

S.~Engelhardt was supported by German Exchange DAAD (Nr. 57369588) and acknowledges the hospitality of the University of Edinburgh. 

G. Smith was supported by The Maxwell Institute Graduate School in Analysis and its Applications, a Centre for Doctoral Training funded by the UK Engineering and Physical Sciences Research Council (grant EP/L016508/01), the Scottish Funding Council, the University of Edinburgh and Heriot-Watt University.


\bibliographystyle{IMANUM-BIB}
\bibliography{Taming-MVSDEs}  

\begin{thebibliography}{}

\bibitem[Adams {\em et~al.}(2020)Adams, dos Reis, Ravaille, Salkeld, \&
  Tugaut]{adams2020large}
{\sc Adams, D., dos Reis, G., Ravaille, R., Salkeld, W. \& Tugaut, J.} (2020)
\newblock Large deviations and exit-times for reflected {McK}ean-{V}lasov
  equations with self-stabilizing terms and superlinear drifts.
\newblock {\em arXiv preprint arXiv:2005.10057\/}.

\bibitem[Baladron {\em et~al.}(2012)Baladron, Fasoli, Faugeras, \&
  Touboul]{BaladronFasoliFaugerasEtAl2012}
{\sc Baladron, J., Fasoli, D., Faugeras, O. \& Touboul, J.} (2012)
\newblock Mean-field description and propagation of chaos in networks of
  {H}odgkin-{H}uxley and {F}itz{H}ugh-{N}agumo neurons.
\newblock {\em The Journal of Mathematical Neuroscience\/}, {\bf 2}, 10.

\bibitem[Bolley {\em et~al.}(2011)Bolley, Canizo, \& Carrillo]{Bolley2011}
{\sc Bolley, F., Canizo, J.~A. \& Carrillo, J.~A.} (2011)
\newblock Stochastic mean-field limit: non-{L}ipschitz forces and swarming.
\newblock {\em Mathematical Models and Methods in Applied Sciences\/}, {\bf
  21}, 2179--2210.

\bibitem[Bossy {\em et~al.}(2015)Bossy, Faugeras, \& Talay]{BossyEtAl2015}
{\sc Bossy, M., Faugeras, O. \& Talay, D.} (2015)
\newblock Clarification and complement to ``mean-field description and
  propagation of chaos in networks of {H}odgkin--{H}uxley and
  {F}itz{H}ugh--{N}agumo neurons''.
\newblock {\em The Journal of Mathematical Neuroscience (JMN)\/}, {\bf 5}, 19.

\bibitem[Bossy \& Talay(1997)Bossy \& Talay]{BossyTalay1997}
{\sc Bossy, M. \& Talay, D.} (1997)
\newblock A stochastic particle method for the {M}c{K}ean-{V}lasov and the
  {B}urgers equation.
\newblock {\em Mathematics of Computation of the American Mathematical
  Society\/}, {\bf 66}, 157--192.

\bibitem[Budhiraja \& Fan(2017)Budhiraja \& Fan]{BudhirajaFan2017}
{\sc Budhiraja, A. \& Fan, W.-T.} (2017)
\newblock Uniform in time interacting particle approximations for nonlinear
  equations of {P}atlak-{K}eller-{S}egel type.
\newblock {\em Electron. J. Probab.}, {\bf 22}, Paper No. 8, 37.

\bibitem[Carmona(2016)Carmona]{Carmona2016}
{\sc Carmona, R.} (2016)
\newblock {\em Lectures on {BSDE}s, stochastic control, and stochastic
  differential games with financial applications\/}. Financial Mathematics,
  vol.~1.
\newblock Society for Industrial and Applied Mathematics (SIAM), Philadelphia,
  PA, pp. ix+265.

\bibitem[Carmona \& Delarue(2018)Carmona \& Delarue]{CarmonaDelarue2017book1}
{\sc Carmona, R. \& Delarue, F.} (2018)
\newblock {\em Probabilistic theory of mean field games with applications.
  {I}\/}. Probability Theory and Stochastic Modelling,  vol.~83.
\newblock Springer, Cham, pp. xxv+713.
\newblock Mean field FBSDEs, control, and games.

\bibitem[Chassagneux {\em et~al.}(2016)Chassagneux, Jacquier, \&
  Mihaylov]{ChassagneuxEtAl2016}
{\sc Chassagneux, J.-F., Jacquier, A. \& Mihaylov, I.} (2016)
\newblock An explicit {E}uler scheme with strong rate of convergence for
  financial {SDE}s with non-{L}ipschitz coefficients.
\newblock {\em SIAM Journal on Financial Mathematics\/}, {\bf 7}, 993--1021.

\bibitem[Delarue {\em et~al.}(2019)Delarue, Lacker, \& Ramanan]{Delarue2019}
{\sc Delarue, F., Lacker, D. \& Ramanan, K.} (2019)
\newblock From the master equation to mean field game limit theory: a central
  limit theorem.
\newblock {\em Electron. J. Probab.}, {\bf 24}, Paper No. 51, 54.

\bibitem[dos Reis {\em et~al.}(2018)dos Reis, Smith, \&
  Tankov]{dosReisEtAl2018}
{\sc dos Reis, G., Smith, G. \& Tankov, P.} (2018)
\newblock Importance sampling for {M}c{K}ean-{V}lasov {SDE}s.
\newblock {\em arXiv:1803.09320\/}.

\bibitem[dos Reis {\em et~al.}(2019)dos Reis, Salkeld, \&
  Tugaut]{dosReisSalkeldTugaut2017}
{\sc dos Reis, G., Salkeld, W. \& Tugaut, J.} (2019)
\newblock Freidlin-{W}entzell {LDP} in path space for {M}c{K}ean-{V}lasov
  equations and the functional iterated logarithm law.
\newblock {\em Ann. Appl. Probab.}, {\bf 29}, 1487--1540.

\bibitem[Dreyer {\em et~al.}(2011)Dreyer, Gaber{\v{s}}{\v{c}}ek, Guhlke, Huth,
  \& Jamnik]{dreyer2011phase}
{\sc Dreyer, W., Gaber{\v{s}}{\v{c}}ek, M., Guhlke, C., Huth, R. \& Jamnik, J.}
  (2011)
\newblock Phase transition in a rechargeable lithium battery.
\newblock {\em European Journal of Applied Mathematics\/}, {\bf 22}, 267--290.

\bibitem[{Fang} \& {Giles}(2020){Fang} \& {Giles}]{FangGiles2016}
{\sc {Fang}, W. \& {Giles}, M.~B.} (2020)
\newblock {Adaptive Euler-Maruyama method for SDEs with nonglobally Lipschitz
  drift}.
\newblock {\em {Ann. Appl. Probab.}}, {\bf 30}, 526--560.

\bibitem[Frikha \& Menozzi(2012)Frikha \& Menozzi]{Frikha2012concentration}
{\sc Frikha, N. \& Menozzi, S.} (2012)
\newblock Erratum: {C}oncentration bounds for stochastic approximations
  [mr2988393].
\newblock {\em Electron. Commun. Probab.}, {\bf 17}, no. 60, 2.

\bibitem[Gobet \& Pagliarani(2018)Gobet \& Pagliarani]{GobetPagliarani2018}
{\sc Gobet, E. \& Pagliarani, S.} (2018)
\newblock Analytical approximations of non-linear {SDE}s of {M}c{K}ean-{V}lasov
  type.
\newblock {\em Journal of Mathematical Analysis and Applications\/}.

\bibitem[Gomes {\em et~al.}(2020)Gomes, Pavliotis, \& Vaes]{gomes2019mean}
{\sc Gomes, S., Pavliotis, G. \& Vaes, U.} (2020)
\newblock Mean field limits for interacting diffusions with colored noise:
  phase transitions and spectral numerical methods.
\newblock {\em Multiscale Modeling \& Simulation\/}, {\bf 18}, 1343--1370.

\bibitem[Guhlke {\em et~al.}(2018)Guhlke, Gajewski, Maurelli, Friz, \&
  Dreyer]{DreyerFrizGajewskiEtAl2016}
{\sc Guhlke, C., Gajewski, P., Maurelli, M., Friz, P.~K. \& Dreyer, W.} (2018)
\newblock Stochastic many-particle model for {LFP} electrodes.
\newblock {\em Contin. Mech. Thermodyn.}, {\bf 30}, 593--628.

\bibitem[Higham {\em et~al.}(2002)Higham, Mao, \& Stuart]{HighamEtAl2002}
{\sc Higham, D.~J., Mao, X. \& Stuart, A.~M.} (2002)
\newblock Strong convergence of {E}uler-type methods for nonlinear stochastic
  differential equations.
\newblock {\em SIAM Journal on Numerical Analysis\/}, {\bf 40}, 1041--1063.

\bibitem[Hutzenthaler {\em et~al.}(2011)Hutzenthaler, Jentzen, \&
  Kloeden]{HutzenthalerEtAl2011}
{\sc Hutzenthaler, M., Jentzen, A. \& Kloeden, P.~E.} (2011)
\newblock Strong and weak divergence in finite time of {E}uler's method for
  stochastic differential equations with non-globally {L}ipschitz continuous
  coefficients.
\newblock {\em Proc. R. Soc. Lond. Ser. A Math. Phys. Eng. Sci.}, {\bf 467},
  1563--1576.

\bibitem[Hutzenthaler {\em et~al.}(2012)Hutzenthaler, Jentzen, \&
  Kloeden]{HutzenthalerEtAl2012}
{\sc Hutzenthaler, M., Jentzen, A. \& Kloeden, P.~E.} (2012)
\newblock Strong convergence of an explicit numerical method for {SDE}s with
  nonglobally {L}ipschitz continuous coefficients.
\newblock {\em The Annals of Applied Probability\/}, {\bf 22}, 1611--1641.

\bibitem[Jourdain \& Tse(2020)Jourdain \& Tse]{jourdain2020central}
{\sc Jourdain, B. \& Tse, A.} (2020)
\newblock Central limit theorem over non-linear functionals of empirical
  measures with applications to the mean-field fluctuation of interacting
  particle systems.
\newblock {\em arXiv preprint arXiv:2002.01458\/}.

\bibitem[Keener(2010)Keener]{Keener2011}
{\sc Keener, R.~W.} (2010)
\newblock {\em Theoretical statistics\/}.
\newblock Springer Texts in Statistics.
\newblock Springer, New York, pp. xviii+538.
\newblock Topics for a core course.

\bibitem[Kohatsu-Higa \& Ogawa(1997)Kohatsu-Higa \& Ogawa]{KohatsuOgawa1997}
{\sc Kohatsu-Higa, A. \& Ogawa, S.} (1997)
\newblock Weak rate of convergence for an {E}uler scheme of nonlinear {SDE}'s.
\newblock {\em Monte Carlo Methods and Applications\/}, {\bf 3}, 327--345.

\bibitem[Lacker(2018)Lacker]{Lacker2018}
{\sc Lacker, D.} (2018)
\newblock On a strong form of propagation of chaos for {M}c{K}ean-{V}lasov
  equations.
\newblock {\em Electronic Communications in Probability\/}, {\bf 23}.

\bibitem[Lionnet {\em et~al.}(2015)Lionnet, dos Reis, \&
  Szpruch]{LionnetReisSzpruch2015}
{\sc Lionnet, A., dos Reis, G. \& Szpruch, L.} (2015)
\newblock Time discretization of {FBSDE} with polynomial growth drivers and
  reaction-diffusion {PDE}s.
\newblock {\em Ann. Appl. Probab.}, {\bf 25}, 2563--2625.

\bibitem[Lionnet {\em et~al.}(2018)Lionnet, dos Reis, \& Szpruch]{Lionnet2018}
{\sc Lionnet, A., dos Reis, G. \& Szpruch, L.} (2018)
\newblock Convergence and qualitative properties of modified explicit schemes
  for {BSDE}s with polynomial growth.
\newblock {\em Ann. Appl. Probab.}, {\bf 28}, 2544--2591.

\bibitem[Malrieu(2003)Malrieu]{Malrieu2003}
{\sc Malrieu, F.} (2003)
\newblock Convergence to equilibrium for granular media equations and their
  {E}uler schemes.
\newblock {\em Ann. Appl. Probab.}, {\bf 13}, 540--560.

\bibitem[Malrieu \& Talay(2006)Malrieu \& Talay]{MalrieuTalay2006}
{\sc Malrieu, F. \& Talay, D.} (2006)
\newblock Concentration inequalities for {E}uler schemes.
\newblock {\em Monte {C}arlo and quasi-{M}onte {C}arlo methods 2004\/}.
\newblock Springer, Berlin, pp. 355--371.

\bibitem[Mao \& Szpruch(2013)Mao \& Szpruch]{MaoSzpruch2013}
{\sc Mao, X. \& Szpruch, L.} (2013)
\newblock Strong convergence and stability of implicit numerical methods for
  stochastic differential equations with non-globally {L}ipschitz continuous
  coefficients.
\newblock {\em Journal of Computational and Applied Mathematics\/}, {\bf 238},
  14--28.

\bibitem[Mehri {\em et~al.}(2020)Mehri, Scheutzow, Stannat, \&
  Zangeneh]{Mehri2020}
{\sc Mehri, S., Scheutzow, M., Stannat, W. \& Zangeneh, B.~Z.} (2020)
\newblock Propagation of chaos for stochastic spatially structured neuronal
  networks with delay driven by jump diffusions.
\newblock {\em Ann. Appl. Probab.}, {\bf 30}, 175--207.

\bibitem[M\'{e}l\'{e}ard(1996)M\'{e}l\'{e}ard]{Meleard1996}
{\sc M\'{e}l\'{e}ard, S.} (1996)
\newblock Asymptotic behaviour of some interacting particle systems;
  {M}c{K}ean-{V}lasov and {B}oltzmann models.
\newblock {\em Probabilistic models for nonlinear partial differential
  equations ({M}ontecatini {T}erme, 1995)\/}. Lecture Notes in Math., vol.
  1627.
\newblock Springer, Berlin, pp. 42--95.

\bibitem[Milstein \& Tretyakov(2005)Milstein \&
  Tretyakov]{MilsteinTretyakov2005}
{\sc Milstein, G. \& Tretyakov, M.~V.} (2005)
\newblock Numerical integration of stochastic differential equations with
  nonglobally {L}ipschitz coefficients.
\newblock {\em SIAM journal on numerical analysis\/}, {\bf 43}, 1139--1154.

\bibitem[Mitrinovic {\em et~al.}(2012)Mitrinovic, Pecaric, \&
  Fink]{MitrinovicEtAl2012}
{\sc Mitrinovic, D.~S., Pecaric, J. \& Fink, A.~M.} (2012)
\newblock {\em Inequalities involving functions and their integrals and
  derivatives\/},  vol.~53.
\newblock Springer Science \& Business Media.

\bibitem[Protter(2005)Protter]{Protter2005}
{\sc Protter, P.~E.} (2005)
\newblock {\em Stochastic integration and differential equations\/}. Stochastic
  Modelling and Applied Probability,  vol.~21.
\newblock Springer-Verlag, Berlin, pp. xiv+419.
\newblock Second edition. Version 2.1, Corrected third printing.

\bibitem[Rachev \& R\"{u}schendorf(1998)Rachev \&
  R\"{u}schendorf]{RachevRueschendorf1998-Vol2}
{\sc Rachev, S.~T. \& R\"{u}schendorf, L.} (1998)
\newblock {\em Mass transportation problems. {V}ol. {II}\/}.
\newblock Probability and its Applications (New York).
\newblock Springer-Verlag, New York, pp. xxvi+430.
\newblock Applications.

\bibitem[Sabanis(2013)Sabanis]{Sabanis2013}
{\sc Sabanis, S.} (2013)
\newblock A note on tamed {E}uler approximations.
\newblock {\em Electron. Commun. Probab.}, {\bf 18}, no. 47, 10.

\bibitem[Sznitman(1991)Sznitman]{Sznitman1991}
{\sc Sznitman, A.-S.} (1991)
\newblock Topics in propagation of chaos.
\newblock {\em Ecole d'Et{\'e} de Probabilit{\'e}s de Saint-Flour XIX ---
  1989\/}, 165--251.

\bibitem[Tien(2013)Tien]{Tien2013}
{\sc Tien, D.~N.} (2013)
\newblock A stochastic {G}inzburg-{L}andau equation with impulsive effects.
\newblock {\em Physica A: Statistical Mechanics and its Applications\/}, {\bf
  392}, 1962--1971.

\bibitem[Villani(2009)Villani]{Villani2008}
{\sc Villani, C.} (2009)
\newblock {\em Optimal transport\/}. Grundlehren der Mathematischen
  Wissenschaften [Fundamental Principles of Mathematical Sciences],  vol. 338.
\newblock Springer-Verlag, Berlin, pp. xxii+973.
\newblock Old and new.

\bibitem[Zeidler(1990)Zeidler]{Zeidler1990-IIB}
{\sc Zeidler, E.} (1990)
\newblock {\em Nonlinear functional analysis and its applications. {II}/{B}\/}.
\newblock Springer-Verlag, New York, pp. i--xvi and 469--1202.
\newblock Nonlinear monotone operators, Translated from the German by the
  author and Leo F. Boron.

\end{thebibliography}

\end{document}